\documentclass[11pt]{article}

\usepackage[utf8]{inputenc}
\usepackage{amsmath}
\usepackage{amssymb,amsfonts,amsthm}
\usepackage{mathrsfs}
\usepackage{mathtools}
\usepackage{algorithm}
\usepackage{algorithmic}
\usepackage{xcolor}
\usepackage{tikz}
\usepackage{fullpage}
\usepackage{dsfont}
\usepackage{thmtools}
\usepackage{thm-restate}
\usepackage{enumitem}
\usepackage{dsfont}
\usepackage{graphicx}
\usepackage{sidecap}

\usepackage{accents}

\usepackage{hyperref}
\usepackage{xcolor}
\hypersetup{
  colorlinks   = true, 
  urlcolor     = blue, 
  linkcolor    = blue, 
  citecolor   = blue 
}
\usepackage[capitalize, nameinlink]{cleveref}

\usepackage{tikz}
\usetikzlibrary{decorations.pathreplacing}
\usetikzlibrary{arrows,positioning}
\DeclareRobustCommand\shape{
 \lower5pt\hbox{
 \hskip-7pt
  \tikzset{circ/.style={circle, draw, fill=black, scale=.15}}
  \begin{tikzpicture}[semithick,scale=.3]
  \node (l1) at (0,.5) [circ]{};
  \node (l3) at (0.5,0.3) [circ]{};
  \draw[-] (l1) to node [auto] {} (l3);
    \end{tikzpicture}
  \hskip-8pt}
}

\newcommand{\bern}{\mathsf{Bern}}
\newcommand{\unif}{\mathsf{Unif}}

\newcommand{\ind}{\mathds{1}}

\newtheorem{theorem}{Theorem}[section]

\newtheorem{lemma}{Lemma}[section]

\newtheorem{corollary}[lemma]{Corollary}
\newtheorem{observation}[lemma]{Observation}
\newtheorem{claim}[lemma]{Claim}
\newtheorem{fact}[lemma]{Fact}
\newtheorem{proposition}[theorem]{Proposition}
\theoremstyle{definition}
\newtheorem{remark}{Remark}[]
\newtheorem{definition}{Definition}

\newcommand{\R}{\mathsf{R}}
\newcommand{\unitD}{\Delta}
\newcommand{\SC}{\mathsf{SC}}
\newcommand{\Star}{\mathsf{Star}}
\newcommand{\RL}{\mathsf{R}_\mathsf{L}}
\newcommand{\RU}{\mathsf{R}_\mathsf{U}}
\newcommand{\FL}{\mathsf{F}_\mathsf{L}}
\newcommand{\FU}{\mathsf{F}_\mathsf{U}}
\newcommand{\signedweight}{\mathsf{SW}}
\newcommand{\TP}{\mathsf{TP}}
\newcommand{\NLP}{\mathsf{NLP}}
\newcommand{\bfA}{\mathbf{A}}
\newcommand{\one}{\mathbf{1}}
\newcommand{\trace}{\mathsf{tr}}

\newcommand{\graphs}{\mathcal{G}}
\newcommand{\graphsnoleaves}{\mathcal{NLG}}

\newcommand{\RGGsphere}{{\mathsf{RGG}(n,\dsphere, p)}}
\newcommand{\RGGspherehalf}{{\mathsf{RGG}(n,\dsphere, \frac{1}{2})}}
\newcommand{\RGGgausshalf}{{\mathsf{RGG}(n,\mathcal{N}(0, \frac{1}{d}I_d), \frac{1}{2})}}
\newcommand{\RGGgauss}{{\mathsf{RGG}(n,\mathcal{N}(0, \frac{1}{d}I_d), p)}}
\newcommand{\standardgauss}{\mathcal{N}(0,1)}
\newcommand{\standardgaussd}{\mathcal{N}(0,\frac{1}{d}I_d)}

\newcommand{\CC}{\mathsf{CC}}

\newcommand{\var}{\mathsf{Var}}
\newcommand{\cov}{\mathsf{Cov}}
\DeclareMathOperator*{\E}{{\rm I}\kern-0.18em{\rm E}}

\newcommand{\prob}{\Pr}
\newcommand{\expect}{\E}
\renewcommand{\Pr}{\,{\rm I}\kern-0.18em{\rm P}}

\newcommand{\bfG}{\mathbf{G}}
\newcommand{\bfH}{\mathbf{H}}

\newcommand{\bfU}{\mathbf{U}}

\newcommand{\RS}{\mathcal{RSC}}
\newcommand{\RGG}{\mathsf{RGG}}

\newcommand{\RSC}{\mathcal{RSC}}

\newcommand{\cF}{\mathcal{F}}

\newcommand{\OEI}{\mathsf{OEI}}

\newcommand{\SOEI}{\mathsf{SOEI}}

\newcommand{\CS}{\mathsf{CS}}
\newcommand{\N}{\mathcal{N}}
\newcommand{\SW}{\mathsf{SW}}

\newcommand{\ER}{{Erd\H{o}s-R\'enyi}}
\newcommand{\ERspace}{{Erd\H{o}s-R\'enyi }}

\newcommand{\polylog}{\mathsf{polylog}}
\newcommand{\Deg}{\mathsf{deg}}
\newcommand{\dsphere}{\mathbb{S}^{d-1}}

\newcommand{\onetorus}{\mathbb{S}^{1}}

\newcommand{\ergraph}{\mathsf{G}(n,p)}
\newcommand{\ergraphhalf}{\mathsf{G}(n,1/2)}

\newcommand{\quadand}{\quad\text{and}\quad}

\newcommand{\PCol}{\mathsf{PCol}}
\newcommand{\bfX}{\mathbf{X}}
\newcommand{\Label}{\mathsf{L}}

\newcommand{\indicator}{\mathds{1}}

\renewcommand{\eth}{\overline{\partial}}

\usepackage[titles]{tocloft}
\author{Kiril Bangachev\thanks{Dept. of EECS, MIT. \texttt{kirilb@mit.edu}. Supported by a Siebel Scholarship.} \quad Guy Bresler\thanks{Dept. of EECS, MIT. \texttt{guy@mit.edu}. Supported by NSF Career Award CCF-1940205.}}

\title{On The Fourier Coefficients of\\ High-Dimensional Random Geometric Graphs}

\begin{document}

\maketitle

\pagenumbering{gobble}

\begin{abstract} The random geometric graph $\RGGsphere$ is formed by sampling $n$ i.i.d. vectors $\{V_i\}_{i = 1}^n$ uniformly on $\dsphere$ and placing an edge between pairs of vertices $i$ and $j$ for which $\langle V_i,V_j\rangle \ge \tau^p_d,$ where $\tau^p_d$ is such that the expected density is $p.$ We study the low-degree Fourier coefficients of the distribution $\RGGsphere$ and its Gaussian analogue. 

Our main conceptual contribution is a novel two-step strategy for bounding Fourier coefficients which we believe is more widely applicable to studying latent space distributions.
First, we \emph{localize the dependence among edges} to few \emph{fragile edges}. Second, we partition the space of latent vector configurations $(\dsphere)^{\otimes n}$ based on the set of {fragile edges} and on each subset of configurations, we define a \emph{noise operator acting independently} on edges not incident (in an appropriate sense) to fragile edges. 

We apply the resulting bounds to: 1) Settle the low-degree polynomial complexity of distinguishing spherical and Gaussian random geometric graphs from \ERspace both in the case of observing a complete set of edges and in the non-adaptively chosen mask $\mathcal{M}$ model recently introduced by \cite{mardia2023finegrainedquery}; 2) Exhibit a statistical-computational gap for distinguishing $\RGG$ and the planted coloring model \cite{kothari2023planted} in a regime when 
$\RGG$ is distinguishable from \ER;
3) Reprove known bounds on the second eigenvalue of random geometric graphs. 
\end{abstract}

\newpage
\setcounter{tocdepth}{2}
\tableofcontents

\clearpage
\pagenumbering{arabic}

\section{Introduction}

Random graphs with a latent high-dimensional geometric structure are increasingly relevant in an era of massive networks over complex computer, social, or biological populations.
Such graphs provide a fruitful, even if idealized, model in which to study algorithmic and statistical questions. For these reasons, in the last 15 years random geometric graphs have seen a surge of attention in the combinatorics, statistics, and computer science communities. 
Tasks addressed in the literature include: 1) \emph{Detecting} the presence of a latent geometric structure \cite{Devroye11,Bubeck14RGG,brennan2019phase,Mikulincer20,Liu21PhaseTransition,Brennan22AnisotropicRGG,Liu2022STOC,Liu2021APV,bangachev2023random,bangachev2023detection}, 2) \emph{Estimating} the {dimension} of the latent geometry
\cite{Bubeck14RGG,friedrich2023dimension,bangachev2023detection}, 3) \emph{Embedding} the graph in a geometric space and \emph{clustering} \cite{Ma2020UniversalLS,OConner20,li2023spectral}, 4) \emph{Matching} unlabelled noisy copies of the same geometric graph \cite{liu2024random,wang22matching}. In a different direction of study, high-dimensional random geometric graphs exhibit an intricate and useful combinatorial structure. Most notably, in \cite{Liu22Expander}, the authors show that in certain regimes spherical random geometric graphs are efficient \emph{2-dimensional expanders}, objects for which no other simple randomized constructions are known as of now.

Two  of the most common models, studied since the early works \cite{Devroye11,Bubeck14RGG}, are spherical and Gaussian (hard thresholds) random geometric graphs. 

\begin{definition}[Spherical and Gaussian Random Geometric Graphs]
\label{def:spherical}
The spherical random geometric graph distribution $\RGGsphere$ on $n$ vertices $[n] = \{1,2,\ldots, n\}$ of dimension $d$ with expected density $p$ is defined as follows. First, $n$ independent vectors $V_1, V_2, \ldots, V_n$ are drawn iid from the uniform distribution on the sphere $\dsphere.$ Then, an edge between $i$ and $j$ is formed if and only if $\langle V_i ,V_j\rangle \ge \tau^p_d.$ Here, $\tau^p_d$ is chosen so that $p = \prob[\langle V_i ,V_j\rangle \ge \tau^p_d].$ 

Similarly, in the Gaussian case $\RGGgauss,$ one samples $Z_1, Z_2,\ldots, Z_n\sim\standardgaussd$ and forms an edge $(ji)$ whenever $\langle Z_i,Z_j\rangle\ge \rho^p_d$ with $\rho^p_d$ chosen such that the expected density is $p.$  
\end{definition}

The main goal of the current paper is to analyse the low-degree Fourier coefficients of the probability mass functions of those two distributions. The Fourier coefficients of an $n$-vertex random graph distribution  $\mathsf{R}$ are parametrized by edge-subgraphs $H.$ The $p$-biased Fourier coefficient corresponding to $H$ is defined by 
\begin{equation}
\label{eq:definefourier}
    \Phi_\mathsf{R}^p(H):=(p(1-p))^{-|E(H)|/2}\times \expect_{\bfG\sim \mathsf{R}}\Big[\prod_{(ji)\in E(H)}(\bfG_{ji}- p)\Big]= 
    (p(1-p))^{-|E(H)|/2}\times \expect_{\bfG\sim \mathsf{R}}\Big[\signedweight^p_{H}(\bfG)\Big].
\end{equation}
Here $\signedweight^p_{H}(\bfG)$ is the signed weight of $H$ defined by the above equation, and thus the Fourier coefficients are (signed) expectations of subgraphs.

Low-degree Fourier coefficients of distributions (and, more generally, Boolean functions) are at the core of many milestone results in theoretical computer science and combinatorics such as  constructing succinct nearly $k$-wise independent distributions \cite{alon1992kwise}, learning various classes of Boolean functions \cite{bshouty1996monotonespectrum,Eskenazis_21}, the Margulis-Russo formula on sharp-thresholds \cite{margulis1974,Russo1981} and many more
(see \cite{ODonellBoolean}). More recently, low-degree Fourier coefficients have become central to the design of efficient algorithms for problems in high-dimensional statistics, as well as providing evidence for computational hardness, 
via the
\emph{low-degree polynomial} framework \cite{hopkins2017bayesian,hopkins18}.


Unfortunately, estimating Fourier coefficients is a highly non-trivial task for complex distributions with dependencies among variables. 
In this work, we introduce a conceptually novel approach (described shortly in \cref{sec:mainideas}) for bounding the Fourier coefficients of distributions with random latent structure and use it for $\RGGsphere$ and $\RGGgauss.$ This unlocks the powerful methods mentioned above which leads to several applications, described next.

\paragraph{1. Testing.} Testing against \ERspace
is one of the most natural and well-studied questions on high-dimensional random geometric graphs, starting with \cite{Devroye11}. 
Testing is a prerequisite for more sophisticated tasks: if one cannot even distinguish a graph from pure noise, one can hardly hope to do any other meaningful inference about its structure.

In the spherical case, one observes a graph $\bfG$ and the goal is to test between the two hypotheses 
    $$H_0: \bfG\sim \ergraph \quadand H_1:\bfG\sim\RGGsphere\,.$$ 
The state-of-the-art results for $p \leq 1/2$ and $p = \Theta(1/n)$ are as follows. By counting signed triangles, one succeeds with high probability whenever $d\le (np)^3(\log 1/p)^C$ for some constant $C$ \cite{Bubeck14RGG,Liu2022STOC}. Counting signed triangles is conjectured to be \textit{information theoretically optimal}, i.e., for $d \gg (np)^3(\log 1/p)^{C}$ it is believed to be impossible to test between the two graph distributions~\cite{brennan2019phase,Liu2022STOC,bangachev2023random}. 
The best bounds on when $\RGGsphere$ and $\ergraph$ are indistinguishable, due to \cite{Liu2022STOC}, are: 1) $d\ge n^3p^2(\log 1/p)^C$ for all $p \leq 1/2$;
2)~$d \ge (np)^3(\log 1/p)^C = \polylog(n)$ for $p = \Theta(1/n).$ In particular, 
the threshold in dimension $d$ at which testing becomes possible is only known (up to lower order terms)
when $p =\tilde{\Theta}(1)$ or $ p = \Theta(1/n).$ 

We make progress in the intermediate regime $1/n\ll p\ll 1/2$ by showing that the signed triangle statistic is \emph{computationally optimal with respect to low-degree polynomial tests} at all densities, even in a stronger non-adaptive edge query model recently introduced by \cite{mardia2023finegrainedquery}. Surprisingly, we show that this is not the case for Gaussian random geometric graphs. For small $p,$ low-degree tests other than the signed-triangle statistics are much more powerful: when $p =\Theta(1/n),$ one can distinguish $\RGGgauss$ and $\ergraph$ for dimensions as large as $\sqrt{n}(\log n)^{C'},$ in sharp contrast to the $d= \polylog(n)$ threshold in the spherical case \cite{Liu2022STOC}.

\begin{figure}[!htb]
    \centering
    \includegraphics[width = .9\linewidth]{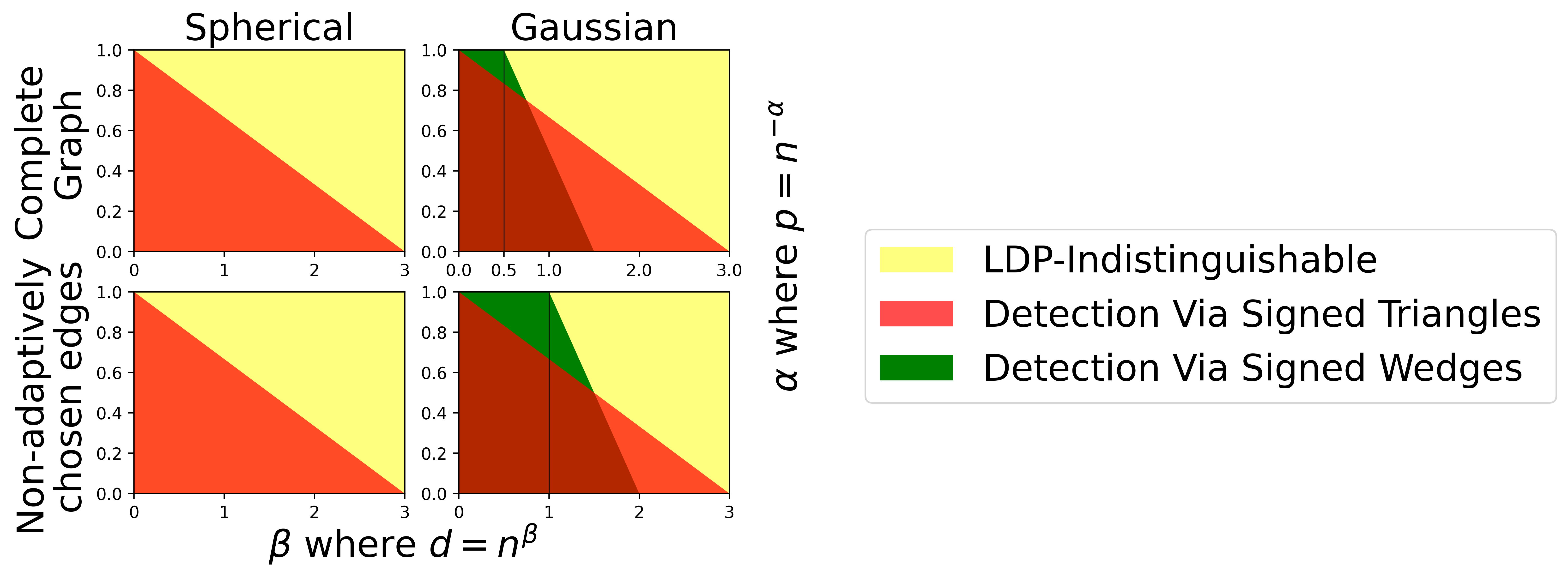}
    \caption{Detecting $d$-dimensional geometry via low-degree polynomials. In the model of non-adaptively queried edges $\mathcal{M},$  $n := \sqrt{|E(\mathcal{M})|}.$ A wedge is a path on 3 vertices.}
    \label{fig:gaussandsphere}
\end{figure}

We additionally prove low-degree indistinguishability between $\RGGspherehalf$ and a planted coloring model \cite{kothari2023planted} in a regime when both are distinguishable from $\ergraphhalf$ via simple low-degree tests. 
The two models can be easily distinguished from one another by determining the largest clique, a computationally \emph{inefficient} test, which shows a computation-information gap for this testing problem.
To the best of our knowledge, this is the first negative result on testing between $\RGGspherehalf$ and a non-geometric distribution when 
$d\ll n^3.$


\paragraph{2. Spectral properties.} The second eigenvalue $\lambda_2$ of $\bfG\sim\RGG$ is captured by low-degree polynomials via the trace method (see \cref{section:spectral}). $\lambda_2$ naturally plays an important role in the expansion properties of $\RGGsphere$
\cite{Liu22Expander}. The top eigenvalues 
are also used in embedding and clustering random geometric graphs via the top eigenvectors \cite{li2023spectral}. These works have characterized the behavior of $\lambda_2$: when $d\ll np$, $\lambda_2 = \tilde{\Theta}(np/\sqrt{d})$ and when $d\gg np,$ the behaviour is similar to \ERspace and $\lambda_2 = \tilde{\Theta}(\sqrt{np}).$\footnote{To be fully precise, \cite{Liu22Expander} considers the normalized adjacency matrix and \cite{li2023spectral} considers a Gaussian rather than spherical random geometric graph} We reprove this bound in the case $p= 1/2$ using our estimates on the Fourier coefficients.
While our approach yields the same quantitative bounds, its methodology is rather different and much more combinatorial.


\subsection{Organization of Paper}
\label{sec:organization}
Our main contribution is a new methodology for deriving strong bounds on
Fourier coefficients which we use to argue about the random geometric graph distributions. In \cref{sec:challenges} we describe the challenges in bounding low-degree Fourier coefficients followed by the main ideas used to overcome them in \cref{sec:mainideas}. Our main theorem followed by applications to testing and the second eigenvalue are stated in \cref{sec:results}. In \cref{sec:mainclaim} we give the full proof of our main theorem.
The different applications follow by variations of what are by now well-known techniques. We present the main ideas in \cref{section:sphericaltoER,section:gaussiantoER,section:RGGtoPC,section:spectral} and delay simple calculations to appendices. For testing, we use the $\chi^2$ low-degree advantage formula (when testing against the planted coloring model, we need a more subtle version of it from \cite{kothari2023planted}). For the second eigenvalue, we use the trace method.

\subsection{Challenges in Bounding Low-Degree Fourier Coefficients}
\label{sec:challenges}


The importance of Fourier coefficients of graph distributions has motivated a series of previous works on probabilistic latent space (hyper)graphs. Existing methods for computing Fourier coefficients, however, seem not fully adequate towards our goal.

\paragraph{Approach 0: Direct Integration.} The most naive approach to estimating Fourier coefficients is a direct integration (summation) over the latent space. Recalling \cref{eq:definefourier}, one can compute the Fourier coefficient of $H$ by integrating $(p(1-p))^{-|E(H)|/2}\times \prod_{(ji)\in E(H)}(\indicator[\langle Z_i,Z_j\rangle \ge \rho^p_d]- p)$ against  $\standardgaussd^{\otimes V(H)}.$ Such a calculation, however, seems out of reach due to the complex dependencies between different terms in the product. As latent vectors $Z_i, Z_j$ vary smoothly, so does the distance between $Z_i$ and $Z_j$ and consequently also the probabilities of various events (such as $Z_k$ being a common neighbour). As concrete evidence of the difficulty of this approach, 
even in the simplest case of triangles for $\RGGsphere,$ the authors of \cite{Bubeck14RGG} spend 5 pages of calculations.  
For a similar random geometric graph model over the torus with $L_p$ geometry, the calculation for triangles is still open when $1\ll p< \infty$ \cite{bangachev2023detection}.

\paragraph{Approach 1: Vertex Conditioning.} 
Many Fourier computations are for problems defined by planting small dense communities in an ambient \ERspace graph \cite{hopkins2017bayesian,rush2022,dhawan2023detection,kothari2023planted,mardia2023finegrainedquery}. In such works, one can use the following simple \emph{vertex conditioning} strategy (exploiting the ambient \ERspace structure) to overcome the technical difficulty of a direct summation/integration.  
As a prototypical example, discussed in \cite{hopkins18}, consider the planted $k$-clique distribution where each vertex $i\in [n]$ independently receives a label $x_i,$ where $\prob[x_i = 1] = k/n, \prob[x_i = 0]= 1-k/n$.
Conditioned on the labels, each edge $G_{ji}$ appears with probability $1$ if $x_i = x_j = 1$ and independently with probability $1/2$ otherwise. Now, consider the Fourier coefficient $\E [\prod_{(ji)\in E(H)}(2G_{ji} - 1)]$ indexed by a graph $H$ without isolated vertices.
Again, there are complex correlations between different edges. However,
unless all vertices of $H$ have label $1,$ there 
is a random (probability $1/2$) edge and this zeros out the Fourier coefficient $\E[\prod_{(ji)\in E(H)}(2G_{ji} - 1)]$. By conditioning on all vertex labels being 1, one shows that the Fourier coefficient equals $(k/n)^{|V(H)|}.$ An approach based on \emph{vertex conditioning} seems to not be applicable to hard threshold random geometric graphs: In $\RGGsphere$, conditioned on the latent vectors there is no randomness left in $\bfG\sim\RGGsphere.$ Hence, one cannot exploit cancellations due to left-over randomness in edges once labels are known, which is crucial in models with ambient \ER.

\paragraph{Approach 2: Lifting From a Single Dimension.}
The work \cite{bangachev2023detection} computes the low-degree Fourier coefficients of (hard threshold) random geometric graphs over the $d$-dimensional torus $\mathbb{R}^d/\mathbb{Z}^d$ with the $L_\infty$ metric. The insight in \cite{bangachev2023detection} is that an $L_\infty$ random geometric graph is the $\text{AND}$ of $d$ 1-dimensional random geometric graphs over $\onetorus.$ They combine the contributions of different coordinates via an analytical approach mimicking the cluster-expansion formula from statistical physics. As explained in \cite{bangachev2023detection}, in our case of $\RGGgauss$ the edges are closer to $\mathsf{MAJORITY}$ over the coordinates. Unfortunately, extending the techniques for the simpler $\mathsf{AND}$ combination to the present setting seems to be technically challenging (in particular, because the Fourier expansion of $\mathsf{AND}$ is much simpler than that of $\mathsf{MAJORITY}$).

\subsection{Main Ideas}
\label{sec:mainideas}
We focus on the density $1/2$ case $\RGGgausshalf$ for concreteness as it captures most of the main ideas. The argument for other densities is similar, but requires some modification (most notably, it additionally exploits a novel energy-entropy trade-off of the $\RGG$ distribution; see \cref{rmk:energyentropytradeoff}).
Modifying to the sphere can be done via the observation that when $Z\sim\standardgaussd,$ $V = Z/\|Z\|_2\sim \unif(\dsphere),$ the variables $V, \|Z\|_2$ are independent and $\|Z\|_2$ concentrates strongly around 1.
Recall that the threshold value is
$\rho^{1/2}_d = 0.$ Our goal is to estimate $\expect_{\bfG\sim\RGGgausshalf}[\SW^{1/2}_{H}(\bfG)]$, where we assume that $|V(H)| = \polylog(d)$ as this is most relevant to our applications. 

\paragraph{Motivation: A Noise-Operator View.} The following \emph{noise-operator} interpretation (appearing, for example, in \cite{brennan2021statisticalquery}) 
of the calculation for planted clique will turn out to be useful. Consider first the standard noise operator $T_\gamma$ for functions over $\{\pm 1\}^n$ \cite{ODonellBoolean}. It acts on functions $f:\{\pm 1\}^n\to\mathbb{R}$ by $T_\gamma f(x) = \E_{y\sim N_\gamma(x)}[f(y)],$ where $N_\gamma(x)$ is the distribution in which each coordinate $y_i$ independently equals $x_i$ with probability $\gamma$ and otherwise with probability $1-\gamma$ it is re-randomized.
This noise operator contracts Fourier coefficients as $\widehat{T_\gamma f}(S) = \gamma^{|S|}\widehat{f}(S).$ 


\begin{observation}[Noise Operator View of Planted Clique]
\label{obs:plantedclqiuenoiseoperator}
The planted clique distribution can itself be viewed as arising from application of a different noise operator. 
Given a function on graphs $f:\{\pm 1\}^{\binom{[n]}{2}}\to \mathbb{R}$, 
let $T_\gamma f(x) = \E_{y\sim N_\gamma(x)}[f(y)]$, where now $N_\gamma(x)$ is the distribution obtained by including each vertex in $A$ with probability $\gamma$ and then rerandomizing all edges in $x$ except those with both endpoints in $A$. This operator again contracts Fourier coefficients, $\widehat{T_\gamma f}(H) = \gamma^{|V(H)|}\widehat{f}(H).$  
If we start with the point mass distribution $\delta_{K_n}$ on the complete graph, then the planted clique probability mass function is obtained by applying $T_{k/n}$.
\end{observation}

Our goal will be to derive such a noise operator perspective for $\RGG$ as well and use it to bound Fourier coefficients. We formally give such a view in \cref{obs:RGGnoiseoperator}, but due to its more complicated nature we gradually build towards it.

\paragraph{1. Strategy: Localizing Edges That Create Dependencies.}
We solve the challenges outlined in the previous subsection with the following high-level idea. 
We will \emph{localize the dependence among edges} to a small set of edges $\mathcal{F}= \cF(Z)$ (depending on the latent vectors). The other edges, in $\cF^c$, will be close to uniformly random. Edges in $\cF$ will in general depend also on edges in $\cF^c$, and we write $\partial \cF$ as the set of edges upon which those in $\cF$ depend.
Letting $\bfU\sim \unif(\{0,1\})^{\otimes {E(H)}}$, 
$$
\bfG = (\bfG_\cF, \bfG_{\cF^c}) \approx (\bfG_\cF, \bfU_{\cF^c}) =  (\bfG_\cF, \bfU_{\partial \cF \cap \cF^c}, \bfU_{\cF^c\setminus \partial \cF })\,.
$$
Note that by definition, edges in $\cF^c\setminus \partial \cF$ are independent of all other edges. Hence, conditioning on $\mathcal{F},$ we can re-randomize $\cF^c\backslash \partial \cF$ (i.e., apply the noise operator $T_0$ on $\cF\backslash \partial \cF$). With this idea, we solve both difficulties arising when one attempts the first two approaches outlined before: randomness ensures cancellations and independence makes calculations easy! 

\paragraph{2. Key Idea: An Edge-Independent Basis of Latent Vectors.} Localizing dependence to a small set of edges $\cF$ means that most of the edges are independent. This may seem impossible at first. When we add even a small amount of noise to $Z_i,$ this will likely affect all inner products 
$\{\langle Z_i,Z_j\rangle\}_{j \in V(H)}$ and, hence, 
all edges incident to $i,$
in a complicated, correlated, fashion.

To overcome this issue, we define a convenient basis for the latent vectors. Namely,
for each edge $(ji)\in E(H),$ we construct a random variable $Z_{ji}$ (depending on latent vectors) such that the collection of random variables $\{Z_{ji}\}_{ji\in E(H)}$ is independent and $Z_{ji}$ nearly determines the edge $ji.$ 
We exploit the fact that independent Gaussian vectors in high dimension are nearly orthonormal. This suggests that the Gram-Schmidt operation on the latent vectors will produce an orthonormal basis close to the original vectors and, hence, projections on the Gram-Schmidt basis will approximate the inner products. 
Applying Gram-Schmidt to the $k = |V(H)|$ Gaussian vectors $Z_1,\dots, Z_k$ corresponding to vertices of $H,$  we obtain the Bartlett decomposition:
\begin{equation}\begin{split}\label{eq:Bartlettintro}
    Z_1 &= (Z_{11},0,0,\dots,0) \\
    Z_2 &= (Z_{21}, Z_{22},0,\dots,0)\\
    Z_3 &= (Z_{31}, Z_{32},Z_{33},\dots,0)\\
    &\;\;\vdots\\
    Z_k &= (Z_{k1},  \dots, Z_{k,k-1},Z_{kk},0,\dots,0)\,.
    \end{split}
\end{equation}
Here $Z_{ji} \sim \N(0,\frac 1d)$ for each $i<j$ and 
$dZ_{jj}^2\sim \chi^2(d-j+1), Z_{jj}\ge 0.$ The collection $(Z_{ji})_{i\le j}$ are jointly independent. These properties can be easily derived from the isotropic nature of $\standardgaussd.$ With respect to this decomposition,
\begin{equation}
\label{eq:innerproductequationintro}
\begin{split}
& \langle Z_i, Z_j\rangle = Z_{ji}Z_{ii} + \sum_{\ell< i}Z_{i\ell}Z_{j\ell} = 
Z_{ji} + \Big(
\sum_{\ell< i}Z_{i\ell}Z_{j\ell} + Z_{ji}(Z_{ii} -1)
\Big).
\end{split}
\end{equation}

Now, each term of the form $Z_{i\ell}Z_{j\ell},$ as well as 
$Z_{ji}(Z_{ii} -1),$
is typically on the order of $\tilde{O}(1/d),$ so the entire right-hand side expression above is on the order of 
$\tilde{O}(|V(H)|/d) = \tilde{O}(1/d).$
In contrast, $Z_{ji}\sim\mathcal{N}(0, 1/d),$ so it is typically on the order of $\tilde{O}(1/\sqrt{d})$. 
Therefore,  the random variable $Z_{ji}$ nearly determines whether $ji$ is an edge ($\indicator[\langle Z_i,Z_j\rangle\ge 0]\approx \indicator[Z_{ji}\ge 0]$). This is very promising as the variables $\{Z_{ji}\}$ are also independent, so we can define the noise operator by rerandomizing (a subset of) the variables $Z_{ji}$ independently and, thus, affecting edges $(ji)$ independently.

\paragraph{3. Construction: Fragile Edges Localize Dependencies.} So far, we constructed independent variables $\{Z_{ji}\}_{(ji)\in E(H)}$ which nearly determine the edges. The key word here is nearly -- it may well be the case that $Z_{ji} = \tilde{O}(1/d)$,
in which case (and with high probability, only in this case) 
$Z_{i\ell}Z_{j\ell}$ or $Z_{ji}(Z_{ii}-1)$ could be comparable to or even larger than $Z_{ji}.$ In that case,\linebreak $G_{ji}=\indicator[Z_{ji} + \big(
\sum_{\ell< i}Z_{i\ell}Z_{j\ell} + Z_{ji}(Z_{ii} -1)
\big)\ge 0]$ depends on edges of the form $(i\ell), (j\ell)$ via the variables 
$Z_{i\ell},Z_{j\ell}.$ As edges $(ji)$ for which
$Z_{ji}=\tilde{O}(1/d)$ are the only ones that can depend on other edges they \emph{localize dependence}.  We call edges $(ji)$ for which $Z_{ji}=\tilde{O}(1/d)$ \emph{fragile} and they form the fragile set $\cF.$ 
The rest of the edges are independent, as demonstrated by a noise operator rerandomizing all $Z_{uv}$ for non-fragile $(uv)$ (in a way that $Z_{uv}$ continues to be large enough so that $(uv)$ is not fragile). One can also view $\cF^c$ as an ambient \ERspace for the $\RGG$ distribution.

Recall that $Z_{ji}$ is distributed as 
$\mathcal{N}(0,1/d)$ and, hence, is smaller than $\tilde{O}(1/d)$ only with probability $\tilde{\Theta}(1/\sqrt{d}).$ As variables $Z_{ji}$ are independent, edges are fragile independently. Thus, the probability of observing many fragile edges is very low.

\paragraph{4. Analysis: Combinatorics of Edge Incidences.}
Our construction so far is of a noise operator which acts independently on all non-fragile edges $\cF^c.$ Hence, even if we condition on all fragile edges, there is some randomness left (unless 
all edges are fragile, but this happens with very low probability) and, so, we have solved the issue of destroying all randomness by conditioning outlined in \cref{sec:challenges}. However, it is still difficult to integrate $\prod_{(ji)\in E(H)}(2G_{ji}-1)$ even conditioned on the set of fragile edges. The reason is that if $(ji)$ is fragile, but $(j\ell)$ is not \emph{for some $\ell<i$},
applying the noise operator on ${(j\ell)}$ via $Z_{j\ell}$ may also affect  $\indicator[Z_{ji} + \big(
\sum_{\ell< i}Z_{i\ell}Z_{j\ell} + Z_{ji}(Z_{ii} -1)\big)\ge 0].$

Our approach to this issue is simple -- we define the noise operator only \emph{over edges not incident to fragile edges in their lexicographically larger vertex} (which we formalize in \cref{def:strongcovering} as $\partial \cF$). If there is even a single such edge, the noise operator re-randomizes the edge and zeroes out the Fourier coefficient (as in planted clique).  

This leads us to analyzing the \emph{combinatorics of edge incidences of subgraphs of $H.$} A crucial step in this analysis is the realization that we have the freedom to choose an optimal ordering (with respect to the graph $H$) for the Gram-Schmidt process so that incidences with lexicographically larger fragile edges (i.e. $|\partial \cF|$) are minimized. Optimizing over orderings leads us to a combinatorial quantity associated to the graph $H$  which we call the \textit{ordered edge independence number} $\OEI(H).$ Altogether, our bound on Fourier coefficients becomes $|\expect[\prod_{(ji)\in E(H)}(2G_{ji} - 1)]| \le  (\log ^c d/\sqrt{d})^{\OEI(H)}.$ Our last step is to understand the growth of $\OEI(H).$ We derive several bounds, simplest and most easily interpretable of which is 
$\OEI(H)\ge \lceil (|V(H)| - 1)/2\rceil.$ Perhaps more interesting is $\OEI(H)\ge \delta(H)+ 1,$ where $\delta(H):=\max_{S\subseteq V(H)}|S| - |\{j \in V(H)\; : \; \exists i \in S \text{ s.t. }(ji)\in E(H)\}|.$ 

\medskip

\subsection{Results}
\label{sec:results}
We now formally describe our results, beginning with the exact bounds on Fourier coefficients we obtain. Throughout, we will make the following assumption:
\begin{equation}
\label{eq:assumption}
    \tag{A}
    \text{There exist some absolute constants }\gamma, \epsilon>0 \text{ such that }1/2\ge p \ge n^{-1+\epsilon}, d\ge n^{\gamma}.
\end{equation}
Admittedly, some non-trivial cases are not covered by this assumption. Specifically,
$p  = n^{-1+o(1)}$ and $d = \polylog(n).$
Nevertheless, we note that in the case $p = \Theta(1/n), d = \polylog(n)$ the testing problem between $\RGGsphere$ and $\ergraph$ is fully resolved by \cite{Liu2022STOC} and, thus, \eqref{eq:assumption} captures most of the open regimes at least for the question of testing against \ER.

Throughout, we will refer to low-degree polynomial hardness and subgraph counting algorithms. As these are by now standard in the literature, we defer the definitions to \cref{section:preliminaries}.

\subsubsection{Main Result: The Fourier Coefficients of Gaussian and Spherical RGG}
Fourier coefficients of $\RGG$ factorize over connected components, so we only state our bounds for $H$ connected. We first define the ordered edge independence mentioned in \cref{sec:mainideas}. 

Given an ordering $\pi$ of the vertices (think of $\pi$ as the Gram-Schmidt ordering), we denote an edge between $u$ and $v$ as $(uv)$ if $u>v$ and $(vu)$ otherwise. We formalize $\partial F$ as follows.

\begin{definition}[Covering Property]
\label{def:covering}
An edge $(uv)\in E(H)\backslash F$ is \emph{covered} by $F$ if there exists an edge in $F$ with endpoint $u.$ Denote with $\partial^\pi_HF$ the set of all edges covered by $F.$
\end{definition}

See \cref{fig:examplediagram} for an illustration.
Going back to \cref{eq:innerproductequationintro}, we interpret as follows. If $(uv)$ is covered by fragile edges $F,$ there exists some fragile edge $(uw).$ Hence, $G_{uw}$ might depend on $G_{uv}$ via $Z_{uv}.$ According to \cref{eq:innerproductequationintro}, we can make this definition stronger by requiring $w> v$ additionally. This, however, will not be the case when discussing spherical $\RGG,$ hence we stick to \cref{def:covering}.

\begin{definition}[Ordered Edge Independence Number] 
\label{def:OEI}
For a connected graph $H$ on $k$ vertices and a bijective labelling $\pi$ of the vertices
with the numbers $\{1,2,\ldots, k\},$ 
we say that
a subset of edges $F\subseteq E(H)$ \emph{strongly covers} $H$ if $F \cup \partial^\pi_H F =  E(H)$. 
We
define the \emph{ordered edge independence number of $H$ with respect to $\pi$} and denote by $\OEI_\pi(H)$ as the size of the smallest strongly covering $F.$
Let $\OEI(H) = \max_\pi \OEI_\pi(H).$
\end{definition}

One should think of $E(H)\backslash (F \cup \partial^\pi_H F)$ as the set of edges which the noise operator rerandomizes.

\begin{SCfigure}[10][!htb]
{\caption{\footnotesize Consider the cycle $H = C_5$ with given labelling $\pi$ and $F = \{(53), (42)\}.$ \\
\cref{def:covering}: 
Edge $(51)$ is covered by $(53)$ while $(21),(43)$ are covered by $(42).$ Hence, $\partial^\pi_HF = \{(51),(21), (43)\}.$ Casework shows that $|F| = \OEI_\pi(H) = \OEI(H) = 2.$
\\
\cref{def:strongcovering}: With respect to the strong covering property, $\eth^\pi_HF = \emptyset.$ For example, $(51)$ is not strongly covered because $F^\pi_{\ge 1, \ni 5} = \{(53)\},$ but $3$ is not a neighbour of $1$ in $H.$
However, if we add the edge $(43)$ to $F,$ i.e. $F' = \{(53),(43), (42)\},$ it is the case that both $(51)$ and $(21)$ are strongly covered. Indeed, 
$(F')^\pi_{\ge 1, \ni 5} = (F')^\pi_{\ge 1, \ni 2} = \{(53),(43),(42)\}$ and $1$ is a neighbour to both $2$ and $5.$ Casework shows that $|F'| = \SOEI_\pi(H) = \SOEI(H) = 3.$
}\label{fig:examplediagram}}
{\includegraphics[width=0.22\linewidth]{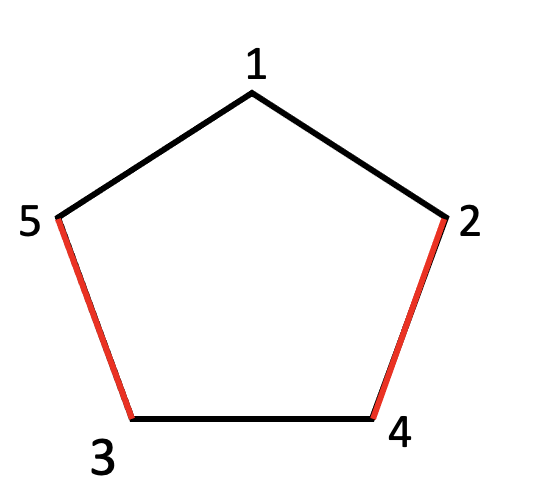}}
\end{SCfigure}

\begin{theorem}
\label{thm:oeiboundonfourier} Suppose that \eqref{eq:assumption} holds and $H$ is connected. Then, there exists some absolute constant $C$ depending only on $\epsilon,\gamma$ in \eqref{eq:assumption} such that for $\bfG\sim \RGGgauss$ and $\bfG\sim\RGGsphere,$
\begin{equation}
    \Big|\E\Big[\prod_{(ji)\in E(H)}(G_{ji} - p)\Big]\Big|\le 
(8p)^{|E(H)|}\times \Big(\frac{C\times |V(H)|\times|E(H)|\times(\log d)^{3/2}}{\sqrt{d}}\Big)^{\OEI(H)}.
\end{equation}
\end{theorem}

For our applications, we only need $|V(H)|,|E(H)| \le (\log n)^{1.1},$ in which case the bound is 
$p^{|E(H)|}(\polylog(n)/\sqrt{d})^{\OEI(H)}.$
Of course, to apply \cref{thm:oeiboundonfourier}, we need explicit bounds on $\OEI(H).$

\begin{proposition}[Bounds on Ordered Edge Independence Number] For a connected graph $H,$
\label{prop:boundsonOEI}
    \begin{enumerate}
    \item $\OEI(H)\ge \lceil(|V(H)| - 1)/2\rceil.$ 
    \item $\OEI(H)\ge \delta(H)+ 1,$ where $\delta(H):=\max_{S\subseteq V(H)}|S| - |\{j \in V(H)\; : \; \exists i \in S \text{ s.t. }(ji)\in E(H)\}|.$
\end{enumerate}
\end{proposition}
The quantity $\delta(H)$ is useful when discussing the non-adaptive edge query model in which one observes any $m$ edges of a random geometric graph. The reason is that a graph on $m$ edges can have at most $\exp(O(|V(H)|^{1.1}))\times m^{(|V(H)|+\delta(H))/2}$ subgraphs isomorphic to $H$, as shown in \cite{Alon1981OnTN}.

\begin{remark}[Towards ``High-Degree'' Hardness] One can show that $\OEI(H)\le |V(H)|.$ Thus, the decay in $\sqrt{d}$ is exponential in $|V(H)|$ rather than $|E(H)|.$ A rather trivial calculation shows that this limits low-degree hardness to polynomials of degree $o(\log^2 n)$ and one cannot hope to prove hardness against polynomials of degree, say $n^{0.001}.$ One potential approach towards hardness against polynomials of higher degree (perhaps even $\binom{n}{2}$ and, hence, information-theoretic hardness) would be conditioning on a subset of latent space configurations, for example as in \cite{dhawan2023detection}.
\end{remark}

\medskip

While the resulting bounds on Fourier coefficients are strong enough for all of our low-degree hardness results, one may still wonder if they are optimal. It turns out that the likely answer is no: In the special case of density $1/2,$ the symmetry of the Gaussian distribution around $0 = \rho^{1/2}_d = \tau^{1/2}_d$ allows us to slightly improve the argument outlined in \cref{sec:mainideas} and define a noise operator that acts also on certain (but not all) edges adjacent to fragile edges (in $\partial_H^{\pi}F$). We describe this next. 
    

\begin{definition}[Strong Covering]
\label{def:strongcovering}
Consider an edge $(uv) \in E(H)\backslash F$ and denote by $F^\pi_{\ge v}$ the subset of $F$ formed by edges with both endpoints at least as large as $v.$ $F^{(\pi)}_{\ge v, \ni u}$ is the connected component of $F^{(\pi)}_{\ge v}$ containing $u.$
We say that $(uv)$ is strongly covered by $F$ if $v$ has a neighbour other than $u$ in $V(F^{(\pi)}_{\ge v, \ni u})$ with respect to $H.$
We denote by $\eth_H^{\,\pi} F$ the set of edges strongly covered by $F.$
\end{definition}

See \cref{fig:examplediagram} for an illustration. The analogue of \cref{def:OEI} is:

\begin{definition}[Strong Ordered Independence Number]
    We define $\SOEI_\pi(H),$ the \emph{strong independence number} of $H$ with respect to $\pi,$ as the minimal cardinality of a set $F$ such that $F\cup \eth_H^{\,\pi} F = E(H).$ $\SOEI(H) = \max_\pi\SOEI_\pi(H).$ 
\end{definition}

The analogue of \cref{thm:oeiboundonfourier} is the following (we state it only for $\RGGspherehalf$ as the Gaussian and spherical models coincide in the 1/2-density case).

\begin{proposition}
    \label{eq:SOEI} Suppose that \eqref{eq:assumption} holds and $H$ is connected. Then, there exists some absolute constant $C$ depending only on $\epsilon,\gamma$ in \eqref{eq:assumption} such that for $\bfG\sim \RGGspherehalf,$
    $$\Big|\E\Big[\prod_{(ji)\in E(H)}(G_{ji} - 1/2)\Big]\Big|\le 
(1/2)^{|E(H)|}\times \Big(\frac{C\times |V(H)|\times|E(H)|\times(\log d)^{3/2})}{\sqrt{d}}\Big)^{\SOEI(H)}.$$
\end{proposition}


If an edge $(uv)$ is strongly covered, then the connected component 
$F^{(\pi)}_{\ge v, \ni u}$ contains a neighbour of $u.$ Hence, there exists a fragile edge with endpoint $u$ and $u$ is also covered according to \cref{def:strongcovering}. Thus, $\SOEI(H)\ge \OEI(H),$ so \cref{eq:SOEI} is at least as strong as \cref{thm:oeiboundonfourier}. It turns out that the inequality is strict for many sparse graphs. For example, one can check that when $H = C_k$ 
is a cycle, $\OEI(C_k) = \lceil (k-1)/2\rceil,$ but $\SOEI(H) = k - 2.$ The latter follows from the following bound:
\begin{proposition}[Strong Edge Independence Number of Sparse Graphs]
\label{prop:sparseSOEI}
Suppose that $H$ is connected. Then, 
$\SOEI(H)\ge 2|V(H)| - |E(H)| - 2.$
\end{proposition}
It turns out that the $\OEI(H)$ bound is too weak for our results on the second eigenvalue of $\RGGspherehalf$ and we need $\SOEI(H)$ and \cref{prop:sparseSOEI}.
We leave open the problem of proving
\cref{eq:SOEI} for all densities $p$.

\subsubsection{Application I: Testing Between Spherical RGG and \ER}
\label{sec:introSphericaltoER}
In the case of spherical random geometric graphs, we not only confirm that the signed triangle statistic is optimal among low-degree polynomial tests, but also show that this is the case even in the non-adaptive edge-query model recently introduced by \cite{mardia2023finegrainedquery}. For a mask $\mathcal{M}\in \{0,1\}^{N\times N}$ and (adjacency) matrix  $A\in \mathbb{R}^{N\times N},$ denote by $A\odot\mathcal{M}$ the $N\times N$ array in which 
$(A\odot \mathcal{M})_{ji}= A_{ji}$ whenever $\mathcal{M}_{ji} = 1$ and $(A\odot \mathcal{M})_{ji}= ?$ whenever $\mathcal{M}_{ji} = 0.$ 

Testing between graph distributions with masks corresponds to a non-adaptive edge query model. Instead of viewing a full graph, one can choose to observe a smaller more structured set of edges in order to obtain a more data-efficient algorithm. This idea was introduced recently in \cite{mardia2023finegrainedquery},
focusing on the planted clique problem. We obtain the following result for $\RGG$.

\begin{theorem}
\label{thm:nonadaptivespherical}
Consider some $M$ where $n\coloneqq\sqrt{M},d,p$ satisfy the assumptions in \eqref{eq:assumption}. Let $\mathcal{M}$ be any graph on $M$ edges without isolated vertices. Denote by $N$ the number of vertices in $\mathcal{M}$.
If $d \ge (M^{1/2}p)^{3+c}$ for any constant $c>0,$ no degree $(\log n)^{1.1}$ polynomial can distinguish with probability $\Omega(1)$ the distributions
$\mathbf{G}_0\odot \mathcal{M}$ and
$\mathbf{G}_1\odot \mathcal{M}$ for 
$\mathbf{G}_0\sim \mathsf{G}(N,p)$ and
$\mathbf{G}_1\sim \mathsf{RGG}(N,\dsphere, p).$
\end{theorem}

We prove \cref{thm:nonadaptivespherical} in \cref{section:sphericaltoER} via a standard $\chi^2$ argument using the bounds on Fourier coefficients in \cref{thm:oeiboundonfourier}.
In the case $\mathcal{M}= K_n,$ we match the conjectured information-theoretic threshold. \cref{thm:nonadaptivespherical} is tight in light of the signed triangle statistic \cite{Liu2022STOC}.

\begin{corollary}
\label{cor:classicspherical}
Consider some $n,d,p$ satisfying assumptions in \eqref{eq:assumption}.
If $d \ge (np)^{3+c}$ for any positive constant $c,$ no degree $(\log n)^{1.1}$ polynomial can distinguish with probability $\Omega(1),$ the distributions 
$\mathsf{G}(n,p)$ and
$\RGGsphere.$
\end{corollary}

\subsubsection{Application II: Testing Between Gaussian RGG and \ER}
\label{sec:introGaussiantoER}
We begin with a brief comparison of the Gaussian and spherical models.

\begin{remark}[Gaussian vs Spherical Random Geometric Graphs] 
\label{rmk:GaussvsSphere}
The Gaussian and spherical models coincide in the case $p = 1/2.$ More generally, they are intimately related due to the facts
\begin{align}
    & \text{If }Z\sim \mathcal{N}(0,\tfrac{1}{d}I_d),\text{ then }V:= Z/\|Z\|_2\sim\unif(\dsphere), \text{ and}\tag{I}
\label{eq:gaussiantospherical}\\[2mm]
    & \text{If }Z\sim \mathcal{N}(0,\tfrac{1}{d}I_d), \text{ then }\|Z\|_2 \approx 1\text{ with high probability}.
    \tag{II}
\label{eq:gaussiantosphericalnormfail}
\end{align}
This correspondence has been used to argue about either model -- in some arguments more helpful is independence of Gaussian coordinates \cite{Devroye11,Bubeck14RGG} while in others orthonormality of the Gegenbauer basis over the sphere \cite{li2023spectral}. We exploit this correspondence in both directions. 

We also show that the two models are qualitatively different in the sparse regime (see \cref{fig:gaussandsphere}). The cause of this difference is the perhaps benign looking fact that \cref{eq:gaussiantosphericalnormfail} is only an approximate statement. 
This creates dependence between edges in the Gaussian case for edges which are independent in the spherical case: For example, under $\bfG\sim\RGGsphere,$ the edges $\bfG_{21}$ and $\bfG_{31}$ are independent. In contrast, under $\bfH\sim\RGGgauss,$ $\bfH_{21},\bfH_{31}$ are positively correlated as both are monotone in $\|Z_1\|_2.$ 
The dependence turns out to be quite strong for small values of $p$ to the point where (signed) wedges are better than signed triangles 
for testing against \ER. 
\end{remark}

\begin{theorem}
\label{thm:unmaskedgaussian}
Consider the task of testing between $\RGGgauss$ and $\ergraph$ under \eqref{eq:assumption}. 
\begin{enumerate}
\item When $d\ge \max\big\{(n^3p^3)^{1+c},(n^{3/2}p)^{1+c}\big\}$ for any constant $c>0,$ no degree $(\log n)^{1.1}$ algorithm can distinguish the two graph models with probability $\Omega(1).$ 
\item When $d\le n^{3/2}p(\log n)^{-5}$ and $p\le 0.49,$ the signed wedge count succeeds w.h.p.
\item When $d\le n^{3}p^3(\log n)^{-5},$ the signed triangle count succeeds w.h.p.
\end{enumerate}
\end{theorem}

In the non-adaptive query complexity model, the difference turns out to be even more dramatic. One can exploit the fact that wedges are highly informative by querying a star-graph, as star graphs maximize the number of wedges for a fixed number of edges.

\begin{theorem}
\label{thm:nonadaptiveGaussian}
Consider some $M$ where $n\coloneqq\sqrt{M},d,p$ satisfy the assumptions in \eqref{eq:assumption}. Let $\mathcal{M}$ be any graph on $M$ edges with no isolated vertices. Let $N$ be the number of vertices in $\mathcal{M}$.
Consider testing between 
$\bfG_0\odot\mathcal{M}$ and $\bfG_1\odot\mathcal{M}$ for
$\bfG_0\sim\mathsf{G}(N,p),
\bfG_1\sim\RGG(N,\standardgaussd, p).$
    \begin{enumerate}
        \item When $d\ge \max\big\{(M^{3/2}p^3)^{1+c},(Mp)^{1+c}\big\} = \max\big\{(n^{3}p^3)^{1+c},(n^2p)^{1+c}\big\}
        $ for any constant $c>0,$ no degree $(\log n)^{1.1}$ algorithm can distinguish the two graph models with probability $\Omega(1).$ 
        \item When $d\le Mp(\log M)^{-5} = \tilde{\Theta}(n^2p)$ and $p\le 0.49,$ the signed wedge count succeeds w.h.p. One possible mask $\mathcal{M}$ is a union of $A = \lceil(\log M)^{17}\rceil$ disjoint stars with $\lfloor M/A\rfloor$ edges each.
        \item When $d\le (M^{3/2}p^3)(\log M)^{-5} = \tilde{\Theta}(n^3p^3),$ the signed triangle count succeeds w.h.p. A possible mask is $\mathcal{M} = K_n.$
    \end{enumerate}
\end{theorem}

The main message of this subsection is that even though the Gaussian and spherical models are closely related and each useful for reasoning about the other, they are also fundamentally different. 

The proofs are similar to the ones in \cref{sec:introSphericaltoER} and are provided in 
\cref{section:gaussiantoER,appendix:gaussiantoERcalc}. We need to take extra care of graphs with leaves (as their Fourier coefficients are non-zero, unlike in the spherical case), which is done in \cref{sec:fourierwithleaveslowerGaussian}.

\begin{remark}
  The work \cite{Brennan21DeFinetti} studies the convergence of masked Wishart matrices to GOE (that is $\bfA_{ji} = \langle Z_i, Z_j\rangle$ instead of $\bfG_{ji} = \indicator[\langle Z_i, Z_j\rangle\ge \rho^p_d].$) Of course, for $p=o(1)$ the $\RGG$ testing problem becomes very different from the Wishart versus GOE problem \cite{brennan2019phase,Liu2022STOC}. 
\end{remark}

\subsubsection{Application III: Testing Between Spherical RGG and Planted Coloring}
\label{sec:introsphericaltoPCol}
In the regime $d \leq (np)^{3-c}$, $\RGG$ is very different from \ER. But is it, perhaps, closely approximated by some other simple model?  
%
%
We show that, with respect to low-degree polynomial tests, $\RGG(n,\dsphere,1/2)$ is indistinguishable from a slight variation of the planted coloring distribution in \cite{kothari2023planted}. We focus on the density 1/2 case, but our arguments can be easily extended (we only use \cref{thm:oeiboundonfourier}, not \cref{eq:SOEI}).

\begin{definition} $\PCol(n,q)$ is the following distribution over $n$ vertex graphs. First, each node $i\in [n]$ independently receives a uniform label $x_i \in [q].$ Then, if $x_i = x_j,$ nodes $i$ and $j$ are adjacent with probability $1.$ If $x_i\neq x_j,$ nodes $i$ and $j$ are adjacent with probability $\frac{1}{2} - \frac{1}{2(q-1)}.$ 
\end{definition}

In comparison, \cite{kothari2023planted} have $i$ and $j$ adjacent with probability $1/2$ when $x_i\neq x_j$. Choosing a value of $q$ so that the signed triangle counts of $\RGGspherehalf$ and $\PCol(n,q)$ (nearly) match, we prove the following fact in \cref{section:RGGtoPC,appendix:rggtopccalc}.

\begin{theorem}
\label{thm:rggtopcthm}
Suppose that $d  \ge n^{8/3+c}$ for any constant $c>0.$ Then, there exists some\linebreak 
$q \in \big[{d^{1/4}}/(\log d),{d^{1/4}}(\log d)\big]$ such that no $(\log n)^{1.1}$-degree  polynomial test distinguishes $\PCol(n,q)$ and $\RGG(n,\dsphere, 1/2)$ with probability $\Omega(1)$. 
\end{theorem}


\begin{remark}
The condition $q = \tilde{\Theta}(d^{1/4})$
establishes a statistical-computational gap when $d\le n^{4-\kappa}$ for any constant $\kappa>0.$
An instance of $\PCol(n,q)$ has a clique of size $n/q = \tilde{\Omega}(nd^{-1/4})$ with probability 1. However, $\RGGspherehalf$ does not contain a clique of size more than 
$3\log_2 n$ with high probability under \eqref{eq:assumption} by \cite{Devroye11}. Perhaps surprisingly, our result holds in the exact same regime as the results of \cite{kothari2023planted} for refuting $q$-colarability. Namely,  $q = \tilde{\Theta}(d^{1/4}),d\ge n^{8/3}$ is equivalent to $q= \tilde{\Omega}(n^{2/3}).$ Our contribution here is not the analysis, but the realisation that $\RGG$ is indistinguishable from $\PCol.$ We prove hardness for detecting $q$-colarability against the natural $\PCol$ model and do not need to construct a more sophisticated ``quiet distribution'' as in \cite{kothari2023planted}.
\end{remark}

\begin{remark}
\cite{chetelat2019middle} studies a similar question for Wishart matrices in the regime $d = o(n^3)$ when Wishart and GOE are distinguishable. The authors obtain a sequence of phase transitions for the Wishart density. The approximating densities are defined in terms of an inverse Fourier transform and are not easily interpretable, in contrast to the simple $\PCol$ distribution. 
\end{remark}

\subsubsection{Application IV: The Second Eigenvalue of Spherical RGG}
\begin{theorem}
\label{thm:secondeigenvalue}
Suppose that $\bfG\sim \RGG(n,\dsphere, 1/2)$ (equivalently $\bfG\sim\RGGgausshalf$).  
\begin{itemize}
    \item If $d \le n (\log n)^8,$ then with high probability
    $|\lambda_2(\bfG)|\le n(\log n)^{10}/\sqrt{d}.$
    \item If $d\ge n (\log n)^8,$ then with high probability
    $|\lambda_2(\bfG)|\le(\log n)^{10}\sqrt{n}.$
\end{itemize}
\end{theorem}
Here, we need the strong bounds in \cref{prop:sparseSOEI} for sparse graphs. As these bounds provably do not hold for $\OEI,$ more work is needed to extend to $p\neq 1/2.$ The proof is in \cref{section:spectral}.

\section{Proving the Bounds on Fourier Coefficients}
\label{sec:mainclaim}
Here, we prove our bounds on Fourier coefficients of random geometric graphs by formalizing the argument in \cref{sec:mainclaim}. Specifically, in \cref{sec:FourierByOEI}, we prove \cref{thm:oeiboundonfourier}. In \cref{sec:FourierBySOEI}, we modify the argument slightly to prove the stronger \cref{eq:SOEI} in the density $1/2$ case.  In \cref{sec:OEIandSOEIbounds} we prove the bounds on the edge independence numbers stated in \cref{prop:boundsonOEI,prop:sparseSOEI}.

\subsection{The Main Argument in \texorpdfstring{\cref{thm:oeiboundonfourier}}{OEI bound}}
\label{sec:FourierByOEI}
Fix a connected graph $H$ on $k= |V(H)|$ vertices and $m=|E(H)|$ edges such that $Ckm(\log d)^{3/2} \le \sqrt{d}.$ Let $\pi$ be any bijective labeling of its vertices by $[k].$ We will identify vertices by their labelling in $\pi$ and optimize over $\pi$ at the end. We prove \cref{thm:oeiboundonfourier} in the Gaussian setting and state the necessary modifications for the spherical setting at the end.

\paragraph{Step 1: High-Probability Bound on $
\sum_{\ell< i}Z_{i\ell}Z_{j\ell} + Z_{ji}(Z_{ii} -1)
$.} Recall \cref{eq:innerproductequationintro}. 
As discussed, dependence between edges is due to the term 
\begin{equation}
\label{eq:definitionofgaussianQ}
\begin{split}
Q^{\mathcal{N}(0,\frac{1}{d}I_d)}_{ji}(\{Z_{j\ell}\}_{1\le \ell \le j}, \{Z_{i\ell}\}_{1\le \ell\le i}) \coloneqq
\sum_{\ell< i}Z_{i\ell}Z_{j\ell} + Z_{ji}(Z_{ii} -1).
\end{split}
\end{equation}
We bound the size of this term, along the way introducing notation that will be used later.
Let $$[\RL, \RU] = 
    \Big[- C\sqrt{\frac{m\log d}{{d}}}, C\sqrt{\frac{m\log d}{{d}}}\Big]$$ be the ``reasonable interval" for each summand in \cref{eq:definitionofgaussianQ}: By Gaussian and $\chi^2$-concentration (\cref{fact:GaussianFatcs,prop:chisquaredconcnetrations}) for any desired constant $C'$, there exists some absolute constant $C$ such that under \eqref{eq:assumption} (which implies $\log 1/p= O(\log d)$ and $ \log n =O(\log d)$) we have
    $\prob[Z_{i\ell}\in [\RL, \RU]]\geq 1-e^{-C'm\log d}$ and the same for $(Z_{ii} -1)$.
Denote by 
$$
\RSC:=\Big\{Z_{uv}\in [\RL, \RU]\text{ and } 
    (Z_{uu} - 1)\in  [\RL, \RU]\; \text{ for all }
    1\le u \le v \le k\Big\}
$$ 
    the ``reasonable set of configurations''. By the union bound its complement has probability
\begin{equation}
\label{eq:reasonablevent}
\begin{split}
    &
    \prob[\RSC^c]
    \le 
    k^2e^{-C'm\log d}
    \le p^md^{-m}.
\end{split}
\end{equation}

As $Q^{\mathcal{N}(0,\frac{1}{d}I_d)}_{ji}$ is a sum of at most $k$ terms of order $(C\sqrt{{m\log d}/{{d}}})^2$ under the high probability event $\RSC$, we conclude that with probability at least $1 - p^md^{-m},$ 
\begin{equation}
\label{eq:boundonQij}
    \Big|
    Q^{\mathcal{N}(0,\frac{1}{d}I_d)}_{ji}(\{Z_{j\ell}\}_{1\le \ell \le j}, \{Z_{i\ell}\}_{1\le \ell\le i})
    \Big|\le
    C^2{\frac{km\log d}{{d}}} = \unitD\quad \text{for all }
    (ji)\in E(H),
\end{equation}
where we defined $\unitD\coloneqq C^2{\frac{km\log d}{{d}}}.$
We condition on $\RS.$ Since $|\prod_{(ji)\in E(H)}(G_{ji} - p)|<1$ a.s.,
\begin{equation}
\begin{split}
    \Big|\E\Big[\prod_{(ji)\in E(H)}(G_{ji} - p)\Big]\Big|
    &\le 
    \Big|\E\Big[\prod_{(ji)\in E(H)}(G_{ji} - p)\Big| \RS\Big]\Big|  + 
    \Pr[\RS^c]\\
    & \le 
    \Big|\E\Big[\prod_{(ji)\in E(H)}(G_{ji} - p)\Big| \RS\Big]\Big|  + 
     p^md^{-m}.
\end{split}
\end{equation}
As $p^{m}d^{-m} = p^{|E(H)|}d^{-|E(H)|}\le 
p^{|E(H)|}\sqrt{d}^{-\OEI(H)}
$, it remains only to bound the first term.

\paragraph{Step 2: Fragile Edges.} Observe that under the high probability event in \cref{eq:boundonQij}, as long as 
$Z_{ji}\not \in \big[\rho^p_d - \unitD, \rho^p_d + \unitD\big],$ 
it is the case that 
$$
\indicator[\langle Z_i,Z_j\rangle \ge \rho^p_d] = 
\indicator[Z_{ji} + {Q}^{\mathcal{N}(0,I_d)}_{ji}\ge \rho^p_d] = 
\indicator[Z_{ji}\ge \rho^p_d]
$$
and variables $Z_{ji}$ are independent (even conditioned on $\RS$). Thus, all edges besides the ones for which $Z_{ji}$ is close to $\rho^p_d$ are independent. We localize dependence to the following fragile edges.

\begin{definition}[Fragile Interval and Fragile Edges]
\label{def:fragileedges}
Denote $\FL = \rho^p_d - \unitD$ and $\FU = \rho^p_d + \unitD.$ The \emph{fragile interval} is $[\FL, \FU]$ and an edge $ji$ is called \emph{fragile} if $Z_{ji}\in [\FL, \FU].$  
\end{definition}

Note that each edge is fragile independently as $\{Z_{uv}\}_{1\le u \le v \le k}$ are independent. Let $\mathcal{F}$ be the set of fragile edges.
Now, $Z_{ji}\sim \mathcal{N}(0,1/d)$ and $\unitD = C^2{\frac{km\log d}{{d}}} = o(1/\sqrt{d}) = o(\rho^p_d/\log d)$ imply 
\begin{equation}
\label{eq:probabilityoffragile}
    \prob\Big[Z_{ji}\in [\FL,\FU]\Big]\le 
     \Delta p\sqrt{C''d\log d}
\end{equation}
for some absolute constant $C''$, because $[\FL,\FU]$ has length $\Delta$ and the Gaussian density around $\rho^p_d$ is $\tilde{O}(p\sqrt{d})$ as $p = \prob[\langle Z_1,Z_2\rangle\ge \rho^p_d] = \prob[Z_{11}Z_{21}\ge \rho^p_d]\approx \prob[Z_{21}\ge \rho^p_d]$ and $Z_{21}\sim\mathcal{N}(0,1/d).$ This is formalized via 
\cref{cor:gaussianinintervals}.
Conditioning on the fragile set yields
\begin{equation}
\label{eq:fragiletowerproperty}
\begin{split}
    & \E\Big[\prod_{(ji)\in E(H)}(G_{ji} - p)\Big| \RS\Big]\\
    & = 
    \sum_{F\subseteq E(H)}
    \E\Big[\prod_{(ji)\in E(H)}(G_{ji} - p)\Big| \RS, \mathcal{F} = F\Big]\times \Pr[\mathcal{F} = F|\RS]\\
    & \le 
     2\sum_{F\subseteq E(H)}
    \Big|\E\Big[\prod_{(ji)\in E(H)}(G_{ji} - p)\Big| \RS, \mathcal{F} = F\Big]\Big|\times 
    \big(\Delta p\sqrt{C''d\log d}\big)^{|F|}.
\end{split}
\end{equation}
We used the fact that $\prob[\RS]\ge 1/2$ so 
$\Pr[\mathcal{F} = F|\RS]\le 2\Pr[\mathcal{F} = F].$ This last conditioning is useful, because our noise operator depends on the set of fragile edges.

\paragraph{Step 3: The noise operator.} Conditioned on the reasonable event $\RSC$ and on the set of fragile edges $\mathcal{F} = F,$ we define the following noise operator. It rerandomizes all variables $Z_{uv}$ such that the value of $Z_{uv}$ does not appear in the expression ${Q}^{\mathcal{N}(0,I_d)}_{ji}$ for any fragile $(ji)\in F.$ In particular, as ${Q}^{\mathcal{N}(0,I_d)}_{ji}$ is only a function of $\{Z_{j\ell}\}_{1\le \ell \le j}, \{Z_{i\ell}\}_{1\le \ell\le i},$ we rerandomize all $(uv)\in E(H)\backslash (F\cup \partial^\pi_H(F)).$
For $(uv)\in E(H)\backslash (F\cup \partial^\pi_H(F)),$ conditioned on $\RS$ and $\mathcal{F} = F,$ the variable $Z_{uv}$ uniquely determines $G_{uv}$ (as $(uv)$ is not fragile). Furthermore,  $Z_{(uv)}$ is independent of $(G_{ji})_{(ji)\in E(H)\backslash \{(uv)\}}$ (as there is no fragile edge of the form $(ji)$ for which $(uv)\in \{{j\ell}\}_{1\le \ell \le j}\cup \{{i\ell}\}_{1\le \ell\le i}$
). For clarity and uniformity with \cref{obs:plantedclqiuenoiseoperator}, we spell this out separately.

\begin{observation}[Noise Operator View on $\RGG$]
\label{obs:RGGnoiseoperator}
The noise operator $T_p^{\pi}$ on the distribution $\RGGgauss$ is parametrized by an ordering of the vertices $\pi$ and marginal edge probability $p.$ To sample from $\RGG,$ one first samples a fragile set $\cF$ by including each edge independently with probability $\prob[Z_{ji}\in [\FL,\FU]]$ (recall that edges are fragile independently). $\cF$ together with $\pi$ determines $\partial_H^\pi \cF.$ Then, one samples $H_{\cF\cup\partial_H^\pi \cF}$ from the marginal distribution on edges $\cF\cup\partial_H^\pi \cF$ from the distribution $\RGGgauss$ conditioned on $\cF$ being the fragile set with respect to $\pi.$ Conditioned on $\cF,H_{\cF\cup\partial_H^\pi \cF},$ $T_p^{\pi}$ acts independently on edges with the following noise rates:
\begin{align*}
    T_p^{\pi}(G)_{uv} = 
    \begin{cases}
    H_{uv}\; \text{ for }(uv)\in \cF\cup\partial_H^\pi \cF \text{ (case of no noise)},\\
    \bern(p)\; \text{ otherwise (case of full noise).}
    \end{cases}
\end{align*}
That is, there is no noise on the edges in $\cF\cup\partial\cF$ and the rest of the edges are fully rerandomized. Hence, $T_p^{\pi}(\delta_{K_n})$ is a sample from $\RGGgauss.$

One key difference with \cref{obs:plantedclqiuenoiseoperator} is that the distribution of the edge set $H_{\cF\cup\partial^\pi \cF}$ is much more complicated than the distribution on the planted clique $A.$ The latter is simply a clique, while $H_{\cF\cup\partial_H^\pi \cF}$ is a subgraph of a random geometric graph which is further conditioned on its fragile set.

Nevertheless, just as in \cref{obs:plantedclqiuenoiseoperator}, the independent rerandomization over the rest of the edges $(\cF\cup\partial_H^\pi \cF)^c$ yields an (exponentially fast) decay of Fourier coefficients, which we discuss next.
\end{observation}

We separate the conditional signed expectation into the portion rerandomized by the noise operator, i.e. $E(H)\backslash(F\cup \partial^\pi_H(F)),$ and a portion that is not rerandomized, i.e. $F\cup \partial^\pi_H(F)$:

\begin{align}
%
    & \Big|\E\Big[\prod_{(ji)\in E(H)}(G_{ji} - p)\Big| \RS, \mathcal{F} = F\Big]\Big| \notag\\
    & = 
    \Big|
    \prod_{(ji)\in E(H)\backslash (F\cup \partial^\pi_H(F))}\E\Big[(G_{ji} - p)\Big| \RS, \mathcal{F} = F\Big]\times 
    \E\Big[\prod_{(ji)\in F\cup \partial^\pi_H(F)}(G_{ji} - p)\Big| \RS, \mathcal{F} = F\Big]\Big|\notag\\
    & \le
    \prod_{(ji)\in E(H)\backslash (F\cup \partial^\pi_H(F))}\Big| 
    \expect\Big[\indicator[Z_{ji}\ge \rho^p_{d}] - p\Big|Z_{ji}\in [\RL,\RU]\backslash [\FL,\FU]\Big]
    \Big|\notag\\
    &\quad\qquad\qquad\qquad\qquad\qquad\qquad \times 
    \E\Big[\prod_{(ji)\in F\cup \partial^\pi_H(F)}|G_{ji} - p|\;\Big| \RS, \mathcal{F} = F\Big].\label{eq:signedconditionedonF}
\end{align}
Next, we bound the factors for $E(H)\backslash (F\cup \partial^\pi_H(F))$ and then the factors for $F\cup \partial^\pi_H(F)$. 

\paragraph{Step 4: Factors $E(H)\backslash (F\cup \partial^\pi_H)$ Rerandomized By The Noise Operator.} A simple calculation with 1-dimensional Gaussian variables using \cref{eq:reasonablevent,eq:probabilityoffragile} 
and the fact that $\prob[Z_{ji}\ge \rho^p_d]= p+ O(p(\log d)/\sqrt{d})$ (see \cref{cor:gaussianinintervals}) gives
\begin{equation}
\label{eq:marginalofnonfragile}
    \Big| 
    \expect\Big[\indicator[Z_{ji}\ge \rho^p_{d}] - p\Big|Z_{ji}\in [\RL,\RU]\backslash [\FL,\FU]\Big]
    \Big|\le \Delta p\sqrt{C'''d\log d}.
\end{equation}

\paragraph{Step 5: Factors $F\cup\partial^\pi_HF$.}
Convexity of $|\cdot|,$ the fact $G_{ji}\le 1$ a.s., and triangle inequality give
\begin{align}
    & \E\Big[\prod_{(ji)\in F\cup \partial^\pi_H(F)}\big|G_{ji} - p\big|\;\Big| \RS, \mathcal{F} = F\Big]\notag\\
    & \le 
\sum_{B\subseteq F\cup \partial^\pi_H(F)}
    p^{|F\cup \partial^\pi_H(F)| - |B|}
    \E\Big[\prod_{(ji)\in B}G_{ji} \; \Big|\; \RS, \mathcal{F} = F\Big]\notag\\
    &\le 
    \sum_{B\subseteq F\cup \partial^\pi_H(F)}
    p^{|F| + |\partial^\pi_H(F)| - |B|}
    \E\Big[\prod_{(ji)\in B\backslash F}G_{ji}\; 
    \Big|\; \RS, \mathcal{F} = F\Big].\label{eq:contributionofadjacenttofragile}
\end{align}
For each $(ji)\in B\backslash F,$ the edge $G_{ji}$ is determined by $Z_{ji},$ conditioned on $\RS, \mathcal{F} = F.$ Thus, 
$$
\E\Big[\prod_{(ji)\in B\backslash F}G_{ji}\; \Big|\; \RS, \mathcal{F} = F\Big] = 
\expect\big[\indicator[Z_{ji}\ge \rho^p_d]| \RS, \mathcal{F} = F\big]^{|B| - |F|}.
$$
By \cref{eq:marginalofnonfragile}, the last expression is at most $ (2p)^{|B| - |F|}.$ Plugging this into \cref{eq:contributionofadjacenttofragile},
\begin{align}
&    \E\Big[\prod_{(ji)\in F\cup \partial^\pi_H(F)}\big|G_{ji} - p\big|\;\Big| \RS, \mathcal{F} = F\Big] \notag \\
    & \qquad\qquad\le 
    \sum_{B\subseteq F\cup \partial^\pi_H(F)}
    p^{|F| + |\partial^\pi_H(F)| - |B|}\times (2p)^{|B| - |F|}\le 4^{|E(H)|}(2p)^{|\partial^\pi_HF|}.\label{eq:contribAdjFrag2}
\end{align}

\paragraph{Step 6: Putting It All Together.}
Plugging \cref{eq:marginalofnonfragile} and \cref{eq:contribAdjFrag2} into \cref{eq:signedconditionedonF}, the conditional Fourier coefficients are bounded as
\begin{equation*}
    \begin{split}
        \Big|\E\Big[\prod_{(ji)\in E(H)}(G_{ji} - p)\Big| \RS, \mathcal{F} = F\Big]\Big|\le 
        \big(\Delta p\sqrt{C'''d\log d}\big)^{|E(H)|- |F| - |\partial^\pi_H F|}
        \times 4^{|E(H)|}\times (2p)^{|\partial^\pi_HF|}.
    \end{split}
\end{equation*}
We combine this with \cref{eq:fragiletowerproperty} to obtain
\begin{align*}
    &\Big|\E\Big[\prod_{(ji)\in E(H)}(G_{ji} - p)\Big| \RS\Big]\Big|\\
    & \le 
    2\sum_{F\subseteq E(H)}
    \big(\Delta p\sqrt{C''d\log d}\big)^{|E(H)|- |F| - |\partial^\pi_H F|}
        \times 4^{|E(H)|}\times (2p)^{|\partial^\pi_HF|}
    \times 
    \big(\Delta p\sqrt{C'''d\log d}\big)^{|F|}\\
    & \le \max_{F\subseteq E(H)}
    2(8p)^{|E(H)|}
    \big(\Delta\sqrt{Cd\log d}\big)^{|E(H)|- |\partial^\pi_HF|}
    \le
    \max_{F\subseteq E(H)}
    2(8p)^{|E(H)|}
    \Big(\frac{Ckm(\log d)^{3/2}}{\sqrt{d}}\Big)^{|E(H)|- |\partial^\pi_HF|}.
\end{align*}

All that is left to show is $|E(H)|- |\partial^\pi_HF|\ge \OEI_\pi(H).$ This follows immediately because for any $F\subseteq E(H),$ the set $E(H)\backslash \partial^\pi_HF$ satisfies the covering properties \cref{def:covering} (note that\linebreak  $F\subseteq E(H)\backslash \partial^\pi_HF$). Finally, we can choose $\pi$ as the maximizer of $\OEI_\pi(H)$ and conclude \cref{thm:oeiboundonfourier} in the Gaussian case.

\begin{remark}[Energy-Entropy Trade-off] 
\label{rmk:energyentropytradeoff}
$|E(H)|- |\partial^\pi_HF|\ge \OEI_\pi(H)$ highlights the following energy-entropy trade-off phenomenon in $\RGG.$ Rewrite it as $|F| +|E(H)\backslash (F\cup \partial^\pi_HF)|\ge \OEI_\pi(H).$ The term $|E(H)\backslash (F\cup \partial^\pi_HF)|$ corresponds to entropy in the distribution as it is the size of the subset of edges which the noise operator rerandomizes (and are independent with all edges in $E(H)$). The term $|F|$ measures energy as $\cF= F$ is the subset of edges with non-trivial interactions (dependence) with other edges in $|E(H)|.$ 
The inequality shows that energy and entropy cannot both be small, and either one being large results in small Fourier coefficient: entropy due to randomness and energy due to low probabilities. 
\end{remark}

\paragraph{The Spherical Case.} The analysis of the spherical case is nearly the same. We generate $V_1, V_2, \ldots, V_k\sim_{iid}\unif(\dsphere)$ as $V_i = Z_i/\|Z_i\|_2,$ where $Z_1, Z_2, \ldots, Z_k\sim_{iid}\mathcal{N}(0,\frac{1}{d}I_d).$ Then, we apply the Gram-Schmidt process on $Z_1, Z_2, \ldots, Z_k.$ With respect to the Bartlett decomposition,
\begin{equation*}
 \langle V_i,V_j\rangle = 
\frac{1}{\|Z_i\|_2\times \|Z_j\|_2}\langle Z_i, Z_j\rangle
 = 
Z_{ji} + {Q}^{\dsphere}_{(ji)}(\{Z_{j\ell}\}_{1\le \ell \le j}, \{Z_{i\ell}\}_{1\le \ell\le i})\,,
\end{equation*}
where
\begin{align*}
   &{Q}^{\dsphere}_{(ji)}(\{Z_{j\ell}\}_{1\le \ell \le j}, \{Z_{i\ell}\}_{1\le \ell\le i})
   \\&\qquad
   \coloneqq\Big(
\sum_{\ell< i}Z_{i\ell}Z_{j\ell} + Z_{ji}(Z_{ii} -1)
\Big)\frac{1}{\|Z_i\|_2\times \|Z_j\|_2} + 
Z_{ji}\times \Big(\frac{1}{\|Z_i\|_2\times \|Z_j\|_2} - 1\Big). 
\end{align*}
The function ${Q}^{\dsphere}_{(ji)}$ depends on the exact same set of variables as $Q^{\mathcal{N}(0,\frac{1}{d}I_d)}_{ji}$ and also takes value in $[-C'''pmk\log d/\sqrt{d},C'''pmk\log d/\sqrt{d}]$ under $\RS.$ The rest of the analysis is identical.

\subsection{Bounds on the Ordered Edge Independence Numbers.}
\label{sec:OEIandSOEIbounds}
\begin{proof}[Proof of first part of \cref{prop:boundsonOEI}]
We first bound $\OEI_\pi(H)$ by another quantity. Denote by $\mathcal{SN}_\pi(H)$ the set of vertices $j$ of $H$ for which there exists some $i<j$ such that 
$(ji)\in E(H),$ i.e. with a smaller neighbour. Then, 
$\OEI_\mathcal{A}(H)\ge 
\lceil |\mathcal{SN}_\pi(H)|/2\rceil.
$ To prove this, suppose, for the sake of contradiction, that there exists a set of $r < |\mathcal{SN}_\pi(H)|/2$
edges $(j_1, i_1), (j_2, i_2), \ldots, (j_r, i_r)$ that satisfy the covering properties from \cref{def:covering}. Since $r<|\mathcal{SN}_\pi(H)|/2,$
$
\Big|\{j_1, i_1, j_2, i_2, \ldots, j_r, i_r\}\Big| <
|\mathcal{SN}_\pi(H)|.
$
Thus, there exists some vertex $j^*\in \mathcal{SN}_\pi(H)$ such that $j^*\not \in \{j_1, i_1, j_2, i_2, \ldots, j_r, i_r\}.$ Let $i^*<j^*$ be a vertex such that $(j^*i^*)$ is an edge. Such an $i^*$ exists by the definition of $\mathcal{SN}_\pi(H).$ But then, clearly, 
$
(j^*i^*)\not \in 
\partial^\pi_HF.
$
This contradicts the fact that the
edges $(j_1, i_1), (j_2, i_2), \ldots, (j_r, i_r)$ satisfy the covering properties.

Now, we need to show that $\max_\pi |\mathcal{SN}_\pi(H)|\ge |V(H)|-1$ for connected $H.$ Let $T$ be a rooted spanning tree of $H.$ Define $\pi$ to be any labelling of $H,$ such that for all $i,$ all vertices on level $i+1$ have a larger label than the vertices on level $i.$ Clearly, the root is the only vertex without a neighbour with a smaller label. \end{proof}

\begin{proof}[Proof of second part of \cref{prop:boundsonOEI}]
Recall that $\delta(H):=\max_{S\subseteq V(H)}\{|S|- |N_H(S)|\},$ where $N_H(S):=\{j \in V(H)\; : \; \exists i \in S\text{ s.t. }(ji)\in E(H)\}$
The inequality $\OEI(H)\ge \delta(H)+1$ is clear if $\delta(H) = 0,$ so we assume throughout that
$\delta(H)\ge 1.$ We need the following fact which we prove for completeness.
\begin{proposition}[\cite{Alon1981OnTN}]
\label{prop:independentmaximizerofdelta}
There exists an independent set $S'\in \arg\max_S\{|S| - |N_H(S)|\}.$
\end{proposition}
\begin{proof} Take any $S''\in \arg\max_S\{|S| - |N_H(S)|\}.$ Then, $S' = S''\backslash N(S'')$ is an independent set and satisfies $|S'| - |N_H(S')|\ge |S''| - |N_H(S'')|.$
\end{proof}

Let $S'\in \arg\max_S{|S| - |N_H(S)|}$ be independent and $N_H(S') = \{u_1, u_2, \ldots, u_\ell\}.$ Define the sets
\begin{enumerate}
    \item $S'_1 = N_H(u_1)\cap S'$ and
    \item $S'_i = \big(N_H(u_i)\backslash (N_H(u_1)\cup N_H(u_2)\cup\cdots\cup N_H(u_{i-1}))\big)\cap S'.$
\end{enumerate}

Take any ordering $\pi$ such that the following vertices appear in the following decreasing order:
$$
S_1', u_1, S_2', u_2, \ldots, S_\ell', u_\ell, V(H)\backslash (S'\cup N_H(S')),
$$
i.e. $\pi(x)> \pi(u_1)$ for all $x\in S_1',$
$\pi(u_1)>\pi(y)$ for all $y \in S_2',$
and so on. Let $F$ be any set satisfying the covering properties in \cref{def:covering} with respect to this ordering. Observe that for each $k\in [\ell],v\in S'_k,$ the vertex $u_k$ is the only neighbour of $v$ among $S_1'\cup S_2'\cdots\cup S'_k\cup \{u_1, u_2, \ldots,u_k\}.$ Thus, the set $F$ must include the edge $(v u_k)$ as otherwise 
$(vu_k)\not \in \partial^\pi_HF.$
As this holds for each  $k\in [\ell],v\in S'_k,$ there are at least $\sum_{k = 1}^\ell |S'_k| = |S'| = \delta(H) + |N(S')|\ge \delta(H)  + 1$ edges in $F.$
\end{proof}

\begin{proof}[Proof of \cref{prop:sparseSOEI}]
Suppose that there exists a set $F$ satisfying the strong covering property for which $|F|<2|V(H)|- |E(H)| - 2.$ In particular, the graph defined by vertices $V(H)$ and edges $F$ has $t > |V(H)|- (2|V(H)|- |E(H)| - 2) = |E(H)| +2- |V(H)|$ connected components. Let $C_1, C_2, \ldots, C_t$ be these connected components. Since $H$ is connected, there exist at least $t-1$ pairs of  different connected components $C_i, C_j$ with an edge between them. However, if $(uv)$ is such an edge between different connected components, where $u\in C_i, v\in C_j$ and $u>v,$
$v$ must have another neighbour in $C_i$ as $F$ satisfies the strong covering property in \cref{def:strongcovering}. Thus, whenever there is an edge between $C_i$ and $C_j,$ there are at least two such edges. Hence, the total number of edges in $H$ is at least 
$$2(t-1)+\sum_{i= 1}^t(|V(C_i)| - 1) = 
|V(H)| - t + 2t - 2 = 
|V(H)| + t - 2> |E(H)|,
$$  
which is a contradiction. \end{proof}

\begin{remark} This proof holds for any ordering $\pi.$ We expect that choosing an optimal $\pi$
will yield an improved bound.
\end{remark}

\subsection{Improving The Bound in The Half-Density Case in \texorpdfstring{\cref{eq:SOEI}}{SOEI bound}}
\label{sec:FourierBySOEI}
The density $1/2$ case is special as there is a measure preserving map between $[\RL, \FL]$ and $[\FU, \RU]$ which also preserves norms -- namely, $\Xi(Z),$ where
$\Xi(Z):= -Z.$ In particular, the analogue of \cref{eq:marginalofnonfragile} is 
$$ 
    \expect\Big[\indicator[Z_{ji}\ge \rho^{1/2}_{d}] - 1/2\Big|Z_{ji}\in [\RL,\RU]\backslash [\FL,\FU]\Big]
     = 0.
    $$
Thus, unless $(F\cup \partial^\pi_HF) = \emptyset,$ the expression in \cref{eq:signedconditionedonF} is equal to 0. This immediately yields an improved bound on the Fourier coefficient of the form $2\times p^{|E(H)|}\times (C'''pkm(\log d)^{3/2}/\sqrt{d})^{\OEI(H)}.$

\paragraph{A more powerful noise operator.}
The map $\Xi(Z),$ however, allows us to do more. One can apply a noise operator on certain edges adjacent to fragile edges. The reason is that 
$$
Q^{\mathcal{N}(0,\frac{1}{d}I_d)}_{ji}(\{Z_{j\ell}\}_{1\le \ell \le j}, \{Z_{i\ell}\}_{1\le \ell\le i}) = 
Q^{\mathcal{N}(0,\frac{1}{d}I_d)}_{ji}(\{\Xi(Z_{j\ell})\}_{1\le \ell \le j, i \neq j}, \{\Xi(Z_{i\ell})\}_{1\le \ell< i}, Z_{ii}, Z_{ji})
$$
and similarly for the spherical analogue $Q^{\dsphere}_{ji}.$

Namely, condition on $\mathcal{F} = F.$ Take any $(ji) \in E(H)\backslash F$ such that, furthermore, $j$ is the unique neighbour of $i$ according to $H$ in 
$V(F^\pi_{\ge i, \ni j})$ (equivalently, 
$(ji) \in E(H) \backslash (F\cup \eth^{\,\pi}_H F)$
by \cref{def:strongcovering}). This means that the operation 
$(Z_{ai})_{a\in V(F^\pi_{\ge i, \ni j})}\longrightarrow 
(\Xi(Z_{ai}))_{a\in V(F^\pi_{\ge i, \ni j})}
$ changes $G_{ji}$ to $1- G_{ji}$ but leaves all other edges $G_{uv}$ unchanged. Indeed, consider the cases for $(uv)$:
\begin{enumerate}
    \item $(uv) = (ja)$ for some $a\in V(F^\pi_{\ge i, \ni j}).$ However, $(ji)$ is the unique edge with this property.
    \item $(uv)$ is not fragile and not of the form $(ja)$ for some $a\in V(F^\pi_{\ge i, \ni j}).$ Then, $Z_{uv}$ remains unchanged and so does $G_{uv}$ as $uv$ is not fragile.
    \item $(uv)$ is fragile and at least one of $u$ and $v$ is not in $V(F^\pi_{\ge i, \ni j}).$ Then, either $\min(u,v)\le i,$ in which case $G_{uv}$ can only depend on 
    $Z_{ai}$ via its norm, $Z_{ai}^2$ (this follows from the definitions of ${Q}^{\standardgaussd}_{uv}$ and ${Q}^{\dsphere}_{uv}$). However, 
    $Z_{ai}^2 = \Xi(Z_{ai})^2.$
    \item $(uv)$ is fragile and 
    $u,v\in V(F^\pi_{\ge i, \ni j}),$ in which case 
    $G_{uv}$ depends on terms of the form $Z_{ai}$ via the product $Z_{ui}Z_{vi},$ but $Z_{ui}Z_{vi} = 
    \Xi(Z_{ui})\Xi(Z_{vi})$ (again, this follows from the definitions of ${Q}^{\standardgaussd}_{uv}$ and ${Q}^{\dsphere}_{uv}$).
\end{enumerate}

With this property in mind, we define the following noise operator: for each edge $(ji) \in E(H) \backslash (F\cup \eth^{\,\pi}_H F),$
one independently applies with probability $1/2$ the operation\linebreak 
$(Z_{ai})_{a\in V(F^\pi_{\ge i, \ni j})}\longrightarrow 
(\Xi(Z_{ai}))_{a\in V(F^\pi_{\ge i, \ni j})}.
$ 
As long as there is at least one such edge, 
the corresponding expected signed weight $\SW_H^{1/2}$ becomes zero.

By \cref{def:strongcovering}, this is always the case if $|F|< \SOEI_\pi(H).$ An analogous argument to the one in \cref{sec:FourierByOEI} allows us to bound 
$$
\Big|\expect\Big[\prod_{ji\in E(H)}(G_{ji} - 1/2)\Big]\Big|\le
2\times p^{|E(H)|}\times (C'''pkm(\log d)^{3/2}/\sqrt{d})^{\SOEI_\pi(H)}.
$$
Optimizing over the labelling $\pi$ yields \cref{eq:SOEI}.

\begin{remark}[Why not other densities?]
\label{rmk:whynototherdensities}
In principle, we could have carried out the same argument for other densities by fixing some measurable bijection $\Xi_p : [\RL,\FL]\longrightarrow[\FU,\RU], \Xi_p: [\RL,\FL]\longrightarrow[\FL,\RL].$ Then, we apply it independently to variables $(Z_{ai})_{a\in V(F^\pi_{\ge i, \ni j})}$ for edges $(ji)\in E(H)\backslash(F\cup \eth^{\,\pi}_HF)$ with some probability $\tilde{p}$ so that marginal distributions remain Gaussian. The difficulty with this approach is that for essentially any other density besides $1/2,$
the equality $Z_{ui}Z_{vi} = 
    \Xi_p(Z_{ui})\Xi_p(Z_{vi})$ will not hold. 
    Thus, the values of ${Q}^{\standardgaussd}$ 
    (resp, ${Q}^{\dsphere}$) will change and so
    it is not clear how fragile edges are affected by the respective noise operator.
\end{remark}

\section{Preliminaries}
\label{section:preliminaries}
\subsection{Graph Notation}
\label{section:graphnotation}
Our graph notation is mostly standard. We denote by $V(H),E(H)$ the vertex and edge sets of a graph. For $i \in V(H),$ we define $N_H(i):=\{j\in V(H)\;:\; (ji)\in E(H)\}$ and for $S\subseteq V(H),$ $N_H(S):=\cup_{i\in S}N_H(i).$ 
$\delta(H):=\max_{S\subseteq V(H)}|S| - |\{j \in V(H)\; : \; \exists i \in S \text{ s.t. }(ji)\in E(H)\}|.$
Denote by $\CC(H)$ the set of connected components of $H.$ We denote by $K_n$ the complete graph on $n$ vertices, and by $\Star_n$ the star graph which has one central vertex $v$ and $n$ leaves $i_1,i_2,\ldots,i_n$ adjacent to $v.$

We will frequently define graphs by the edges that induce them. That is, for $A\subseteq E(K_n),$ we also denote by $A$ the graph on vertex set $\{i\in V(K_n)\; :\;\exists j\in V(K_n)\text{ s.t. }(ji)\in A\}$ and edge set $A.$ Identifying graphs by edges that span them is convenient as edges are the variables of polynomials that we consider.

We define by $\graphs_D$ the set of graphs (up to isomorphism) on at least 1 edge and at  most $D$ edges and by $\graphsnoleaves_D\subseteq \graphs_D$ the subset of those graphs with no leaves.

\subsection{Low-Degree Polynomials}
\label{section:ldp}
Our results in \cref{sec:introSphericaltoER,sec:introGaussiantoER,sec:introsphericaltoPCol} are based on the low-degree polynomial framework introduced in \cite{hopkins2017bayesian,hopkins18}. One way to motivate it is the following. When testing between graph distributions $H_0$ and $H_1$ (say, $H_0 =\ergraph, H_1 = \RGGsphere$), one observes a single graph $\bfG$ and needs to output 0 or 1. The graph $\bfG$ is simply a bit sequence in $\{0,1\}^{E(K_n)}.$ Hence, the output is a function $f:\{0,1\}^{E(K_n)}\longrightarrow\{0,1\}.$ All Boolean functions are polynomials \cite{ODonellBoolean}. Therefore, one simply needs to compute a polynomial in the edges. Importantly, one can write polynomials over $\{0,1\}^m$ in their Fourier expansion. In the $p$-biased case over graphs, one represents $f:\{0,1\}^{E(K_n)}\longrightarrow\mathbb{R}$ as
\begin{equation}
\label{eq:pbasedFourierexpansion}
    f(G)= \sum_{H\; : \; H\subseteq E(K_n)}
    \widehat{f}(H)\times (\sqrt{p(1-p)})^{-|E(H)|}
    \prod_{(ji)\in E(H)}(G_{ji}-p).
\end{equation}
Here, $\widehat{f}(H)$ is just a constant (the Fourier coefficient corresponding to $H$) and \linebreak 
$\Big\{\sqrt{p(1-p)}^{-|E(H)|}
    \prod_{(ji)\in E(H)}(G_{ji}-p)\Big\}_{H\subseteq E(K_n)}$ is a basis of polynomials. Conveniently, as can be seen from \cref{eq:definefourier}, this basis is composed of signed-subgraphs $\SW^p_H.$ What makes it useful is the following fact \cite{ODonellBoolean}: 
\begin{equation}
\label{eq:orthonormalityofstandardbasis}
\begin{split}
    &\text{The polynomials }\Big\{\sqrt{p(1-p)}^{-|E(H)|}
    \prod_{(ji)\in E(H)}(G_{ji}-p)\Big\}_{H\subseteq E(K_n)}\\
    &
    \text{are orthonomral with respect to }
    \bfG\sim\ergraph.
    \end{split}
\end{equation}

When computationally restricted to poly-time algorithms, a tester needs to apply a poly-time computable polynomial $f.$ What are classes of poly-time computable polynomials? One such class is of sufficiently low-degree polynomials (where degree refers to the largest number of edges in a monomial corresponding to some $H$ for which $\widehat{f}(H)$ is non-zero). Since those are usually not $\{0,1\}$-valued, one needs to threshold after computing the polynomial, which leads to the following definition, motivated by Chebyshev's inequality.

\begin{definition}[Success of a Low-Degree Polynomial, e.g. \cite{hopkins18}] 
\label{def:successofLDP}
We say that a polynomial $f:\{0,1\}^{E(K_n)}\longrightarrow\{0,1\}$ distinguishes  $H_0$ and $H_1$ with high probability if 
$$
\Big|
\expect_{\bfG\sim H_0}[f(\bfG)] - 
\expect_{\bfG\sim H_1}[f(\bfG)]
\Big| = 
\omega\Big(
\sqrt{
\var_{\bfG\sim H_0}[f(\bfG)] + 
\var_{\bfG\sim H_1}[f(\bfG)]}
\Big).
$$
If $f$ is poly-time computable, this leads to the poly-time algorithm which compares $f(\bfG)$ to\linebreak 
$\big(\expect_{\bfG\sim H_0}[f(\bfG)] +
\expect_{\bfG\sim H_1}[f(\bfG)]\big)/2.$
\end{definition}
Very commonly, one takes $f$ to be a signed subgraph count \cite{Bubeck14RGG}. That is, for some small graph $A$ (e.g. triangle or wedge), one computes 
\begin{equation}
\label{eq:signedcounts}
\SC^p_A(\bfG):=\sum_{H\subseteq K_n\; : H\sim A}\SW^p_H(\bfG),
\end{equation}
where $H\sim A$ denotes graph isomorphism. I.e., one computes the total number of signed weights.

Importantly, the framework of \cite{hopkins18} allows one to refute the existence of low-degree polynomials which distinguish with high probability $H_0$ and $H_1.$  Namely, the condition in \cref{def:successofLDP} fails for all low-degree polynomials. Of course, one needs to quantify  ``low-degree''. Typically, this means degree $O(\log n).$ While not all $O(\log n)$-degree polynomials are necessarily poly-time computable, the class of $O(\log n)$-degree polynomials captures a broad class of algorithms including subgraph counting algorithms \cite{hopkins18}, spectral algorithms \cite{bandeira2019computational}, SQ algorithms (subject to certain conditions) \cite{brennan2021statisticalquery}, approximate message passing algorithms (with constant number of rounds) \cite{montanari2022equivalence}, and are in general conjectured to capture all poly-time algorithms for statistical tasks in sufficiently noisy high-dimensional regimes \cite{hopkins18}. 

\begin{definition}[Low-Degree Polynomial Hardness]
\label{def:ldphardness}
We say that no low-degree polynomial distinguishes $H_0$ and $H_1$ with probability $\Omega(1)$ if there exists some $D= \omega(\log n)$ such that 
$$
\Big|
\expect_{\bfG\sim H_0}[f(\bfG)] - 
\expect_{\bfG\sim H_1}[f(\bfG)]
\Big| = 
o\Big(
\sqrt{
\var_{\bfG\sim H_0}[f(\bfG)] + 
\var_{\bfG\sim H_1}[f(\bfG)]}
\Big)
$$
holds for all polynomials of degree at most $D.$ In particular, this holds (see \cite{kothari2023planted}, for example), if $\expect_{\bfG\sim H_1}[f(\bfG)]/\expect_{\bfG\sim H_0}[f(\bfG)^2] - 1 = o_n(1)$ for all polynomials $f$ of degree at most $D.$   
\end{definition}

Using the orthonormality in \cref{eq:orthonormalityofstandardbasis}, this condition simplifies significantly when\linebreak
$H_0 = \mathsf{G}(N,p)\odot\mathcal{M}$ for some mask $\mathcal{M}.$\footnote{This fact holds and usually stated for general binary distributions when $H_0$ is a product distribution \cite{hopkins18}, but we only state the result in the case of interest.} Recall the notation in \cref{eq:definefourier}.

\begin{claim}[\cite{hopkins18}]
\label{claim:ldphardagainstER}
Suppose that $H_0 = \mathsf{G}(N,p)\odot\mathcal{M}, H_1= \R\odot \mathcal{M}$ and $D = \omega(\log n).$ If
$$
\sum_{H\subseteq E(\mathcal{M})\;: \; 1\le |E(H)|\le D}
\Big(\Phi^p_\R(H)\Big)^2= o_n(1),
$$
no low-degree polynomial can distinguish $H_0$ and $H_1$ with probability $\Omega(1).$
\end{claim}

For our results on the planted coloring, we need a more sophisticated version of \cref{claim:ldphardagainstER} due to \cite{kothari2023planted} when $H_0= \PCol.$ We phrase and prove a variant of it in \cref{prop:chisquaredtocoloring}. When applying \cref{claim:ldphardagainstER} and variations of it, we will frequently reduce calculations to the following inequality. It is similar to \cite[Proposition 4.9] {kothari2023planted}.
\begin{proposition}
\label{prop:whenchisquaredissmall}
For any absolute constant $\theta>0$ and $D= O((\log n)^{1.1}),$
$$
\sum_{H \in \graphs_D} n^{-\theta|V(H)|}=o_n(1).
$$
\end{proposition}
\begin{proof}
When $v = |V(H)|\le D^{2/3},$ the number of graphs on $v$ vertices is at most $2^{\binom{v}{2}}\le 2^{v^2}\le 2^{vD^{2/3}} = 2^{o_n(|V(H)|\log n)}.$ On the other hand, when $v = |V(H)|\ge D^{2/3},$ the number of graphs on $v$ vertices and at most $D$ edges is at most $\sum_{j = 1}^D\binom{\binom{v}{2}}{D}\le 
D\Big(\frac{2v^2}{D}\Big)^D\le 2^{4D\log v}.
$ As $D = (\log n)^{1.1},$
$v\ge D^{2/3},$ one has 
$4D\log v = 4v (D\log v)/v = o_n(v\log n) = o_n(|V(H)|\log n).$ In either case, the number of graphs on $v$ vertices and at most $D$ edges is $2^{o_n(|V(H)|\log n)},$ which for large enough $n$ is at most 
$e^{\frac{\theta}{2}|V(H)|\log n}.$
Hence, for large enough $n,$ 
$$
\sum_{H \in \graphs_D}
e^{-\theta|V(H)|\log n} \le 
\sum_{v = 1}^{2D}
e^{-\frac{\theta}{2}v\log n} = O(Dn^{-\theta/4}) = o_n(1),
$$
as desired. We used the fact that a graph induced by $D$ edges has at most $2D$ vertices.
\end{proof}

\subsection{Signed Expectations of Graphs with Leaves}
\label{section:graphswithleaves}
Graphs with leaves are special to the current work as they separate $\RGGgauss$ and\linebreak $\RGGsphere.$ In the spherical case, one can prove the following fact. 

\begin{proposition}
\label{prop:leavesinspherical}
Suppose that $H$ is graph with a leaf $u$ which has neighbour $v.$ Suppose further that $(f_{ji})_{(ji)\in E(H)}:[-1,1]\longrightarrow [0,1]$ are some measurable  functions.
Then,
\begin{align*}
&\expect_{V_1, V_2, \ldots, V_n\sim_{iid} \unif(\dsphere)}\Big[\prod_{(ji)\in E(H)}f_{ji}(\langle V_i,V_j\rangle)\Big]\\
& = 
\expect_{\{V_i\}_{i\neq u}\sim_{iid} \unif(\dsphere)}\Big[\prod_{(ji)\in E(H)\backslash \{(uv)\}}f_{ji}(\langle V_i,V_j\rangle)\Big]\times 
\expect_{V_u,V_v\sim_{iid} \unif(\dsphere)}\Big[f_{uv}(\langle V_u,V_v\rangle)\Big].
\end{align*}
In particular, if $H$ is a tree,
$$
\expect_{V_1, V_2, \ldots, V_n\sim_{iid} \unif(\dsphere)}\Big[\prod_{(ji)\in E(H)}f_{ji}(\langle V_i,V_j\rangle)\Big] = 
\prod_{(ji)\in E(H)}\expect_{V_i,V_j\sim_{iid} \unif(\dsphere)}\Big[f_{ji}(\langle V_i,V_j\rangle)\Big].
$$
\end{proposition}
\begin{proof}
    This follows immediately from the following homogeneity property of $\RGGsphere:$
\begin{equation}
\label{eq:homogeneity}
    \text{For every fixed $v_j\in \dsphere,$ }\langle v_j, V_i\rangle \text{ has the same distribution}.
\end{equation}
Hence, 

\begin{equation*}
    \begin{split}
        & 
        \expect_{V_1,V_2, \ldots, V_n\sim_{iid} \unif(\dsphere)}\Big[\prod_{(ji)\in E(H)}f_{ji}(\langle V_i,V_j\rangle)\Big|
        \Big]\\
        & =
        \expect_{\{V_i\}_{i\neq u}\sim_{iid} \unif(\dsphere)}\Big[\expect_{V_u\sim\unif(\dsphere)}\Big[\prod_{(ji)\in E(H)}f_{ji}(\langle V_i,V_j\rangle)\Big|\{V_i\}_{i\neq u}\Big]\Big] \\
        & =
        \expect_{\{V_i\}_{i\neq u}\sim_{iid} \unif(\dsphere)}\Big[\prod_{(ji)\in E(H)\backslash \{(uv)\}}f_{ji}(\langle V_i,V_j\rangle)\expect_{V_u\sim\unif(\dsphere)}\Big[f_{uv}(\langle V_u,V_v\rangle)\Big|V_v\Big]\Big]\\
        & = 
        \expect_{\{V_i\}_{i\neq u}\sim_{iid} \unif(\dsphere)}\Big[\prod_{(ji)\in E(H)\backslash \{(uv)\}}f_{ji}(\langle V_i,V_j\rangle)\expect_{V_u, V'_v\sim\unif(\dsphere)}\Big[f_{uv}(\langle V_u,V'_v\rangle)\Big|V_v\Big]\Big]\\
        & = 
        \expect_{\{V_i\}_{i\neq u}\sim_{iid} \unif(\dsphere)}\Big[\prod_{(ji)\in E(H)\backslash \{(uv)\}}f_{ji}(\langle V_i,V_j\rangle)\Big]\times 
\expect_{V_u,V_v\sim_{iid} \unif(\dsphere)}\Big[f_{uv}(\langle V_u,V_v\rangle)\Big],
    \end{split}
\end{equation*}
where $V'$ is some independent copy of $V'_v.$
\end{proof}
For $f_{ji}(x) = \indicator[x\ge \tau^p_d]-p,$ this shows that signed expectations of graphs with leaves vanish with respect to $\RGGsphere.$ It turns out that this is not true for Gaussian geometric graphs. Namely, one can show the following statement.

\begin{lemma}[Lower Bound on Signed Expectation of Wedges]
\label{lem:wedgelowerbound}
Consider the wedge on 3 vertices $1,2,3$ with edges $(21), (31).$ Then, if $p\le 0.49,$ 
$$
\E_{\bfG\sim \RGGgauss}\Big[
(G_{21}- p)(G_{31}-p)
\Big]\ge \frac{p^2}{{d}}(\log d)^{-4}.
$$
\end{lemma}

The proof of this fact is simple. It follows by conditioning on the norms $\|Z_i\|_2$ of Gaussian vectors. This reduces the problem to a spherical random geometric graph in which, however, the thresholds are dependent on the values $\|Z_i\|_2$ and, in particular, not equal to $\tau^p_d$ (or $\rho^p_d$). Nevertheless, we can still use the factorization property of the spherical random geometric graphs \cref{prop:leavesinspherical}. We carry out the calculation in \cref{sec:LowerBoundsGaussianFourier}.

The following decomposition will also be useful when studying Gaussian random geometric graphs and spherical against the planted coloring (in the planted coloring setting, note that signed weights are not orthonormal for the null distribution $\PCol(n,q),$ hence we need to consider a different basis and, thus, also graphs with leaves might have non-trivial expectations in this basis).

\begin{definition}[Leaf Decomposition]
\label{def:leafdecomposition}
For a graph $H,$ decompose its edges into two parts -- a tree part $\TP(H)$ and no-leaves part $\NLP(H)$ as follows. Initially, let $H_0 = H.$ While $H_0$ has a leaf $u,$ find its neighbor $v,$ add the edge $(uv)$ to $\TP(H),$ and remove the edge $(uv)$ from $H_0.$ At the end, $\NLP(H)= E(H)\backslash \TP(H).$ 
\end{definition}
\begin{observation}
\label{obs:leafdecomposition}
The following claims hold for the leaf decomposition.
\begin{enumerate}
    \item $\TP(H)$ is a forest and $\NLP(H)$ has minimal degree 2 in each connected component.
    \item If $H$ has no acyclic connected  component,
    $|V(H)|= |V(\NLP(H))|+ |E(\TP(H))|.$
\end{enumerate}
\end{observation}

We obtain a mild and easy improvement of \cref{thm:oeiboundonfourier} which, however, will be useful when discussing the edge-query model for Gaussian random geometric graphs. We prove it in \cref{sec:fourierwithleaveslowerGaussian}.

\begin{proposition}
\label{prop:graphswithleaves}
Suppose that \eqref{eq:assumption} holds and $H$ is connected. Then, there exists some absolute constant $C$ depending only on $\epsilon,\gamma$ in \eqref{eq:assumption} such that for $\bfG\sim \RGGgauss$ and $\bfG\sim\RGGsphere,$
\begin{equation}
    \Big|\E\Big[\prod_{(ji)\in E(H)}(G_{ji} - p)\Big]\Big|\le 
(8p)^{|E(H)|}\times \Big(\frac{C\times |V(H)|\times|E(H)|\times (\log d)^{2}}{\sqrt{d}}\Big)^{|E(\TP(H))| +\OEI(\NLP(H))}.
\end{equation}
\end{proposition}

\begin{remark}
This inequality would follow directly from an improvement of 
\cref{thm:oeiboundonfourier} to $\SOEI(H)$ as one can show that $\SOEI(H)\ge |E(\TP(H))| + \SOEI(\NLP(H)).$ Indeed, this follows from the following observation. Let $\pi'$ be in $\arg \max_{\sigma}\SOEI_\sigma(H)$ over permutations of $[|V(\NLP(H))|].$ Extend $\pi'$ to $\pi$ by assigning the numbers $\{|V(\NLP(H))|+1, |V(\NLP(H))|+2, \ldots,|V(H)|\}$ to the rest of the vertices in decreasing order as in the leaf-processing step of the definition of $\TP(H).$
 \end{remark}
\subsection{Probability}
\label{section:probability}
The bounds in this section are all elementary and aim to rigorously quantify the following heuristic.
Consider $Z_1, Z_2\sim\standardgaussd.$ Then, $\langle Z_1, Z_2\rangle$ is approximately distributed as $\mathcal{N}(0,\frac{1}{d}).$ The rapid decay of Gaussian tails suggests that the density of $\langle Z_1, Z_2\rangle$ around $\rho^p_d$ (recall that $\rho^p_d$ is defined by 
$p=\prob[\langle Z_1, Z_2\rangle\ge \rho^p_d]$) 
is roughly $p\sqrt{d}.$ Hence, for any ``short'' interval $[\rho^p_d + a, \rho^p_d+a+\Delta]$ of length $\Delta$ sufficiently close to $\rho^p_d,$ it must be 
the case that $\prob[\langle Z_1, Z_2\rangle\in [\rho^p_d + a, \rho^p_d+a+\Delta]]\approx p\sqrt{d}\Delta.$ Bounds of this 
form (where $\langle Z_1, Z_2\rangle$ is replaced by $\mathcal{N}(0,\frac{1}{d})$ or $\langle V_1, V_2\rangle$ for $V_1, V_2\sim\unif(\dsphere)$ will be frequently used throughout.

An impatient reader is welcome to read the statements of \cref{fact:GaussianFatcs,cor:innerproductsinintervals,cor:gaussianinintervals} and continue to the much more conceptually interesting \cref{sec:mainclaim}.

\paragraph{1. Gaussian Random Variables.} We use normalized Gaussian random variables $\mathcal{N}(0,\frac{1}{d}).$ The density of $X\sim \mathcal{N}(0,\frac{1}{d})$ is given by $\phi_d(x):=\frac{\sqrt{d}}{\sqrt{2\pi}}\exp(-dx^2/2).$ We record several simple facts about Gaussians, all of which follow immediately from the following elementary proposition.

\begin{proposition}[Folklore]
\label{prop:gaussiantalesfolklore}
If $x>0$ and $Y\sim\standardgauss,$
$
\Big(\frac{1}{x}-\frac{1}{x^3}\Big)\frac{e^{-x^2/2}}{\sqrt{2\pi}}\le \prob[Y\ge x]\le 
\frac{1}{x}\frac{e^{-x^2/2}}{\sqrt{2\pi}}
.
$
    
\end{proposition}
 
\begin{fact}
\label{fact:GaussianFatcs}
For $X\sim \mathcal{N}(0,\frac{1}{d}),$ the following hold:
\begin{enumerate}
    \item $\prob[X\ge t]\le\exp(-t^2d/2)$ for $t\ge\sqrt{1/d}.$
    \item If $\xi^p_d$ is such that $\prob[X\ge \xi_d^p]= p$ for some $p,d$ under \eqref{eq:assumption}, then\linebreak  
    $\frac{0.1\sqrt{\log(1/p)}-10}{\sqrt{d}}\le\xi_d^p\le 
    \frac{\sqrt{2\log(1/p)}}{\sqrt{d}}.$
    \item Under \eqref{eq:assumption},
    $\phi_d(\xi^p_d)\le C'pd\xi^p_d\le C'p\sqrt{\log d}\sqrt{d}$ for some constant $C'.$
    \item Suppose that \eqref{eq:assumption} holds $a\in\mathbb{R}, \Delta\in \mathbb{R}_{\ge 0}$ are such that $|a|,|\Delta|\le (\log d)^{-1}/\sqrt{d}.$ Then, for some absolute constant $C''>0,$
    $$
    {\Delta p}{\sqrt{d}}/(C''\sqrt{\log d})\le
    \prob\Big[X\in [\xi^p_d+ a, \xi^p_d + a + \Delta]\Big]\le 
    {\Delta p}{\sqrt{d}}\times(C''\sqrt{\log d}).
    $$
\end{enumerate}
\end{fact}
\begin{proof}
\textbf{1.} Follows from setting $x = t\sqrt{d}$ in \cref{prop:gaussiantalesfolklore}. \textbf{2.} Again, this holds 
from \cref{prop:gaussiantalesfolklore}. For the upper bound, it is enough to show that 
$\frac{1}{\sqrt{2\log(1/p)}}\frac{e^{-\log(1/p)}}{\sqrt{2\pi}}<p,$ which is trivial. For the lower bound, there is nothing to prove if $0.1\sqrt{\log(1/p)}\le 10.$ Otherwise, we can prove the stronger inequality $\sqrt{\log(1/p)}/\sqrt{d}\le \xi^p_d.$ Indeed,
\begin{align*}
&\Big(\frac{1}{\sqrt{\log(1/p)}}-\frac{1}{\sqrt{\log(1/p)}^3}\Big)e^{-\sqrt{\log(1/p)}^2/2}\ge 
\frac{1}{2\sqrt{\log(1/p)}}
\sqrt{p}\ge p.
\end{align*}
\textbf{3.} This follows from 2 for the following reason. Consider the density at $\xi^p_d+ \frac{1}{\sqrt{d}\sqrt{(\log d)}}.$ Then,
\begin{align*}
&p\ge \int^{\xi^p_d+ \frac{1}{\sqrt{d}\sqrt{(\log d)}}}_{\xi^p_d}\phi_d(t)d t \ge 
\frac{1}{\sqrt{d}\sqrt{(\log d)}}\phi_d(\xi^p_d+ \frac{1}{\sqrt{d}\sqrt{(\log d)}})\\
& \ge \frac{1}{\sqrt{d}\sqrt{(\log d)}}\phi_d(\xi^p_d)\times 
\exp\Big(-d\big(\xi^p_d\frac{1}{\sqrt{d}\sqrt{(\log d)}} + \frac{1}{2d(\log d)}\big)\Big),
\end{align*}
which immediately gives the first inequality. Using the bounds on $\xi^p_d$ from 2., we bound the last expression above and obtain 
$p\ge \frac{1}{\sqrt{d}\sqrt{(\log d)}}\phi_d(\xi^p_d)C''$ which gives the second inequality.
\textbf{4.} Part 1 implies that for some constant $\eta$ depending only on $\gamma,\epsilon$ in \eqref{eq:assumption},\linebreak $\prob[X\in [\xi^p_d, \xi^p_d + \frac{\eta \sqrt{\log d}}{\sqrt{d}}]]\ge p/2.$ On the other hand, whenever $|\xi^p_d - u|\le \frac{(\log d)^{-1}}{\sqrt{d}},$ then\linebreak $\phi_d(\xi^p_d)/\phi_d(u) = \exp(-d((\xi^p_d-u)\xi^p_d - (\xi^p_d-u)^2/2))\in [1/C''',C''']$ for some absolute constant $C'''$ under \eqref{eq:assumption} as  $|\xi^p_d|\times|\xi^p_d - u|= o(1/d),
(\xi^p_d - u)^2 = o(1/d)
$ by 2. Combining these statements, at least a $\Omega(\Delta\sqrt{d}/\sqrt{(\log d)})$ fraction of the $p/2$ mass is in $[\xi^p_d+ a, \xi^p_d + a + \Delta]$ which immiediately gives the lower bound as the density in this interval is $\Omega(\phi_d(\xi^p_d)).$ Similarly, at most a $O(\Delta\sqrt{d}\sqrt{(\log d)})$ fraction of the $p/2$ mass is in $[\xi^p_d+ a, \xi^p_d + a + \Delta]$ which gives the upper bound.
\end{proof}

\paragraph{2. $\chi^2$ Random Variables.} The $\chi^2(d)$ distribution is the distribution of $\sum_{i=1}^d X_i^2,$ where\linebreak $X_1,\ldots, X_i\sim_{iid}\standardgauss.$ We will need the following proposition.

\begin{proposition}[\cite{laurent2000chisquared}]
\label{prop:chisquaredconcnetrations}
If  $Z\sim \chi^2(d),$ then 
$|\prob[|Z/d-1|\ge t]|\le 2\exp( - dt^2/16)$ for all $t \le 1.$
\end{proposition}

\paragraph{3. Thresholds.}  Finally, we will prove bounds on the values $\rho^p_d, \tau^p_d$ from \cref{def:spherical}. 

\begin{proposition}
\label{prop:thresholdsclosetogaussian}
Under \eqref{eq:assumption}, 
$|\tau^p_d - \xi^p_d|\le \frac{(\log d)^2}{d}$ and $|\rho^p_d - \xi^p_d|\le \frac{(\log d)^2}{d}.$     
\end{proposition}
\begin{proof} We begin with $\rho^p_d.$ Let $Z_1,Z_2\sim\standardgaussd.$ By \cref{eq:Bartlettintro}, $\langle Z_1, Z_2\rangle = Z_{11}Z_{21},$ where $dZ^2_{11}\sim\chi^2(d),Z_{21}\sim\mathcal{N}(0,1)$ are independent. Hence, 
\begin{align*}
&p= \prob[\langle Z_1, Z_2\rangle\ge \rho^p_d] = 
\prob[Z_{11}Z_{21}\ge \rho^p_d]\\
&\ge 
\prob\Big[Z_{21}\ge \frac{\rho^p_d}{1-(\log d)/\sqrt{d}}\Big]
-\prob[Z_{11}\le 1-(\log d)/\sqrt{d}]\\
&\ge 
\prob\Big[Z_{21}\ge \frac{\rho^p_d}{1-(\log d)/\sqrt{d}}\Big] - \exp(-\Omega((\log^2d))).
\end{align*}
Hence, $\rho^p_d/(1 - (\log d)/\sqrt{d})\le \xi^p_d +\frac{1}{d}$ by part 4 of \cref{fact:GaussianFatcs} applied to $\xi^p_d+a=\rho^p_d, \xi^p_d+a+\Delta = \rho^p_d/(1 - (\log d)/\sqrt{d}).$ Similarly, 
$\rho^p_d(1 +(\log d)/\sqrt{d})\ge \xi^p_d -\frac{1}{d}.$ The argument for $\tau^p_d$ is the same except that one argues about $\langle Z_1, Z_2\rangle/(\|Z_1\|_2\|Z_2\|_2= Z_{21}/\|Z_2\|_2.$
\end{proof}

Combining with the 4-th fact in \cref{fact:GaussianFatcs}, we obtain:  

\begin{corollary}
\label{cor:gaussianinintervals}
Let $X\sim\mathcal{N}(0,\frac{1}{d}).$ Suppose that \eqref{eq:assumption} holds $a\in\mathbb{R}, \Delta\in \mathbb{R}_{\ge 0}$ are such that $|a|,|\Delta|\le (\log d)^{-1}/\sqrt{d}.$ Then, for some absolute constant $C''>0,$
    $$
    {\Delta p}{\sqrt{d}}/(C''\sqrt{\log d})\le
    \prob\Big[X\in [\tau^p_d+ a, \tau^p_d + a + \Delta]\Big]\le 
    {\Delta p}{\sqrt{d}}\times(C''\sqrt{\log d}),
    $$
    $$
    {\Delta p}{\sqrt{d}}/(C''\sqrt{\log d})\le
    \prob\Big[X\in [\rho^p_d+ a, \rho^p_d + a + \Delta]\Big]\le 
    {\Delta p}{\sqrt{d}}\times(C''\sqrt{\log d}),
    $$    
\end{corollary}

\begin{corollary}
\label{cor:innerproductsinintervals}
Suppose that \eqref{eq:assumption} holds $a\in\mathbb{R}, b\in \mathbb{R}_{\ge 0}$ are such that $|a|,|\Delta|\le (\log d)^{-1}/\sqrt{d}.$ Then, for some absolute constant $C''>0,$
\begin{align*}
    &{\Delta p}{\sqrt{d}}/(C''\sqrt{\log d}) - e^{-C''(\log d)^2}\le 
    \prob_{V_1, V_2\sim_{iid}\unif(\dsphere)}\Big[\langle V_1, V_2\rangle\in [\tau^p_d+ a, \tau^p_d + a + \Delta]\Big]\le\\&\quad\quad\quad\quad\quad\quad\quad\quad\quad\quad\quad\quad\quad\quad\quad\quad\quad\le 
    {\Delta p}{\sqrt{d}}\times(C''\sqrt{\log d}) +e^{-C''(\log d)^2},\\ 
    &{\Delta p}{\sqrt{d}}/(C''\sqrt{\log d}) - e^{-C''(\log d)^2}\le 
    \prob_{Z_1, Z_2\sim_{iid}\standardgaussd}\Big[\langle Z_1, Z_2\rangle\in [\rho^p_d+ a, \rho^p_d + a + \Delta ]\Big]\le\\&\quad\quad\quad\quad\quad\quad\quad\quad\quad\quad\quad\quad\quad\quad\quad\quad\quad\le 
    {\Delta p}{\sqrt{d}}\times(C''\sqrt{\log d}) +e^{-C''(\log d)^2}.
\end{align*}
\end{corollary}
\begin{proof}
The proofs in the Gaussian and spherical settings are essentially the same. In the Gaussian case, we write $\langle Z_1,Z_2\rangle =Z_{11}Z_{21}.$ Then, we condition on $Z_{11}\in 1\pm (\log d )/\sqrt{d}$ as in \cref{prop:thresholdsclosetogaussian} which happens with probability $1-\exp(-\Omega((\log d)^2)).$ Then, we are in the setting of \cref{cor:gaussianinintervals} with a value of $\Delta$ changed by a multiplicative factor of $1\pm O(\log d/\sqrt{d}).$ We apply \cref{cor:gaussianinintervals}.  
\end{proof}

\section{Distinguishing Spherical RGG and \ER}
\label{section:sphericaltoER}
\begin{proof}[Proof of \cref{thm:nonadaptivespherical}]
Suppose that $d\ge (M^{1/2}p)^{3+c}$ for some constant $c>0.$
Let $D= (\log n)^{1.1}$ and $\bfG_0\sim \mathsf{G}(N,p), \bfG_1\sim\RGG(N,\dsphere,p).$ In light of \cref{claim:ldphardagainstER}, we simply need to show that
\begin{equation}
\label{eq:chisquaredlowdeg}
    \sum_{H \subseteq E(\mathcal{M})\; :\; 1\le |E(H)|\le D} \Big(\Phi^p_\RGGsphere(H) \Big)^2 = o_n(1).
\end{equation}
Note that each isomorphic copy of $H$ has the same Fourier coefficient under the $\RGGsphere$ distribution. Hence, if $N(H,\mathcal{M})$ denotes the number of subgraphs of $\mathcal{M},$ this is equivalent to 
\begin{equation}
\label{eq:chisquaresphericaltoER}
    \sum_{H\in\graphs_D} N(H,\mathcal{M})\Big(\Phi^p_\RGGsphere(H) \Big)^2 = o_n(1).
\end{equation}
A beautiful result in \cite{Alon1981OnTN} shows that 
$N(H,\mathcal{M})\le C^{|V(H)|^{1.1}}M^{(|V(H)| + \delta(H))/2}$ for all $H$ and $M$-vertex graphs $\mathcal{M},$ where $C$ is some absolute constant.\footnote{
\cite{Alon1981OnTN} explicitly only shows that 
$N(H,\mathcal{M})\le C_HM^{(|V(H)| + \delta(H))/2}$
where $C_H$ is some constant depending only on $H.$ However, one can easily track in the argument this constant to be $|V(H)|^{O(|V(H)|)}$ which is enough. In fact, Alon shows that this result is tight up to the choice of constant $C_H.$} 

By \cref{prop:leavesinspherical}, it is enough to consider the case when $H$ has no leaves. 

 Now, we will use \cref{thm:oeiboundonfourier,prop:boundsonOEI}. Observe that for $H$ connected, 
 $\OEI(H)\ge \frac{1}{3}(2\lceil\frac{|V(H)|-1}{2}\rceil + \delta(H) + 1)\ge \frac{|V(H)| + \delta(H)}{3}.$ However, as all of $\OEI(H),|V(H)|,\delta(H)$ are obviously additive over disjoint unions, the inequality $\OEI(H)\ge \frac{|V(H)| + \delta(H)}{3}$ holds for disconnected $H$ as well. We obtain:
\begin{equation}
\label{eq:beginningchisqauredspherical}
\begin{split}
    & \sum_{H\in \graphsnoleaves_D} N(H,\mathcal{M})\Big(\Phi^p_\RGGsphere(H) \Big)^2\\ & \le 
    \sum_{H\in \graphsnoleaves_D}
    C^{|V(H)|^{1.1}}M^{(|V(H)| + \delta(H))/2}\times 
    \Big(\Big(8\sqrt{\frac{p}{1-p}})^{|E(H)|}\Big(\frac{(\log n)^4}{\sqrt{d}}\Big)^{\OEI(H)}\Big)^2\\
    & \le
    \sum_{H\in \graphsnoleaves_D}
    C^{|V(H)|^{1.1}}
    128^{|E(H)|}p^{|E(H)|}
    M^{(|V(H)| + \delta(H))/2}
    \times \Big(\frac{(\log n)^8}{d}\Big)^{(|V(H)| + \delta(H))/3}.
\end{split}
\end{equation}

Now, we use the following proposition.

\begin{proposition}
\label{prop:smalldeltaifdegreeatleast2}
Suppose that $H$ is a graph without isolated vertices and leaves. Then,\linebreak 
$|E(H)|\ge |V(H)| + \delta(H).$
\end{proposition}
\begin{proof}
Note that in a graph $H$ without isolated vertices and leaves, each vertex has degree at least $2.$ By \cref{prop:independentmaximizerofdelta}, there exists some independent set $S$ such that $|S| - |N_H(S)| = \delta(H)$ and, hence, $S\cap N_H(S) = \emptyset.$ In particular, this means that among the vertices
$A = S\cup N_{H}(S),$ there are at least $2|S|$ edges with both endpoints in $A$ (as each vertex in $S$ has at least two neighbours and both are in $N_H(S)$ by definition of $N_H(\cdot)$). On the other hand, each vertex in $V(H)\backslash A$ is also adjacent to at least two edges, so there are at least 
$\frac{1}{2}\times 2(|V(H)| - |A|)$ more edges. Altogether, this means that the number of edges in $H$ is at least 
$$
2|S| + (|V(H)| - |A|) = 
2|S| + |V(H)| - |S| - |N_H(S)| = 
|V(H)| + |S| - |N_H(S)| = 
|V(H)|  + \delta(H),
$$
as desired. 
 \end{proof}

Going back to \cref{eq:beginningchisqauredspherical}, we bound 

\begin{equation}
\label{eq:secondeqsphericalchisquared}
    \begin{split}
& \sum_{H\in \graphsnoleaves_D}
    C^{|V(H)|^{1.1}}
    128^{|E(H)|}p^{|E(H)|}
    M^{(|V(H)| + \delta(H))/2}
    \times \Big(\frac{(\log n)^8}{d}\Big)^{(|V(H)| + \delta(H))/3}\\
& \le 
\sum_{H\in \graphsnoleaves_D}
    C^{|V(H)|^{1.1}}
    128^{|E(H)|}p^{|V(H)| + \delta(H)}
    M^{(|V(H)| + \delta(H))/2}
    \times \Big(\frac{(\log n)^8}{d}\Big)^{(|V(H)| + \delta(H))/3}\\
& = 
\sum_{H\in \graphsnoleaves_D}
    C^{|V(H)|^{1.1}}
    128^{|E(H)|}
\Big(
\frac{M^{1/2}p(\log n)^{8}}{d^{1/3}}
\Big)^{(|V(H)| + \delta(H))}\\
    \end{split}
\end{equation}
As $d\ge (M^{1/2}p)^{3+c},$ under \eqref{eq:assumption}, 
$\Big(
\frac{M^{1/2}p(\log n)^{8}}{d^{1/3}}
\Big) \le n^{-\theta}$ for some constant $\theta>0$ (depending only on $\delta, \gamma$ from \eqref{eq:assumption}, and $c$ from \cref{thm:nonadaptivespherical}). Recall that $M = \Theta(n^2).$ Also, $\delta(H)\ge 0$ (by taking $S= V(H)$ in the definition).
Hence, we can bound \cref{eq:secondeqsphericalchisquared} by 
\begin{equation}
\label{eq:thirdeqsphericalchisquared}
    \begin{split}
& 
\sum_{H\in \graphsnoleaves_D}
    C^{|V(H)|^{1.1}}
    128^{|E(H)|}
\Big(
\frac{M^{1/2}p(\log n)^{8}}{d^{1/3}}
\Big)^{(|V(H)| + \delta(H))}\\
& \le 
\sum_{H \in \graphsnoleaves_D}
    C^{|V(H)|^{1.1}}
    128^{|E(H)|}
    n^{-|V(H)|\theta} = 
\sum_{H \in \graphsnoleaves_D}
    C^{|V(H)|^{1.1}}
    128^{|E(H)|}e^{-|V(H)|\theta\log n}.    
    \end{split}
\end{equation}
Note that $|V(H)|\le 2|E(H)|\le 2D = O((\log n)^{1.1}).$ Thus, $|V(H)|^{1.1} = o_n(|V(H)|\log n).$ Similarly, 
$E(H)\le \min(|V(H)|^2, D)\le |V(H)|D^{1/2} = o_n(|V(H)|\log n).$ Hence, for all large enough $n,$ the expression above is bounded by 
$$
\sum_{H \in \graphsnoleaves_D}
e^{-\frac{\theta}{2}|V(H)|\log n}.
$$
The last expression is bounded by $
\sum_{H \in \graphs_D}
e^{-\frac{\theta}{2}|V(H)|\log n}
$ which is enough by \cref{prop:whenchisquaredissmall}.
\end{proof}

\section{Distinguishing Gaussian RGG and \ER}
\label{section:gaussiantoER}

\subsection{Proof of \texorpdfstring{\cref{thm:unmaskedgaussian}}{Unmasked Gaussian}}
\subsubsection{Detection Lower Bound}
\label{sec:proofunmaskedgaussian}
Suppose that $d\ge \max(n^{3/2}p)^{1+c},(n^3p^3)^{1+c}).$ We will prove that no degree $D=(\log n)^{1.1}$ polynomial distinguishes the two graph models. We proceed as in \cref{section:sphericaltoER} with $\mathcal{M}=K_n.$ Using the obvious bound $N(K_n,H)\le n^{|V(H)|}$ (as one can choose each of the $|V(H)|$ labeled vertices of $H$ in at most $n$ ways), we need to show the following analogue of \cref{eq:chisquaresphericaltoER}.
$$
\sum_{H\in\graphs_D}
n^{|V(H)|}\Big(\Phi^p_\RGGgauss(H)\Big)^2 =o_n(1).
$$
Now, consider any graph $H$ on $v$ vertices and $e$ edges. Suppose that its connected components are $H_1,H_2,\ldots,H_r,$ where $H_i$ has $v_i$ vertices and $e_i$ edges. Clearly, $\sum_{i=1}^r e_i = e, \sum_{i =1}^r v_i= v$ and  $\Phi^p_\RGGgauss(H) = \prod_{i = 1}^r \Phi^p_\RGGgauss(H_i).$ Clearly, if any of the graphs $H_i$ is a single edge, $\Phi^p_\RGGgauss(H_i)= 0,$ so we can focus on the case $e_i\ge 2, v_i \ge 3$ $\forall i.$ Now, suppose that $H_1,H_2,\ldots, H_t$ are trees and 
$H_{t+1}, H_{t+2}, \ldots,H_{r}$ are not. Then, using \cref{prop:graphswithleaves},
\begin{enumerate}
    \item For $1\le u \le r,$ 
    \begin{align*}
    & \Big|\Phi^p_\RGGgauss(H_i)\Big|=
    \Big|\Big(\sqrt{\frac{1}{p(1-p)}}\Big)^{e_i}
    \expect[\prod_{(uv)\in E(H_i)}(G_{uv} - p)]\Big|\\
    & \le 
    \Big(\sqrt{\frac{1}{p(1-p)}}\Big)^{e_i}
    \Big(\frac{p(\log n)^4}{\sqrt{d}}\Big)^{e_i}\le
    \Big(
    \frac{2p(\log n)^8}{d}
    \Big)^{v_i/3}.
    \end{align*} We used $e_i =v_i-1\ge 2v_i/3$ twice. 
    \item For $r+1\le u \le t,$
    \begin{align*}
    &\Big|\Phi^p_\RGGgauss(H_i)\Big|=
    \Big|\Big(\sqrt{\frac{1}{p(1-p)}}\Big)^{e_i}
    \expect[\prod_{(uv)\in E(H_i)}(G_{uv} - p)]\Big|\\
    & \le \Big(\sqrt{\frac{1}{p(1-p)}}\Big)^{e_i}
    p^{e_i}\times
    \Big(\frac{(\log n)^4}{\sqrt{d}}\Big)^{\lceil (v_i -1)/2\rceil}\le
    (2\sqrt{p})^{e_i}
    \Big(\frac{(\log n)^4}{\sqrt{d}}\Big)^{v_i/3}\le
    \Big(\frac{8p^3(\log n)^4}{\sqrt{d}}\Big)^{v_i/3}.
    \end{align*}
    This time, we used $v_i \ge e_i$ (as $H_i$ is connected and not a tree) and $\lceil (v_i -1)/2\rceil\ge v_i/3.$
\end{enumerate}
Altogether, this means that 
$$
\Big|\Phi^p_\RGGgauss(H)\Big|\le
\Big(
    \frac{2p(\log n)^8}{d}
    \Big)^{\sum_{i=1}^rv_i/3}
    \times 
    \Big(\frac{8p^{3/2}(\log n)^4}{\sqrt{d}}\Big)^{\sum_{i= r+1}^tv_i/3}.
$$
Hence, as $|V(H)|=\sum_{i= 1}^t v_i,$ 
$$
n^{|V(H)|}\Phi^p_\RGGgauss(H)^2\le
\Big(
    \frac{2p(\log n)^8n^{3/2}}{d}
    \Big)^{\sum_{i=1}^r2v_i/3}
    \times 
    \Big(\frac{8p^3n^3(\log n)^4}{{d}}\Big)^{\sum_{i= r+1}^tv_i/3}.
$$
As $d\ge (n^{3/2}p)^{1+c},d\ge (n^3p^3)^{1+c}$ and \eqref{eq:assumption} holds, this last expression is clearly at most $n^{-\theta |V(H)|}$ for some absolute constant $\theta$ depending only on $\epsilon, \gamma$ from \eqref{eq:assumption} and $c$ from \cref{thm:nonadaptiveGaussian}. The claim follows from \cref{prop:whenchisquaredissmall}.
\subsubsection{Detection Algorithms}
 \paragraph{1. Detecting via signed wedges.} We want to show that if $d=o\Big( (n^{3/2}p)(\log n)^{-4}\Big),p\le0.49$ and \eqref{eq:assumption} holds, 
the signed wedge count statistic distinguishes $\RGGgauss$ and $\ergraph$ with high probability in the sense of \cref{def:successofLDP}. The condition $p<0.49$ is needed as $\RGGspherehalf$ and $\RGGgausshalf$ are the same distribution.

Recall the definition of signed counts \cref{eq:signedcounts}.
By \cref{lem:wedgelowerbound},
$$\E_{\bfG\sim \RGGgauss}[\CS^p_\wedge(\bfG)] = n\binom{n-1}{2}\E[\SW^p_{\wedge(1,2,3)}(\bfG)] \ge 
\frac{n^3p^2}{4d}\times (\log d)^{-4}.
$$
Clearly, $\E_{\bfH\sim \mathsf{\bfG}(n,p)}[\SW^p_\wedge(\bfH)] = 0.$ Now, we compute the variances. Note that, with respect to either distribution, 
$$
\var\Big[\CS^p_\wedge\Big] = 
\sum_{(i,j,k)}\var[\SW^p_{\wedge(i,j,k)}] + 
\sum_{\wedge(i,j,k)\neq \wedge(i',j',k')}
\cov[\SW^p_{\wedge(i,j,k)}, \SW^p_{\wedge(i',j',k')}].
$$
Importantly, observe that for any two labelled subgraphs $H_1,H_2,$ 
one has $\cov[\SW^p_{H_1}, \SW^p_{H_2}] = 
\expect[\SW^p_{H_1}\SW^p_{H_2}] -
\expect[\SW^p_{H_1}]\expect[\SW^p_{H_2}].
$ However,
\begin{equation}
\label{eq:covarianceexpansion}
\begin{split}
&\signedweight^p_{H_1}(G)\signedweight^p_{H_2}(G)\\
& = \prod_{e_1\in E(H_1)}(G_{e_1}-p)\times\prod_{e_2\in E(H_2)}(G_{e_2}-p)\\
& = \prod_{e\in E(H_1)\triangle E(H_2)}
(G_e-p)\times \prod_{e\in E(H_1)\cap E(H_2)}(G_e-p)^2\\
& = \prod_{e\in E(H_1)\triangle E(H_2)}
(G_e-p)\times \prod_{e\in E(H_1)\cap E(H_2)}((G_e-p)(1-2p)+(p-p^2))\\
& =\sum_{A\subseteq E(H_1)\cap E(H_2)}
\prod_{e\in (E(H_1)\triangle E(H_2))\cup A}(G_e-p)(1-2p)^{|A|}(p-p^2)^{|E(H_1)\cap E(H_2)| - |A|}.
\end{split}    
\end{equation}
In particular, the covariance is also easily expressed as a linear combination of signed weights (which should be no surprise as any polynomial can be expanded in the Fourier basis). The above expression immediately shows that unless $H_1 = H_2,$ the covariance with respect to $\ergraph$ is 0 as an edge in $E(H_1)\triangle E(H_2)$ has expectation $0.$ When $H_1 = H_2 = \wedge(i,j,k),$ the covariance is bounded by $p^{|E(H)|} = p^2.$ Hence, 
$\var_{\bfG\sim\ergraph}\Big[\CS^p_\wedge(\bfG)\Big] = O(n^3p^2)$
For the geometric distribution, we can use \cref{prop:graphswithleaves} to bound the covariance. We carry out this trivial calculation in \cref{appendix:varianceunmaskedwedge} and show that $\var_{\bfG\sim\RGG}\Big[\CS^p_\wedge(\bfG)\Big] = o\Big( (\log d)^8\times \Big(
n^3p^2 + n^5p^4/d^2 + n^4p^3/d + n^4p^4/d^2
\Big)\Big),$ which is enough to conclude by \cref{def:successofLDP}.

\paragraph{2. Detecting via signed triangles.} We want to show that if $d=o\Big( (n^3p^3)(\log n)^{-1}\Big)$ for some absolute constant $c$ and \eqref{eq:assumption} holds, 
the signed triangle count statistic distinguishes $\RGGgauss$ and $\ergraph.$ The proof is analogous to the ones given in \cite{Bubeck14RGG,Liu2022STOC}, but we give it for completeness. We will use the following claim, the proof of which is presented in \cref{sec:lowertriangles}. 

\begin{lemma} 
    \label{thm:signedtrianglegauss}
    $\displaystyle
\expect_{\bfG\sim\RGGgauss}\Big[\SW^{p}_\triangle(\bfG)\Big] \ge \frac{Cp^3(\log d)^{-1}}{\sqrt{d}}
    $ for some absolute constant $C.$
\end{lemma}

As in the proof for wedges, we know that $\expect_{\bfG\sim\RGGgauss}[\CS^p_{\triangle}(\bfG)]\ge \frac{Cn^3p^3(\log d)^{-1}}{6\sqrt{d}},$\linebreak $\expect_{\bfG\sim\ergraph}[\CS^p_{\triangle}(\bfG)]= 0,\var_{\bfG\sim\ergraph}[\CS^p_{\triangle}(\bfG)]= \Theta(n^3p^3).$ In \cref{appendix:varianceunmaskedtriangle} we show that \linebreak
$\var_{\bfG\sim\RGG}[\CS^p_{\triangle}(\bfG)]= O\Big(
n^5p^6(\log d)^4/d + 
n^4p^5(\log d)^4/d + 
n^3p^3
\Big),$ which is enough. 

\subsection{Proof of \texorpdfstring{\cref{thm:nonadaptiveGaussian}}{Non-Adaptive Gaussian}}
\subsubsection{Detection Lower Bounds}
\label{sec:proofmaskedgaussian}
Suppose that $d\ge \max\Big((Mp)^{1+c}, (M^{3/2}p^3)^{1+c}\Big).$ We want to show that for all $\mathcal{M}$ on $M$ edges and $N$ vertices, the distributions 
$\bfG_0\sim \mathsf{G}(N,p)\odot \mathcal{M},\bfG_1\sim\RGG(N,\standardgaussd,p)$ are  indistinguishable with respect to low degree polynomials. As in \cref{section:sphericaltoER},
we simply need to show that 
\begin{align}
\label{eq:firstlowerboundgaussiannonadaptive}
\sum_{H\in\graphs_D}
C^{|V(H)|^{1.1}}\times
M^{(|V(H)|+\delta(H))/2}\times \Big(\Phi^p_{\RGGgauss}(H)\Big)^2=o_n(1).
\end{align}
Decompose $H$ into $\TP(H)$ and $\NLP(H)$ as in \cref{def:leafdecomposition}. Let $\TP(H)$ be composed of trees $T_1, T_2, \ldots, T_k.$ We will use the following claim.

\begin{proposition}
Let $H$ be a graph. Suppose that there exist two graphs $H_1,H_2$ such that \linebreak 
$|V(H_1)\cap V(H_2)|\le 1$ and 
$E(H)= E(H_1)\cup E(H_2).$ Then,
$$
|V(H)| + \delta(H)\le 
(|V(H_1)|+ \delta(H_1))+ 
(|V(H_2)|+ \delta(H_2)).
$$
\end{proposition}\begin{proof} If $|V(H_1)\cap V(H_2)|= 0,$ this claim is trivial as  both $|V(H)|$ and $\delta(H)$ are additive over disjoint unions. Suppose that 
$V(H_1)\cap V(H_2)= \{a\}.$ Let $S \in \arg \max_{S'\subseteq V(H)}|S'| - |N_H(S')|.$ Let $S_1 = S\cap V(H_1),S_2= S\cap V(H_2).$
It is enough to show that 
\begin{align*}
& |V(H)|+ |S| - |N_H(S)|\le 
|V(H_1)|+ |S_1| - |N_{H_1}(S_1)| + 
|V(H_2)| + |S_2| -|N_{H_2}(S_2)|\; \Longleftrightarrow\\
& |S| + |N_{H_1}(S_1)|+ |N_{H_2}(S_2)| \le 
|S_1|+ |S_2| + |N_H(S_1\cup S_2)|+1.
\end{align*}
Now, observe that $|S|\le|S_1|+|S_2|$ trivially. However, $|N_{H_1}(S_1)|+ |N_{H_2}(S_2)| \le |N_H(S_1\cup S_2)|+ 1.$ Indeed, this holds as $N_{H_1}(S_1)\cup N_{H_2}(S_2)\subseteq N_H(S_1\cup S_2)$ and $N_{H_1}(S_1)\cap N_{H_2}(S_2)\subseteq\{a\}.$ 
\end{proof}

As each tree $T_i$ has at most one common vertex with $\NLP(H)$ and none with other trees, we apply the above lemma inductively and obtain
$$
|V(H)|+\delta(H)\le 
\sum_{i = 1}^k (|V(T_i)|+ \delta(T_i)) +
|V(\NLP(H))|+ \delta(\NLP(H)).
$$

Using \cref{prop:graphswithleaves}, we bound \cref{eq:firstlowerboundgaussiannonadaptive} as follows.
\begin{equation*}
\begin{split}
& \sum_{H\in\graphs_D}
C^{|V(H)|^{1.1}}\times
M^{(|V(H)|+\delta(H))/2}\times p^{|E(H)|}\Big(\frac{(\log n)^{10}}{\sqrt{d}}\Big)^{2|E(\TP(H))|+ 2\OEI(\NLP(H))}\\
& \le \sum_{H\in\graphs_D}
C^{|V(H)|^{1.1}}\times
M^{\sum_{i= 1}^k(|V(T_i)| +\delta(T_i))/2 + (|V(\NLP(H))|+ \delta(\NLP(H)))/2}\\
&\quad\quad\quad\quad\quad\times 
p^{\sum_{i = 1}^t |E(T_i)| + |E(\NLP(H))|}
\Big(
\frac{(\log n)^{10}}{\sqrt{d}}
\Big)^{2\sum_{i = 1}^t |E(T_i)| + 2\OEI(\NLP(H))}.
\end{split}
\end{equation*}

Now, we use the following simple combinatorial claim.
\begin{claim}
\label{claim:deltaoftrees}
For any non-empty tree $T,$
$\delta(T)\le|V(T)| - 2.$
\end{claim}
\begin{proof}
Suppose, for the sake of contradiction, that $\delta(T)\ge |V(T)|-1.$ As in \cref{prop:independentmaximizerofdelta}, 
there exists some independent set $S$ such that 
$\delta(T) = |S| - |N_T(S)|.$ Hence, $|S|\ge |V(T)|-1.$ As $T$ is connected, $|N_T(S)|\ge 1.$ As $S$ is independent, 
$S\subseteq V(T)\backslash N_T(S),$ so $|S|\le|V(T)|-1.$ Altogether, $\delta(T) = |S| - |N_T(S)|\le |V(T)|-2,$ contradiction.
\end{proof}

We also use $V(\NLP(H))+ \delta(\NLP(H))\le|E(H)|$ from \cref{prop:smalldeltaifdegreeatleast2} as $\NLP(H)$ has minimal degree $2$ and $\OEI(H)\ge \frac{|V(\NLP(H))|+ \delta(H)}{3}$ from \cref{prop:boundsonOEI}.

Combining these, we bound the last expression by
\begin{align*}
&\sum_{H\in\graphs_D}
C^{|V(H)|^{1.1}}\times
M^{\sum_{i= 1}^k(|V(T_i)|-1) + (|V(\NLP(H))|+ \delta(\NLP(H)))/2}\times \\
&\quad\quad\quad\quad\quad p^{\sum_{i = 1}^t |E(T_i)| + |E(H)|}
\Big(
\frac{(\log n)^{10}}{\sqrt{d}}
\Big)^{2\sum_{i = 1}^t |E(T_i)| + 2\frac{|V(\NLP(H))|+\delta(H)}{3}}\\
&\le \sum_{H\in\graphs_D}
C^{|V(H)|^{1.1}}\times
M^{\sum_{i= 1}^k|E(T_i)| + (|V(\NLP(H))|+ \delta(\NLP(H)))/2}\times \\
&\quad\quad\quad\quad\quad p^{\sum_{i = 1}^t |E(T_i)| + (|V(\NLP(H))|+ \delta(\NLP(H)))}
\Big(
\frac{(\log n)^{10}}{\sqrt{d}}
\Big)^{2\sum_{i = 1}^t |E(T_i)| + 2\frac{|V(\NLP(H))|+\delta(H)}{3}}\\
& \le 
\sum_{H\in\graphs_D}
C^{|V(H)|^{1.1}}\times\prod_{i= 1}^t
\Big(\frac{Mp(\log n)^{20}}{d}\Big)^{|E(T_i)|}\times 
\Big(\frac{M^{3/2}p^3(\log n)^{20}}{d}\Big)^{(|V(\NLP(H))|+ \delta(H))/3}.
\end{align*}
Now, whenever $d\ge \max\Big((Mp)^{1+c}, (M^{3/2}p^3)^{1+c}\Big),$ this last expression is bounded by $n^{-\theta |V(H)|}$ for some absolute constant $\theta$ under 
$\eqref{eq:assumption}.$
As $|V(H)|= O(D) = O((\log n)^{1.1}),$ the entire expression is bounded by 
$$
\sum_{H\in\graphs_D}
C^{|V(H)|^{1.1}}
n^{-\theta |V(H)|} \le 
\sum_{H\in\graphs_D}
n^{-\theta |V(H)|/2} = o_n(1),
$$
the last following from \cref{prop:whenchisquaredissmall}.

\subsubsection{Detection Algorithm}
To detect with sample complexity $M= d(\log d)^{17}/p,$ choose a mask $\mathcal{M}$ which is the union of $A=(\log n)^{17}$ disjoint stars on $N = d$ leaves. Then,
$M=AN= (Np)(\log d)^{17}$
Observe that all subgraphs are stars and the number of $\Star_k$ subgraphs is $A\times \binom{N}{k}=\Theta(AN^k)$ for $k$ constant.

By \cref{lem:wedgelowerbound}, the expect signed count of wedges is
$$\expect_{\bfG\sim\RGGgauss}[\CS_\wedge]= \Theta\Big(AN^2\expect_{\bfG\sim\RGGgauss}[\SW_{\wedge(1,2,3)(\bfG)}]\Big) = 
\Omega(AN^2p^2(\log d)^{-4}/d).
$$
The expectation with respect to $\ergraph$ is clearly 0 and the 
variance $\Theta(AN^2p^2).$ As in the unmasked case, we analyze the variance with respect to $\RGGgauss$ in \cref{appendix:variancemaskedwedge} based on the overlap patterns and show that it is of order 
$O\Big(
AN^4p^4(\log d)^8/d^2+ 
AN^3p^3(\log n)^4/d+ 
AN^2p^2\Big)$ which is enough for the condition in \cref{def:successofLDP} to hold.

\section{RGG versus Planted Coloring}
\label{section:RGGtoPC}
Our first step is to choose $q$ appropriately. We do this in such a way that the ``most informative'' triangle Fourier coefficients of $\RGG(n,\dsphere,1/2)$ and $\PCol(n,q)$ (nearly) match.

After that, we follow nearly explicitly the method of \cite{kothari2023planted} for proving low-degree indistinguishability to $\PCol(n,q).$ It is based on a clever analysis of the bound in \cref{def:ldphardness} exploiting the ambient \ERspace structure of $\PCol(n,q).$ As a large part of the analysis is nearly identical to \cite{kothari2023planted}, we delay many of the proofs to \cref{appendix:rggtopccalc}.

\subsection{Choice of \texorpdfstring{$q$}{q}}
From \cref{thm:oeiboundonfourier} and
\cite[Lemma 1]{Bubeck14RGG}, it immediately follows that 
$\expect_{\bfG\sim\RGG(n,\dsphere,1/2)}[\SW_{\triangle}(\bfG)]\in \Big[\frac{(\log d)^{-1}}{\sqrt{d}},\frac{\log d}{\sqrt{d}}\Big].$ We want to find some $q$ which leads to the same signed expectation of a triangle in $\PCol(n,q).$ Computing signed expectations for $\PCol(n,q)$ is very easy with the same strategy for computing Fourier coefficients as in the planted clique case described in \cref{sec:challenges}.

\begin{proposition}
\label{prop:pccolors}
Consider some graph $H$ and $\bfG\sim \PCol(n,q).$ Then,
\begin{enumerate}
    \item If $H$ has a leaf, $\expect[\SW^{1/2}_H(\bfG)]= 0.$
    \item $\expect[\SW^{1/2}_{\triangle}(\bfG)] = \frac{1}{8q^2} + \frac{3}{8(q-1)q^2} - 
    \frac{(q-2)}{8q^2(q-1)^2}
    $
    \item $|\expect[\SW^{1/2}_{H}]|\le \Big(\frac{|V(H)|}{q}\Big)^{|V(H)|-|\CC(H)|}.$
\end{enumerate}
\end{proposition}
We prove this proposition in \cref{appendix:signedweightsinpcol} and now continue to the choice of $q.$
We define $q$ by rounding some real $q_1$ which exactly matches the signed triangle expectation. 
\begin{equation}
    \begin{split}
    & \frac{1}{8q_1^2} + \frac{3}{8(q_1-1)q_1^2} - 
    \frac{(q_1-2)}{8_1q^2(q_1-1)^2}\coloneqq  \expect_{\bfG\sim\RGG(n,\dsphere,1/2)}[\SW_{\triangle}(\bfG)]\in \Big[\frac{(\log d)^{-1}}{\sqrt{d}},\frac{\log d}{\sqrt{d}}\Big].\\
    & q\coloneqq \lceil q_1\rceil.
    \end{split}
\end{equation}
With this, one immediately concludes that 
\begin{equation}
    \label{eq:choiceofq}
\begin{split}
    &q\in \Big[{d^{1/4}}/(\log d),{d^{1/4}}(\log d)\Big],\\
        &\Big|
    \expect_{\bfG\sim \PCol(n,q)}[\SW^{1/2}_{\triangle}(\bfG)]- 
    \expect_{\bfG\sim \RGGsphere}[\SW^{1/2}_{\triangle}(\bfG)]
    \Big|\le \frac{1}{q^3}\le  \frac{(\log d)^3}{d^{3/4}}.
\end{split}
\end{equation}
Similarly, for any connected graph $H$ on at most $D = (\log n)^{1.1}$ edges and at least 4 vertices, one can observe that 
\begin{equation*}
\label{eq:differenceofFourierRGGandPCol}
    \begin{split}
    &\Big|
    \expect_{\bfG\sim \PCol(n,q)}[\SW^{1/2}_{H}(\bfG)]- 
    \expect_{\bfG\sim \RGGsphere}[\SW^{1/2}_{H}(\bfG)]
    \Big|\\
    &
    \le \Big|
    \expect_{\bfG\sim \PCol(n,q)}[\SW^{1/2}_{H}(\bfG)]\Big|+ 
    \Big|\expect_{\bfG\sim \RGGsphere}[\SW^{1/2}_{H}(\bfG)]
    \Big|\\
    & \le \Big(\frac{(\log n)^4}{\sqrt{d}}\Big)^{\lceil (|V(H)|-1)/2\rceil}+ 
    \Big(\frac{(\log n)^2}{q}\Big)^{(|V(H)|-1)} \le 
    \Big(\frac{(\log n)^8}{q}\Big)^{(|V(H)|-1)}.
    \end{split}
\end{equation*}
Combining with \cref{eq:choiceofq} and using the fact that when $H$ has a leaf, both signed weights equal 0, we reach the following conclusion.

\begin{proposition}
\label{prop:differenceoffourier}
For any graph $H$ on at most $D = (\log n)^{1.1}$ vertices, 
$$
\Big|
    \expect_{\bfG\sim \PCol(n,q)}[\SW^{1/2}_{H}(\bfG)]- 
    \expect_{\bfG\sim \RGGsphere}[\SW^{1/2}_{H}(\bfG)]
    \Big|\le 
D\Big(\frac{(\log n)^8}{q}\Big)^{\max(3,|V(H)|-|\CC(H)|)}. 
$$
\end{proposition}
\begin{proof} If $H$ is connected, this follows immediately from \cref{eq:choiceofq,eq:differenceofFourierRGGandPCol}. Otherwise, assume that  $H$ has $i$ connected components $H_1,H_2,\ldots,H_i,$ where necessarily $i\le D.$ If one of them is on less than 3 vertices, $H$ has a leaf and, hence, the entire expression is equal to zero. Thus, we assume that each $H_i$ is on at least 3 vertices. Now, we use triangle inequality and multiplicativity of $\SW_H$ under disjoint connected components as follows.
\begin{align*}
    &\Big|
    \expect_{\bfG\sim \PCol(n,q)}[\SW^{1/2}_{H}(\bfG)]- 
    \expect_{\bfG\sim \RGGsphere}[\SW^{1/2}_{H}(\bfG)]
    \Big|\\
    &\le 
    \prod_{H_i \in \CC(H)}\Big(
    \Big|
    \expect_{\bfG\sim \PCol(n,q)}[\SW^{1/2}_{H_i}(\bfG)]- 
    \expect_{\bfG\sim \RGGsphere}[\SW^{1/2}_{H_i}(\bfG)]
    \Big|\times\\
    & \quad\quad\quad\quad\quad\quad\quad\times 
    \Big|\prod_{j =1}^{i-1}\expect_{\bfG\sim \PCol(n,q)}[\SW^{1/2}_{H_j}(\bfG)]\prod_{j = i+1}^{|\CC(H)|}\expect_{\bfG\sim \RGGsphere}[\SW^{1/2}_{H_j}(\bfG)]\Big|\Big).
\end{align*}
The conclusion follows.
\end{proof}

\subsection{Low-Degree Indistinguishability With Respect To Planted Coloring}
\label{sec:chisquarePCol}
Denote $\psi_q\coloneqq 1/2-1/(2(q-1)).$ 
Observe that the planted coloring model can be generated in the following way. One first samples $\bfH\sim\mathsf{G}(n,\psi_q).$ Then, one samples $\bfX$ by assigning a uniform label over $[q]$ to each vertex and drawing an edge between $i,j$ if and only if $i$ and $j$ have the same label. Denote by $\Label(n,q)$ the distribution of $\bfX.$ Then, the union of $\bfX$ and $\bfH$ given by $\bfX\vee\bfH$ is distributed as $\PCol(n,q).$ This representation explicitly captures the ambient \ERspace structure. Using it, the authors of \cite{kothari2023planted} derive an analogue of \cref{claim:ldphardagainstER} for low-degree indistinguishability against a planted coloring model. In their case, the planted coloring has density $1/2$ rather than $\psi_q$ between vertices with distinct labels. Hence, we state and prove in \cref{sec:proofofldptoPCol} the variant for $\psi_q$ explicitly, even though the arguments are straightforward modifications of those in \cite{kothari2023planted}.

\begin{proposition}[Low-Degree Indistinguishability from $\PCol(n,q)$]
\label{prop:chisquaredtocoloring}
Consider $\RGGsphere$ and the $\psi_q= (\frac{1}{2}- \frac{1}{2(q-1)})$-based Fourier coefficients given by 
$$
\Phi^{\psi_q}_\RGGspherehalf(H)\coloneqq  
\sqrt{\psi_q(1-\psi_q)}^{-|E(H)|}\times
\expect_{\bfG\sim\RGGspherehalf}[\SW^{\psi_q}_H(\bfG)].
$$
For $K,H$ such that $E(K)\subseteq E(H)\subseteq E(K_n)$ and $0\le |E(H)|\le D,$ define $$M_{K,H}\coloneqq \prob[\bfX|_{E(H)} = E(H)\backslash E(K)]
\sqrt{(1 - \psi_q)/\psi_q}^{E(H)\backslash E(K)},$$ where $\bfX|_{E(H)}$ denotes the induced subgraph of $\bfX$ indexed by edges $E(H).$
Also, recursively define $w_H$ for $E(H)\subseteq E(K_n)$ and $0\le |E(H)|\le D,$ where $w_\emptyset = 1$ and
\begin{align}
\label{eq:wrecurrence}
w_H =
\frac{1}{M_{H,H}}
\Big(
\Phi^{\psi_q}_\RGGspherehalf(H)-
\sum_{K\subsetneq H}w_K
M_{K,H}
\Big).
\end{align}
We define $w_H$ to equal $+\infty$ if division by zero occurs in $1/M_{H,H}.$ Then,
$$
\sup_{f\; : \; \Deg(f)\le D}
\frac{\expect_{\bfG\sim\RGGsphere}[f(\bfG)]}{\expect_{\bfG\sim \PCol(n,q)}[f(\bfG)^2]} -1 \le \sum_{H\subseteq E(K_n)\; : \; 1\le |E(H)|\le D}w_H^2.
$$
\end{proposition}
Thus, all we need to do is show that $\sum_{H\subseteq E(K_n)\; : \; 1\le |E(H)|\le D}w_H^2 = o_n(1).$
\subsection{Bounds on The Weights \texorpdfstring{$w_H$}{w H}}
As in \cite{kothari2023planted}, we will instead work with the quantities $\widehat{w}_H$ defined by 
$\widehat{w}_H\coloneqq  M_{H,H}w_H.$ First, one can observe that the recurrence \cref{eq:wrecurrence} turns into 
$$
\widehat{w}_H = 
\Phi^{\psi_q}_\RGGspherehalf(H)-
\sum_{K\subsetneq H}w_K
M_{K,H} = 
\Phi^{\psi_q}_\RGGspherehalf(H)-
M_{\emptyset, H} - 
\sum_{\emptyset \subsetneq K\subsetneq H}\widehat{w}_K
\frac{M_{K,H}}{M_{H,H}}.
$$
Note, however, that 
\begin{equation}
    \begin{split}
& M_{\emptyset, H} = 
\prob[\bfX|_{E(H)} = E(H)]
\sqrt{(1 - \psi_q)/\psi_q}^{E(H)}\\
& = (\psi_q)^{|E(H)|}\sqrt{(1 - \psi_q)/\psi_q}^{E(H)} = \sqrt{(1 - \psi_q)/\psi_q}^{-E(H)}
(1 - \psi_q)^{E(H)}\\
& = \sqrt{(1 - \psi_q)\psi_q}^{-E(H)}\sum_{K\subseteq H}\Big(
\expect_{\bfH\sim \mathsf{G}(n,\psi_q), \bfX\sim \Label(n,q)}\Big[\prod_{(ji)\in E(H)}((\bfX\vee\bfH)_{(ji)} - \psi_q)\Big|\bfX|_{E(H)} = E(H)\backslash E(K)\Big]\times \\
&\qquad\qquad\qquad\qquad\qquad\qquad\qquad\qquad\qquad\qquad\qquad\qquad\qquad\prob[\bfX|_{E(H)} = E(H)\backslash E(K)\Big]\Big)\\
& = \sqrt{(1 - \psi_q)\psi_q}^{-E(H)}
\expect\Big[\prod_{(ji)\in E(H)}((\bfX\vee\bfH)_{(ji)} - \psi_q)\Big] = \Phi^{\psi_q}_{\PCol(n,q)}(H).
    \end{split}
\end{equation}
Indeed, this holds as $\expect_{\bfH\sim \mathsf{G}(n,\psi_q), \bfX\sim \Label(n,q)}[\prod_{(ji)\in E(H)}((\bfX\vee\bfH)_{(ji)} - \psi_q)|\bfX|_{E(H)} = E(H)\backslash E(K)]$ vanishes if $E(K)\neq E(H)$ and whenever $E(H)= E(K),$ equals $(1-\psi_q)^{|E(H)|}.$ Hence, 
\begin{equation}
    \label{eq:whatrecursion}
\widehat{w}_H = 
(\Phi^{\psi_q}_\RGGspherehalf(H)-
\Phi^{\psi_q}_{\PCol(n,q)}(H))-
\sum_{\emptyset \subsetneq K\subsetneq H}\widehat{w}_K
\frac{M_{K,H}}{M_{K,K}}.
\end{equation}
Next, we bound the values of $M_{H,H}.$

\begin{proposition} $M_{H,H}\ge 1 - \frac{|V(H)|^2}{q}.$
\end{proposition}
\begin{proof}
Observe that 
$
M_{H,H} = \prob[\bfX|_{E(H)} = \emptyset].
$
by definition. This happens, in particular, if all vertices of $H$ receive different labels from $[q]$ when $\bfX\sim\Label(n,q).$ This clearly happens with probability 
\begin{equation*}
\frac{q-1}{q}
\frac{q-2}{q}\cdots
\frac{q-|V(H)|}{q}\ge 
(1 - |V(H)|/q)^{|V(H)|}\ge 
1- |V(H)|^2/q.\qedhere
\end{equation*}
\end{proof}
In light of \cref{eq:whatrecursion}, it is now useful to bound the differences of Fourier coefficients. Using $\psi_q = 1/2 + o(1)$ and triangle inequality, we obtain the following bound.

\begin{proposition}
\label{prop:differenceinshifteedfourier}
    For any connected graph $H$ on at most $D$ vertices, 
    $$
    \Big|
\Phi^{\psi_q}_\RGGspherehalf(H)-
\Phi^{\psi_q}_{\PCol(n,q)}(H)
        \Big|\le 
        8^{|E(H)|}
\Big(\frac{(\log n)^{9}}{q}\Big)^{\max(3,|V(H)|-1)}
       8^{|E(H)|}
\Big(\frac{(\log n)^{9}}{q}\Big)^{3|V(H)|/4}.
    $$
\end{proposition}
\begin{proof} We start by writing out the quantity to be bounded,
\begin{equation}
    \begin{split}
        &\Big|
\Phi^{\psi_q}_\RGGspherehalf(H)-
\Phi^{\psi_q}_{\PCol(n,q)}(H)
        \Big|\\
        & = \sqrt{(1-\psi_q)\psi_q}^{-|E(H)|}
\Big|
\expect_{\bfG\sim\RGGspherehalf}[\SW^{\psi_q}_H(\bfG)]-
\expect_{\bfG\sim\PCol(n,q)}[\SW^{\psi_q}_H(\bfG)]
\Big|\\
& 
\le 4^{|E(H)|}
\Big|
\expect_{\bfG\sim\RGGspherehalf}\Big[\prod_{(ji)\in E(H)}\Big(\bfG_{ji} - \frac{1}{2}- \frac{1}{2(q-1)}\Big)\Big]\\
&\quad\quad\quad\quad\quad\quad\quad-
\expect_{\bfG\sim\PCol(n,q)}\Big[\prod_{(ji)\in E(H)}\Big(\bfG_{ji} - \frac{1}{2}- \frac{1}{2(q-1)}\Big)\Big]
\Big|\\
& 
\le 4^{|E(H)|}
\sum_{K\subseteq E(H)}
\frac{1}{(2(q-1))^{|E(K)|}}
\Big|
\expect_{\bfG\sim\RGGspherehalf}\Big[\prod_{(ji)\in E(H)\backslash E(K)}\Big(\bfG_{ji} - \frac{1}{2}\Big)\Big]\\
&\quad\quad\quad\quad\quad\quad\quad-
\expect_{\bfG\sim\PCol(n,q)}\Big[\prod_{(ji)\in E(H)\backslash E(K)}\Big(\bfG_{ji} - \frac{1}{2}\Big)\Big]
\Big|\\
& = 
4^{|E(H)|}
\sum_{K\subseteq E(H)}
\frac{1}{(2(q-1))^{|E(K)|}}
\Big|
\expect_{\bfG\sim\RGGspherehalf}\Big[\SW^{1/2}_{H\backslash E(K)}(\bfG)\Big]-
\expect_{\bfG\sim\PCol(n,q)}\Big[\SW^{1/2}_{H\backslash E(K)}(\bfG)\Big]
\Big|\\
&\le 
8^{|E(H)|}
\max_{K\subseteq E(H)}
\frac{1}{(2(q-1))^{|E(K)|}}
\Big|
\expect_{\bfG\sim\RGGspherehalf}\Big[\SW^{1/2}_{H\backslash E(K)}(\bfG)\Big]-
\expect_{\bfG\sim\PCol(n,q)}\Big[\SW^{1/2}_{H\backslash E(K)}(\bfG)\Big]
\Big|.\\
    \end{split}
\end{equation}
Now, we analyze the maximum above. Suppose that $H\backslash E(K)$ has a leaf. Then, both signed expectations equal 0. Hence, we only consider $K$ such that $H\backslash E(K)$ has no leaves. Now, recall the leaf decomposition \cref{def:leafdecomposition}. Clearly, it must be the case that all edges in $\TP(H)$ should be in $K$ so that $H\backslash E(K)$ has no leaves. Note, however, that if $H = \TP(H),$ meaning that $H$ is a tree, $\expect_{\bfG\sim\RGGspherehalf}\Big[\SW^{1/2}_{H\backslash E(K)}(\bfG)\Big]-
\expect_{\bfG\sim\PCol(n,q)}\Big[\SW^{1/2}_{H\backslash E(K)}(\bfG)\Big] = 0$ holds for all $K$ and, thus, the expression equals zero. Hence, we can assume that $H$ is not a tree. As $H$ is connected, $\NLP(H)\neq \emptyset.$ By \cref{obs:leafdecomposition},
 $|E(\TP(H))|= |V(H)| - |V(\NLP(H))|.$
Denote $A = \NLP(H)\cap K.$ 
\begin{align*}
    & 8^{|E(H)|}
\max_{K\subseteq E(H)}
\frac{1}{(2(q-1))^{|E(K)|}}
\Big|
\expect_{\bfG\sim\RGGspherehalf}\Big[\SW^{1/2}_{H\backslash E(K)}(\bfG)\Big]-
\expect_{\bfG\sim\PCol(n,q)}\Big[\SW^{1/2}_{H\backslash E(K)}(\bfG)\Big]\Big|\\
    & \le 8^{|E(H)|}
\max_{K\subseteq E(H)}
\frac{1}{q^{|E(A)| + |E(\TP(H))|}}
\Big|
\expect_{\bfG\sim\RGGspherehalf}\Big[\SW^{1/2}_{\NLP(H)\backslash A}(\bfG)\Big]-
\expect_{\bfG\sim\PCol(n,q)}\Big[\SW^{1/2}_{\NLP(H)\backslash A}(\bfG)\Big]
\Big|\\
& \le 8^{|E(H)|}
\max_{A\subseteq \NLP(H)}
\frac{1}{q^{|E(A)| + |V(H)|-|V(\NLP(H))|}}\times D\times\Big(\frac{(\log n)^8}{q}\Big)^{\max(3,|V(\NLP(H)\backslash A)| - |\CC(\NLP(H)\backslash A)|)}
\end{align*}
by \cref{prop:differenceoffourier}.

Thus, it is enough to show that
$$
|E(A)| + |V(H)|-|V(\NLP(H))| + 
\max(3,|V(\NLP(H)\backslash A)| - |\CC(\NLP(H)\backslash A)|)\ge 
\max(3,|V(H)|-1).
$$
If $\max(3,|V(H)|-1) = 3,$ this means that $H = \triangle$ as otherwise $\NLP(H) = \emptyset.$ In that case, $\NLP(H) =\triangle.$ Thus, the left-hand side is 
$$
|E(A)|  + 
\max(3,|V(\NLP(H)\backslash A)| - |\CC(\NLP(H)\backslash A)|)\ge 3,
$$
as desired.
Otherwise, it is enough to show that
\begin{align*}
& |E(A)| + |V(H)|-|V(\NLP(H))| + 
|V(\NLP(H)\backslash A)| - |\CC(\NLP(H)\backslash A)|) \ge 
|V(H)|-1\\
& \Longleftrightarrow
|E(A)|+1\ge 
|\CC(\NLP(H)\backslash A)| + 
\Big(
|V(\NLP(H))| -
|V(\NLP(H)\backslash A)|
\Big).
\end{align*}
This, however, is trivial as $\NLP(H)$ is connected whenever $H$ is connected. Furthermore, the graph induced by $V(\NLP(H))$ on edges 
$\NLP(H)\backslash A$ has connected components \linebreak
$\CC(\NLP(H)\backslash A)$ and
$V(\NLP(H))\backslash\Big(V(\NLP(H)\backslash A)\Big).$ Hence, we need to add at least \linebreak  
$|\CC(\NLP(H)\backslash A)| + |V(\NLP(H))\backslash\Big(V(\NLP(H)\backslash A)\Big)| - 1$ edges to make $\NLP(H)$ on vertex set\linebreak $V(\NLP(H))$ connected.
\end{proof}

Now, we go back to \cref{eq:whatrecursion}. As we know the difference of Fourier coefficients, we also need to understand the coefficients $M_{K,H}/M_{H,H}.$ 
We will prove the following fact.

\begin{proposition}[\cite{kothari2023planted}]
\label{prop:mkratio}
For any graph $H$ on at most $D$ vertices and any $K\subseteq H,$
$$
\frac{M_{K,H}}{M_{K,K}}\le 
4^{|E(H)| -|E(K)|}\times 
(1/q)^{|V(H)| - |V(K)|}.
$$
\end{proposition}
\begin{proof} We give the proof in \cite{kothari2023planted} for completeness.
Clearly, it is enough to prove the fact for $H$ connected as the bound factorizes over connected components (and the connected components of $K$ form a more refined partitioning of vertices than the components of $H$). Now, suppose that $H$ is connected. Then
\begin{align*}
&\frac{M_{K,H}}{M_{K,K}} = 
\frac{\prob[\bfX|_{E(H)} = E(H)\backslash E(K)]
\sqrt{(1 - \psi_q)/\psi_q}^{E(H)\backslash E(K)}}{\prob[\bfX|_{E(K)} = \emptyset]}\\
& \le 4^{|E(H)|\backslash E(K)}
\prob[\bfX|_{E(H)} = E(H)\backslash E(K)|\bfX\cap E(K) = \emptyset]\le 4^{|E(H)| - |E(K)|}(1/{q})^{|V(H)|-|V(K)|}.
\end{align*}
Each edge in $E(H)\backslash E(K)$ has at least one vertex in $V(H)\backslash V(K).$ As edges in $\bfX$ are determined by the colors assigned to vertices and $H$ is connected, the event $\bfX|_{E(H)} = E(H)\backslash E(K)$ (conditioned on the colors of vertices in $V(K)$) completely determines the colors of the vertices in $V(H)\backslash V(K).$ Thus, each vertex  in $V(H)\backslash V(K)$ can be assigned at most one color.
\end{proof}

With all of this, we can readily bound the values of $\widehat{w}_H.$ The following proposition from \cite{kothari2023planted} shows that it is enough to do so for connected graphs. Note that both $\RGGspherehalf$ and $\PCol(n,q)$ factorize over their connected components. One can easily obtain the following fact:
\begin{proposition} The values $\widehat{w}_H$ factorize over their connected components. That is, 
if the connected components of $H$ are $H_1, H_2, \ldots, H_t,$
$
\widehat{w}_H = \prod_{i = 1}^t\widehat{w}_{H_i}.
$
\end{proposition}

\begin{proposition}
    For any $H$ on at most $D$ vertices, 
    $$
    |\widehat{w}_H|
    \le
    (12|E(H)|)^{|E(H)|}
    \Big(\frac{(\log n)^{9}}{q}\Big)^{3|V(H)|/4}.
    $$
\end{proposition}
\begin{proof}
We prove this is by induction on $|E(H)|.$ It clearly holds for $H = \emptyset.$
The left-hand side is multiplicative over disjoint connected components and the right-hand side is super-multiplicative, so we only need to prove the statement for $H$ connected.
 Now, we use the recurrence relation \cref{eq:whatrecursion} as well as the bounds in \cref{prop:differenceinshifteedfourier,prop:mkratio} to obtain
\begin{equation*}
    \begin{split}
        &|\widehat{w}_H| = 
        \Big|
        (\Phi^{\psi_q}_\RGGspherehalf(H)-
\Phi^{\psi_q}_{\PCol(n,q)}(H))-
\sum_{\emptyset \subsetneq K\subsetneq H}\widehat{w}_K
\frac{M_{K,H}}{M_{K,K}}
        \Big|\\
        & \le 
        |\Phi^{\psi_q}_\RGGspherehalf(H)-
\Phi^{\psi_q}_{\PCol(n,q)}(H))| + 
\sum_{\emptyset \subsetneq K\subsetneq H}|\widehat{w}_K|
\frac{M_{K,H}}{M_{K,K}}\\
& \le 
8^{|E(H)|}
\Big(\frac{(\log n)^{9}}{q}\Big)^{3|V(H)|/4}\\
&\quad\quad\quad\quad\quad\quad+ 
\sum_{\emptyset \subsetneq K\subsetneq H}
(12|E(K)|)^{|E(K)|}
    \Big(\frac{(\log n)^{9}}{q}\Big)^{3|V(K)|/4}
    4^{|E(H)| - |E(K)|}\Big(\frac{1}{q}\Big)^{|V(H)| - |V(K)|}\\
& \le 
\Big(\frac{(\log n)^{9}}{q}\Big)^{3|V(H)|/4}\times 
\Big(
8^{|E(H)|} + 
\sum_{t= 0}^{|E(H)|-1}
(12 t)^t\binom{|E(H)|}{t}
4^{|E(H)| - t}
\Big)\\
& \le 
\Big(\frac{(\log n)^{9}}{q}\Big)^{3|V(H)|/4}\times 
\Big(
8^{|E(H)|} + 
\sum_{t= 0}^{|E(H)|-1}
(12 (|E(H)|-1))^t\binom{|E(H)|}{t}
4^{|E(H)| - t}
\Big)\\
& \le
\Big(\frac{(\log n)^{9}}{q}\Big)^{3|V(H)|/4}\times 
\Big(
8^{|E(H)|} + (12|E(H)| - 8)^{|E(H)|}\Big)\\
&\le(12|E(H)|)^{|E(H)|}\Big(\frac{(\log n)^{9}}{q}\Big)^{3|V(H)|/4},
    \end{split}
\end{equation*}
as desired.
\end{proof}
These bounds, together with \cref{prop:chisquaredtocoloring} are sufficient via an elementary calculation which we carry out in \cref{appendix:chisqcalcpcol}.

\section{The Second Eigenvalue of Half-Density RGG}
\label{section:spectral}
\subsection{The Trace Method}
For a real symmetric matrix $A\in \mathbb(R)^{n\times n},$ denote by $\lambda_1(A), \lambda_2(A), \ldots, \lambda_{n}(A)$ its eigenvalues in decreasing order. It is a well known fact that for any symmetric real matrix $A\in \mathbb{R}^{n\times n},$
$\lambda_2(A)= \min_{v}\lambda_1(A - vv^T).$ In particular, if $\bfA$ is the adjacency matrix of 
$\bfG\sim\RGG(n,\dsphere,1/2),$
$\lambda_2(\bfA)\le \lambda_1\Big(\bfA - \frac{1}{2}\one\one^T\Big).$

Now, we can bound $\lambda_1(\bfA- \frac{1}{2}\one\one^T)$ by the trace method. Namely, let $D$ be any even number.  
$$
\sum_{i=1}^n\lambda_i\Big(\bfA- \frac{1}{2}\one\one^T\Big)^D =
\trace\Big(\Big(\bfA - \frac{1}{2}\one\one^T\Big)^D\Big).
$$
In particular, $|\lambda_2(\bfA)|\le |\lambda_1\Big(\bfA- \frac{1}{2}\one\one^T\Big)|\le
\trace\Big(\Big(\bfA - \frac{1}{2}\one\one^T\Big)^D\Big)^{1/D}.
$
Hence, by Markov's inequality,
$$
\frac{1}{n}\ge \prob\Big[
\trace\Big(\Big(\bfA - \frac{1}{2}\one\one^T\Big)^D\Big)\ge
n \times \expect\Big[\trace\Big(\Big(\bfA - \frac{1}{2}\one\one^T\Big)^D\Big)\Big]
\Big]\ge 
\prob\Big[
|\lambda_2(\bfA)|^D\ge
n \times \expect\Big[\trace\Big((\bfA - \frac{1}{2}\one\one^T)^D\Big)\Big]
\Big].
$$
Thus, with high probability, 
$|\lambda_2(\bfA)|\le {n}^{1/D}\times \expect\Big[\trace\Big(\Big(\bfA - \frac{1}{2}\one\one^T\Big)^D\Big)\Big]^{1/D}.$
Taking $D = \omega(\log n),$ with high probability, 
\begin{equation}
\label{eq:lambda2bound}
    |\lambda_2(\bfA)|\le (1 +o_n(1))\expect\Big[\trace\Big(\Big(\bfA - \frac{1}{2}\one\one^T\Big)^D\Big)\Big]^{1/D}.
\end{equation}
In the rest of the section, we will be bounding the trace above. We choose $D = 2\Big\lceil(\log n)^{1.1}\rceil.$
\subsection{Fourier Coefficients in The Trace}
$\expect\Big[\trace\Big(\Big(\bfA - \frac{1}{2}\one\one^T\Big)^D\Big)\Big]$ has a very convenient expression as a weighted sum of Fourier coefficients. 
\begin{equation}
\label{eq:firstraceexpansion}
    \begin{split}
        & \expect\Big[\trace\Big(\Big(\bfA - \frac{1}{2}\one\one^T\Big)^D\Big)\Big] =\\
        & = \sum_{1\le i_1, i_2,\ldots,i_D\le n}
        \expect\Big[\prod_{s=1}^{D}(A_{i_si_{s+1}} - 1/2)\Big]\\
        & = 
        \sum_{1\le i_1, i_2,\ldots,i_m\le n}
        \expect\Big[\prod_{s=1}^{D}(G_{i_si_{s+1}} - 1/2)\Big]\Longrightarrow\\
        & \Big| \expect\Big[\trace\Big(\Big(\bfA - \frac{1}{2}\one\one^T\Big)^D\Big)\Big]\Big|\le 
        \sum_{1\le i_1, i_2,\ldots,i_D\le n}
        \Big|\expect\Big[\prod_{s=1}^{D}(G_{i_si_{s+1}} - 1/2)\Big]\Big|,
    \end{split}
\end{equation}
where $i_{D+1}:=i_1.$
Now, observe that for any indicator $I,$ one has $(I-1/2)^2 = 1/4.$ Thus, relevant are only terms that appear an odd number of times. Explicitly, let $H(i_1, i_2, \ldots,i_D)$ be the graph induced by the edges appearing an odd number of times in the multiset $\{(i_1,i_2),(i_2, i_3), \ldots, (i_D,i_1)\}.$ Then, \cref{eq:firstraceexpansion} is upper bounded by  
\begin{equation}
\label{eq:subgraphsumintrace}
    \begin{split}
        & \sum_{1\le i_1, i_2,\ldots,i_D\le n} 
        \expect[\SW_{H(i_1, i_2, \ldots,i_D)}(\bfG)]| = \sum_{H \in\graphsnoleaves_D}
        |\expect[\SW_{H}(\bfG)]|\times f_{D,n}(H),
    \end{split}
\end{equation}
where $f_{D,n}(H)$ is the number of $D$-tuples in $[n]^D$ such that $H(i_1, i_2, \ldots,i_D)$ is isomorphic to
$H.$ Again, we only sum over graphs with no leaves.

\subsection{Bounding The Trace}

We first bound $f_{D,n}(H).$ 

\begin{proposition} Suppose that $H$ is of minimal degree $2.$ Then,
$$
f_{D,n}(H)\le 2D^{3D}n^{
|V(H)|- |\CC(H)| + 
(D-|E(H)|)/2+1}
$$
\end{proposition}
\begin{proof} For $i_1, i_2, \ldots, i_D,$ denote by $K(i_1,i_2, \ldots, i_D)$ the multigraph on vertices $i_1, i_2, \ldots,i_D$ with edges $\{(i_1,i_2),(i_2, i_3), \ldots, (i_D,i_1)\}$ (counted with multiplicities). 

First, we claim that $K(i_1,i_2, \ldots, i_D)$ has at most $D - |E(H)|$ vertices if 
$H(i_1,i_2, \ldots, i_D) = H.$ Indeed, consider  the graph $K^c(i_1,i_2, \ldots, i_D)$ in which one modifies $K(i_1,i_2, \ldots, i_D)$ by contracting each connected component of $H(i_1,i_2, \ldots, i_D)$ to a single vertex. Then, the resulting multigraph has exactly $D - |E(H)|$ edges, is connected as $K$ clearly is, and each edge is of even multiplicity. In particular, this means that $|V(K^c(i_1,i_2,\ldots, i_D))|\le (D-|E(H)|)/2+1.$ However, 
$$
|V(K(i_1,i_2,\ldots, i_D))|= 
|V(K(i_1,i_2,\ldots, i_D))| + 
|V(H)|- |\CC(H)|\le
|V(H)|- |\CC(H)| + 
(D-|E(H)|)/2+1.
$$
Hence, one can choose the vertices $i_1,i_2, \ldots,i_D$ in at most 
$$
\sum_{t = |V(H)|}^{
|V(H)|- |\CC(H)| + 
(D-|E(H)|)/2+1}
n^t\le 
2n^{
|V(H)|- |\CC(H)| + 
(D-|E(H)|)/2+1}
$$
ways. On the other hand, for each fixed choice of $V(K(i_1,i_2,\ldots, i_D)),$ one needs to add an extra $(D-|E(H)|)/2$ double edges besides the ones determined by 
$H(i_1,i_2,\ldots, i_D).$\linebreak As $|V(K(i_1,i_2,\ldots, i_D))|\le D,$ this can be done in at most 
$$
\sum_{t = 0}^D
\binom{\binom{t}{2}}{D}\le
\binom{D^2}{D}\le
D^{2D}
$$
ways. Finally, note that for each fixed $K= K(i_1,\ldots, i_D),$ there are at most $D^D$ walks on the vertices $i_1, i_2, \ldots, i_D.$ 
\end{proof}

We now bound $|\SW_H(\bfG)|.$ 
First, from \cref{thm:oeiboundonfourier,prop:boundsonOEI},
$$
|\expect[\SW_{H}(\bfG)]|\le
\Big((\log n)^4/\sqrt{d}\Big)^{\sum_{H_i\in \CC(H)}\lceil (|V(H_i)|-1)/2\rceil}.
$$
Second, by \cref{eq:SOEI,prop:sparseSOEI},
$$
|\expect[\SW_{H}(\bfG)]|\le
\Big((\log n)^4/\sqrt{d}\Big)^{\sum_{H_i\in \CC(H)}\Big( 2|V(H_i)|-|E(H_i)|-2\Big)}.
$$
Hence, if we define 
$\mu(A):=\max\Big(
        \lceil (|V(A)|-1)/2\rceil, 
        2|V(A)|- |E(A)|- 2
        \Big),$ we obtain 
        
\begin{equation}
    \label{eq:compoundfourierboundfortrace}
    \begin{split}
        & |\expect[\SW_{H}(\bfG)]|\le 
        \Big((\log n)^4/\sqrt{d}\Big)^{\sum_{H_i\in \CC(H)}\mu(H_i)}.\\
    \end{split}
\end{equation}
Now, we go back to \cref{eq:firstraceexpansion}. 

\begin{equation}
\label{eq:tracefinal}
    \begin{split}
        & \sum_{H \in\graphsnoleaves_D}
        |\expect[\SW_{H}(\bfG)]|\times f_{D,n}(H)\\
        & \le 
        \sum_{H \in\graphsnoleaves_D}
        2D^{3D}n^{
|V(H)|- |\CC(H)| + 
(D-|E(H)|)/2+1}
\prod_{H_i \in \CC(H)}
\Big((\log n)^4/\sqrt{d}\Big)^{\mu(H_i)}\\
& = 
2D^{3D}n^{D/2+ 1}
\sum_{H \in\graphsnoleaves_D}
\prod_{H_i \in \CC(H)}
n^{|V(H_i)|- |E(H_i)|/2-1}
\Big((\log n)^4/\sqrt{d}\Big)^{\mu(H_i)}\\
& = 
2D^{3D}n^{D/2+ 1}
\sum_{H \in\graphsnoleaves_D}
\prod_{H_i \in \CC(H)}
n^{(2|V(H_i)|- |E(H_i)|-2)/2}\times
n^{|V(H_i)|/2}
\Big((\log n)^4/\sqrt{d}\Big)^{\mu(H_i)}\\
&\le
2D^{3D}n^{D/2+ 2}
\sum_{H \in\graphsnoleaves_D}
\prod_{H_i \in \CC(H)}
\Big(\sqrt{n}(\log n)^4/\sqrt{d}\Big)^{\mu(H_i)}\\
& \le 
2D^{5D}n^{D/2+ 2}
\max_{H \in\graphsnoleaves_D}
\prod_{H_i \in \CC(H)}
\Big(\sqrt{n}(\log n)^4/\sqrt{d}\Big)^{\mu(H_i)}.
    \end{split}
\end{equation}
For the last line, we simply counted the number of graphs $H$ on at most $D$ edges and $D$ vertices. Summing first over vertices and then edges, one can bound
\begin{align*}
& \sum_{v = 1}^D
\sum_{1\le e \le \max(\binom{v}{2}, D)}
\binom{\binom{v}{2}}{e} \le
\sum_{v = 1}^D
D\min\Big(2^{\binom{v}{2}}, \binom{\binom{v}{2}}{D}\Big)\le 
D^2\binom{\binom{D}{2}}{D}\le
D^{2D}.
\end{align*}

The last line of \cref{eq:tracefinal} clearly suggests two cases.

\paragraph{Case 1)} If $d\le n(\log n)^8.$ Then, the expression is maximized when $\sum_{H_i\in \CC(H)}\mu(H_i)$ is maximized. As $\mu(H_i)\le |V(H_i)|$ holds,
$\sum_{H_i\in \CC(H)}\mu(H_i)\le |V(H)|\le|V(H)|=D.$ Going back to \cref{eq:lambda2bound},
\begin{align*}
    &|\lambda_2(\bfA)|\le
    \Big(2D^{5D}n^{D/2+ 2}
\Big(\sqrt{n}(\log n)^4/\sqrt{d}\Big)^D
    \Big)^{1/D}
    \le
    (\log n)^{10}
    n/\sqrt{d}.
\end{align*}

\paragraph{Case 2)} If $d\ge n(\log n)^8,$ then the expression is maximized when $\sum_{H_i\in \CC(H)}\mu(H_i)$ is minimized. As $\mu(H_i)\ge 0,$ 
\begin{align*}
    &|\lambda_2(\bfA)|\le
    \Big(2D^{5D}n^{D/2+ 2}
    \Big)^{1/D}
    \le
    (\log n)^{10}
    \sqrt{n}.
\end{align*}

\section{Discussion}
We introduced a novel strategy for bounding the Fourier coefficients of graph distributions with high-dimensional latent geometry. It is based on localizing dependence to few edges and applying a noise operator to (some of) the remaining edges. Not only is this method useful for our concrete goal of bounding Fourier coefficients, but it also explains how and where dependence among edges is created. In the setting of $\RGGgauss, \RGGsphere,$ dependence is localized to fragile edges. For them, the signal $Z_{ji}$ is too low to overwhelm the bias $\sum_{\ell<i}Z_{i\ell}Z_{j\ell}$ introduced by other edges. We also showed that fragile edges give rise to an energy-entropy trade-off phenomenon in $\RGG$ (see \cref{rmk:energyentropytradeoff}).
We anticipate future applications of the fragile edges approach.

One future direction is tightening our bounds on the Fourier coefficients. In particular, proving a bound based on $\SOEI$ as in \cref{eq:SOEI} for all densities is appealing as it will extend \cref{thm:secondeigenvalue} to all densities. Can one improve further or is 
\cref{eq:SOEI} tight (up to lower-order terms)?

We used our bounds in several contexts related to low-degree hardness. The information-theoretic counterparts of many of these questions remain open. Is it possible to prove such information theoretic convergence using $\chi^2$-like arguments based on squares of Fourier coefficients? A simple calculation shows that bounds scaling as $p^{|E(H)|}d^{-\theta |V(H)|}$ for constant $\theta$ (of which form \cref{thm:oeiboundonfourier,eq:SOEI} both are) are insufficient to show $\chi^2(\RGGsphere\|\ergraph)= o_n(1)$ as there are $\exp(\omega(|V(H)|(\log d)))$ copies of $H$ in $K_n$ when $|V(H)| = \omega(\log^2n).$
Nevertheless, one could potentially use a tensorization argument first to reduce the computation to distributions over a smaller number of edges, for example as in \cite{Liu2021APV,Liu2022STOC}.

Finally, we separated sparse Gaussian and spherical random geometric graphs. It seems worthwhile to analyse this discrepancy further. First, it suggests that different methods are needed for studying their information-theoretic convergence to \ERspace in the regime $1/n\ll p\ll 1/2.$
Second, to what extent do sparse Gaussian random geometric graphs enjoy properties of their spherical counterparts such as high-dimensional expansion \cite{Liu22Expander}?

\section*{Acknowledgements}
We thank Chenghao Guo for insightful discussions at the initial stages of this project and Dheeraj Nagaraj for many conversations on random geometric graphs over the years. We are also grateful to three anonymous reviewers for the feedback and suggestions on the exposition.

\bibliography{ref}
\bibliographystyle{alpha}
\appendix

\section{Omitted Details from \texorpdfstring{\cref{section:gaussiantoER}}{Gaussian Theorems}}
\label{appendix:gaussiantoERcalc}
\subsection{Variance of The Signed Wedge Count in The Unmasked Case}
\label{appendix:varianceunmaskedwedge}
Going through the same steps as in the \ERspace case, one can show that 
\begin{equation}
    \var_{\bfG\sim \RGG}[\signedweight_{\wedge(1,2,3)}(\bfG)] = (p-p^2)^2 + (1-2p)^2\E[(G_{12} - p)(G_{13} - p)]\le 
    p^2 + \frac{p^2(\log d)^4}{d}\le 2p^2
\end{equation}
for large enough $d$ by \cref{prop:graphswithleaves}. So far the contribution to the variance is $O(n^{3}p^2).$ Next, we analyze the covariances based on the overlap pattern.

\paragraph{Case 2.1) Overlap of zero vertices.} That is, $\wedge(1,2,3)$ and $\wedge(4,5,6).$ Those are clearly independent as they are determined by a disjoint set of latent vectors, so the covariance is $0.$

\paragraph{Case 2.2) Overlap of one vertex.} There are several possible patterns: $\wedge(1,2,3)$ and $\wedge(1,4,5);$ $\wedge(1,2,3)$ and $\wedge(4,2,5);$ $\wedge(1,2,3)$ and $\wedge(4,1,5).$ Importantly, in all three cases, when we take the union of the two wedges we get a tree $T$ with $4$ edges. Thus, 
\begin{align*}
&\cov[\SW^p_{\wedge(i,j,k)}, \SW^p_{\wedge(i',j',k')}]\\
& = 
\E_{G\sim \RGG}[\SW^p_T(\bfG)] - 
\E_{G\sim\RGG}[\SW^p_{\wedge(1,2,3)}(\bfG)]^2\\
& \le
\E_{G\sim \RGG}[\SW^p_T(\bfG)] \le 
\frac{p^4}{d^2}\times (\log d)^8,
\end{align*}
where we used \cref{prop:graphswithleaves} for a tree on 4 edges. As there are $O(n^5)$ ways to choose the tree $T$ as it has 5 vertices, the contribution to the variance is of order $\tilde{O}({n^5p^4}/{d^2}).$

\paragraph{Case 2.3) Overlap of two vertices.} We analyze several different patterns separately:

\paragraph{Case 2.3.1)} $\wedge(1,2,3)$ and $\wedge(1,2,4).$ Then, 
\begin{equation}
\begin{split}
    &\E[\SW^p_{\wedge(1,2,3)}(\bfG)\SW^p_{\wedge(1,2,4)}(\bfG)]\\
    & = \E[(G_{12}-p)^2(G_{13}-p)(G_{14-p})]\\
    & = (1-2p)\E[(G_{12}-p)(G_{13}-p)(G_{14}-p)] + 
    (p-p^2)\E[(G_{13}-p)(G_{14}-p)]\\
    & = \tilde{O}(p^3/d^{3/2} + p^3/d) = {O}((\log d)^6p^3/d),
    \end{split}
\end{equation}
where we used \cref{prop:graphswithleaves} for the trees defined by edges $(12),(13)$ and $(12),(13),(14).$ Thus, the contribution is $O((\log d)^6 n^4p^3/d).$

\paragraph{Case 2.3.2)} $\wedge(1,2,3)$ and $\wedge(2,1,4).$ Then, 
\begin{equation}
\begin{split}
    &\E[\SW^p_{\wedge(1,2,3)}(\bfG)\SW^p_{\wedge(2,1,4)}(\bfG)]\\
    & = \E[(G_{12}-p)^2(G_{13}-p)(G_{24}-p)]\\
    & = (1-2p)\E[(G_{12}-p)(G_{13}-p)(G_{24}-p)] + 
    (p-p^2)\E[(G_{13}-p)(G_{24-p})]\\
    & = O((\log d)^6(p^3/d^{3/2} + p^3/d))= O((\log d)^6p^3/d),
\end{split}
\end{equation}
where we used \cref{prop:graphswithleaves} for the trees defined by edges $(12),(13)$ and $(12),(13),(24).$ Thus, the contribution is $O((\log d)^6n^4p^3/d).$

\paragraph{Case 2.3.3)} $\wedge(1,2,3)$ and $\wedge(4,2,3).$ Then, 
\begin{equation}
\begin{split}
    &\E[\SW^p_{\wedge(1,2,3)}(\bfG)\SW^p_{\wedge(4,2,3)}(\bfG)]\\
    & = \E[(G_{12}-p)(G_{13}-p)(G_{24}-p)(G_{34} - p)]\\
    & = O((\log d)^6(p^3/d^{3/2} + p^4/d^2))= O((\log d)^6p^4/d^2),
\end{split}
\end{equation}
where we used \cref{thm:oeiboundonfourier} and \cref{prop:boundsonOEI}  for the 4-cycle graph. The contribution is $\tilde{O}(n^4p^4/d^2).$

\paragraph{Case 2.4) Overlap of 3 vertices.} There are two patterns - $\wedge(1,2,3)$ and $\wedge(1,2,3),$ which is the variance case and $\wedge(1,2,3)$ and $\wedge(2,1,3).$ 
 In the latter case, 
 \begin{equation}
\begin{split}
    &\E[\SW^p_{\wedge(1,2,3)}(\bfG)\SW^p_{\wedge(2,1,3)}(\bfG)]\\
    & = \E[(G_{12}-p)^2(G_{13}-p)(G_{23} - p)]\\
    & =
    (1-2p)\E[(G_{12}-p)(G_{13}-p)(G_{23} - p)]+
    (p-p^2)
    \E[(G_{13}-p)(G_{23} - p)]\\
    & = O((\log d)^4(p^3/d^{1/2} + p^3/d))= O((\log d)^4p^3/d^{1/2}),
\end{split}
\end{equation}
by \cref{prop:graphswithleaves} for the triangle $(12),(23),(13)$ and the wedge $(13),(23).$ Thus, the contribution is $O((\log d)^4n^3p^3/d^{1/2}).$

Since $\E_{G\sim \RGG}[\CS^p_\wedge(\bfG)] -
\E_{H\sim \ergraph}[\CS^p_\wedge(H)] = \tilde{\Omega}(n^3p^2/d),$ for the signed wedge count test to distinguish the two models with high probability, it is sufficient that
$$
n^6p^4/((\log d)^8d^2)=\omega\Big( (\log d)^8\times \Big(
n^3p^2 + n^5p^4/d^2 + n^4p^3/d + n^4p^4/d^2
\Big)\Big).
$$
We analyze the inequalities separately:
\begin{enumerate}
    \item $n^6p^4/d^2= \omega\Big( (\log d)^{16}n^3p^2\Big).$ This holds if and only if $d=o(n^{3/2}p(\log d)^{-8}).$
    \item $n^6p^4/d^2=\omega\Big((\log d)^{16}n^5p^4/d^2\Big).$ 
    This holds for all $d,1/p$ under \eqref{eq:assumption}.
    \item $n^6p^4/d^2=\omega\Big((\log d)^{16}Cn^4p^3/d\Big).$ This holds if and only if $d=o(n^{2}p(\log d)^{-16}).$
    \item $n^6p^4/d^2=\omega\Big((\log d)^{16}n^4p^4/d^2\Big).$ This holds for all $d,1/p$ under \eqref{eq:assumption}.
\end{enumerate}
Altogether, whenever $d=o(n^{3/2}p(\log d)^{-8}),$ the signed-wedge test distinguishes the two graph models with high probability.\hfill\qedhere

\subsection{Variance of The Signed Triangle Count in The Unmasked Case}
\label{appendix:varianceunmaskedtriangle}

\paragraph{Case 2.1) Overlap of zero vertices.} That is $\triangle(1,2,3)$ and $\triangle(2,3,4)$ are clearly independent.

\paragraph{Case 2.2)  Overlap of one vertex.} $\triangle(1,2,3)$ and $\triangle(1,4,5).$ Then, 
\begin{equation}
\begin{split}
    &\E[\SW_{\triangle(1,2,3)}(\bfG)\SW_{\triangle(1,4,5)}(\bfG)]=
    \E[\SW_H(\bfG)]=O(p^6(\log d)^4/d),
    \end{split}
\end{equation}
where $H$ is the graph on vertex set $[5]$ with edges $(12),(13),(23), (14),(15),(45).$ We used \cref{thm:oeiboundonfourier} and \cref{prop:boundsonOEI}. The contribution is $O(n^5p^6(\log d)^4/d).$ Indeed, $\OEI(H)\ge 2$ as no edge is adjacent to all other edges.

\paragraph{Case 2.2)  Overlap of two vertices.} 
$\triangle(1,2,3)$ and $\triangle(1,2,4).$ Then, 

\begin{equation}
\begin{split}
    &\E[\SW_{\triangle(1,2,3)}(\bfG)\SW_{\triangle(1,2,4)}(\bfG)]\\
    &=
    \E[(G_{12}-p)^2(G_{13}-p)(G_{23}-p)(G_{14}-p)(G_{24}-p)]\\
    & =
    (1-2p)\E[(G_{12}-p)(G_{13}-p)(G_{23}-p)(G_{14}-p)(G_{24}-p)]+\\
    & +\quad\quad\quad(p-p^2)
    \E[(G_{13}-p)(G_{23}-p)(G_{14}-p)(G_{24}-p)]\\
    & = o(p^5(\log d)^4/d^2),
    \end{split}
\end{equation}
where we used \cref{thm:oeiboundonfourier} and \cref{prop:boundsonOEI} on the 4-cycle and the graph with $4$ vertices and 5 edges. The contribution is $O(n^4p^5(\log d)^4/d).$

\paragraph{Case 2.3)  Overlap of three vertices.} 
$\triangle(1,2,3)$ and $\triangle(1,2,3).$ Then, 

\begin{equation}
\begin{split}
    &\E[\SW_{\triangle(1,2,3)}(\bfG)\SW_{\triangle(1,2,3)}(\bfG)]\\
    &=
    \E[(G_{12}-p)^2(G_{13}-p)^2(G_{23}-p)^2]\\
    & =
    (1-2p)^3\E[(G_{12}-p)(G_{13}-p)(G_{23}-p)]+3(1-2p)(p-p^2)\E[(G_{12}-p)(G_{13}-p)]\\
    &\quad\quad\quad+
    3(1-2p)(p-p^2)\E[(G_{12}-p)]+ 
    (p-p^2)^3
    \\
    & = O(p^3),
    \end{split}
\end{equation}
where we used \cref{thm:oeiboundonfourier} and \cref{prop:boundsonOEI} on the triangle and \cref{prop:graphswithleaves} for the wedge. The total contribution is $O(n^3p^3).$

Finally, in light of \cref{def:successofLDP}, we have to show 
$$
n^6p^6(\log d)^{-2}/d=\omega\Big(
n^5p^6(\log d)^4/d + 
n^4p^5(\log d)^4/d + 
n^3p^3
\Big)
$$
This clearly holds whenever $d\le (n^3p^3)(\log d)^{-2}.$

\subsection{Variance of The Signed Wedge Count in the Masked Case}
\label{appendix:variancemaskedwedge}

\paragraph{Case1) Overlap of 0 vertices.} As usual, the correlation is 0.

\paragraph{Case 2) Overlap of 1 vertex.} $\wedge(1,2,3)$ and $\wedge(1,4,5).$ Let $H$ be the star with 4 leaves. Then, 
$$
\expect[\signedweight_{\wedge(1,2,3)}(\bfG)\signedweight_{\wedge(1,4,5)}(\bfG)]= 
\E[\signedweight_{H}(\bfG)]= O\Big( p^4(\log d)^8/d^2\Big)
$$
by\cref{prop:graphswithleaves}.
Thus, the total contribution is $O(AN^4p^4(\log d)^8/d^2).$

\paragraph{Case 3) Overlap of 2 vertices.} $\wedge(1,2,3)$ and $\wedge(1,2,4).$ Then, 
\begin{align*}
& \expect[\signedweight_{\wedge(1,2,3)}(\bfG)\signedweight_{\wedge(1,2,4)}(\bfG)]\\
& = \expect[(G_{12}-p)^2(G_{13}-p)(G_{14}-p)]\\
& = (1-2p)\expect[(G_{12}-p)(G_{13}-p)(G_{14}-p)]+
(p-p^2)\expect[(G_{13}-p)(G_{14}-p)]
= O\Big( p^3(\log d)^4/d\Big),
\end{align*}
using \cref{prop:graphswithleaves} on the wedge and 3-star.
Thus, the total contribution is $O(AN^3p^3(\log d)^4/d).$

\paragraph{Case 4) Overlap of 3 vertices.} $\wedge(1,2,3)$ and $\wedge(1,2,3).$ Then, 
\begin{align*}
& \expect[\signedweight_{\wedge(1,2,3)}(\bfG)\signedweight_{\wedge(1,2,3)}(\bfG)]\\
& = \expect[(G_{12}-p)^2(G_{13}-p)^2]\\
& = (1-2p)^2\expect[(G_{12}-p)(G_{13}-p)]+
2(1-2p)(p-p^2)\expect[(G_{13}-p)]+p^2
= O( p^2),
\end{align*}
using \cref{prop:graphswithleaves} on the wedge.
Thus, the total contribution is $O(AN^2p^2).$

Thus, in order to distinguish between $\RGGgauss$ and $\ergraph$ with mask $\mathcal{M},$ we need the condition 
$$
A^2N^4p^4(\log d)^{-8}/d^2= 
\omega\Big(
AN^4p^4(\log d)^8/d^2+ 
AN^3p^3(\log n)^4/d+ 
AN^2p^2\Big).
$$
One can check that this holds when $A = (\log d)^{17},d\le Np.$
\subsection{Lower Bound on Expected Signed Wedges in Gaussian RGG}
\label{sec:LowerBoundsGaussianFourier}
\begin{proof}[Proof of \cref{lem:wedgelowerbound}] To prove this statement, we decompose each Gaussian vector $Z_i$ as $V_i\times\|Z_i\|_2,$ where $V_i\sim \unif(\dsphere)$ and $d\|Z_i\|^2\sim \chi^2(d).$ Importantly, 
$V_i$ and $\|Z_i\|_2$ are independent. In this decomposition, $\langle V_1, V_3\rangle$ and $\langle V_1,V_2\rangle$ are independent, unlike $\langle Z_1,Z_2\rangle.$ We use the independence in the same way as in \cref{prop:leavesinspherical}.

\begin{equation}
\label{eq:wedgeexpectation}
    \begin{split}
& \E\Big[
(G_{21}- p)(G_{31}-p)
\Big]\\
& = 
\E\Big[\Big(
\ind\Big[ \langle V_1, V_2\rangle \ge \frac{\rho^p_d}{\|Z_1\|_2\times \|Z_2\|_2}\Big]-p\Big)\times 
\Big(
\ind\Big[ \langle V_1, V_3\rangle \ge \frac{\rho^p_d}{\|Z_1\|_2\times \|Z_3\|_2}\Big]-p\Big)
\Big]\\
& = 
\E\Big[\E\Big[
\Big(
\ind\Big[ \langle V_1, V_2\rangle \ge \frac{\rho^p_d}{\|Z_1\|_2\times \|Z_2\|_2}\Big]-p\Big)\times 
\Big(
\ind\Big[ \langle V_1, V_3\rangle \ge \frac{\rho^p_d}{\|Z_1\|_2\times \|Z_3\|_2}\Big]-p\Big)\; \Big|\; \|Z_1\|_2
\Big]\Big]
\\
& = 
\E\Big[\E\Big[\Big(\ind\Big[ \langle V_1, V_2\rangle \ge \frac{\rho^p_d}{\|Z_1\|_2\times \|Z_2\|_2}\Big]-p\Big)\Big|\|Z_1\|\Big]^2\Big].
    \end{split}
\end{equation}
For brevity denote $h(x):=\E\Big[\Big(\ind\Big[ \langle V_1, V_2\rangle \ge \frac{\rho^p_d}{x\times \|Z_2\|_2}\Big]-p\Big)\Big].$
We know that $\E[h(\|Z_1\|)] = 0$ by definition of $\rho^p_d.$ Also, we can rewrite \cref{eq:wedgeexpectation} as
$
\E[h(\|Z_1\|)^2].
$

Now, by $\chi$-square concentration, the median of $\|Z_1\|$ is $1 +\frac{m}{\sqrt{d}}$ where $|m|\le 20.$ Indeed, this holds by setting $y =8/\sqrt{d}$ in \cref{prop:chisquaredconcnetrations}. 

Now, using the chi-squared density, one can easily prove that $\Pr[\|Z_1\|\ge 1 + \frac{(m+1)}{\sqrt{d}}]\ge C$ for some absolute constant $C.$ On the one hand,  $\|Z_1\|_2\in [1-10/\sqrt{d},1+\sqrt{10}/d]$ with constant probability by \cref{prop:chisquaredconcnetrations}. On the other hand, the density of $\|Z_1\|_2$ on  $[1 +\frac{21}{\sqrt{d}}, \frac{22}{\sqrt{d}}]$ is not that much smaller than the density on $[(1-10/\sqrt{d}),(1+\sqrt{10}/d)],$ hence also a constant fraction of the mass of $\|Z_1\|_2$ is in $[1 +\frac{21}{\sqrt{d}}, \frac{22}{\sqrt{d}}].$ Formally, the density $f_d(v)$ of $\|Z_1\|^2$ is given by 
$$
f_d(v) = \frac{1}{2^{d/2}\Gamma(d/2)}(dv)^{d/2 -1} e^{-(dv)/2}.
$$
However, whenever $\|Z_1\|_2\in [1-10/\sqrt{d},1+\sqrt{10}/d]$ holds, $\|Z_1\|^2_2\in [1-21/\sqrt{d},1+\sqrt{21}/d]$ holds and whenever 
$\|Z_1\|_2\in[1 +\frac{21}{\sqrt{d}}, \frac{22}{\sqrt{d}}]$ holds, 
$\|Z_1\|_2\in[1 +\frac{40}{\sqrt{d}}, \frac{44}{\sqrt{d}}]$ holds. However, for any $a\in [-21,21], b\in [40, 44],$ one has
\begin{equation}
    \begin{split}
        & \frac{f_d(1 + \frac{b}{d})}{f_d(1 - \frac{a}{d})}
        = 
        \Big(\frac{1 +\frac{b}{\sqrt{d}}}{1 - \frac{a}{\sqrt{d}}}\Big)^{d/2-1}
        e^{- (b-a)\sqrt{d}/2}\\
        & \ge 
        \frac{1}{2}
        \Big(\Big(\frac{1 +\frac{b}{\sqrt{d}}}{1 - \frac{a}{\sqrt{d}}}\Big)^{\sqrt{d}}e^{-(b-a)}\Big)^{\sqrt{d}/2}\ge 
        \frac{1}{2}
        \Big(\Big(\frac{1 +\frac{b}{\sqrt{d}}}{1 - \frac{a}{\sqrt{d}}}\Big)^{\sqrt{d}}\Big(1 - \frac{(b-a)}{\sqrt{d}-(b-a)}\Big)^{\sqrt{d}}\Big)^{\sqrt{d}/2}\\
        & =
        \frac{1}{2}
        \Big(\Big(\frac{d +(2a-b)\sqrt{d} - 2b(b-a)}{d  +(2a-b)\sqrt{d} + a(b-a)}\Big)\Big)^{d/2}\ge  
        \exp(- 10(2b-a)(b-a))= \Omega(1),
    \end{split}
\end{equation}
where in the last line we just used $(1-x/n)^n\ge \exp(-10x)$ for all large enough $x.$ This is enough as $a,b$ are both at most $44$ in absolute value.

Going back, this means that 
\begin{equation}
\label{eq:isolatingonlyhigherhalf}
    \begin{split}
        & \E\Big[
(G_{21}- p)(G_{31}-p)
\Big]  = 
\E[h(\|Z_1\|)^2]\\
& \ge \E\Big[h(\|Z_1\|)^2\Big|\|Z_1\|_2\ge 1 + \frac{(m+1)}{\sqrt{d}}\Big]\times \Pr\Big[\|Z_1\|_2\ge 1 + \frac{(m+1)}{\sqrt{d}}\Big]\\
& \ge C\times \E\Big[h(\|Z_1\|)\Big|\|Z_1\|_2\ge 1 + \frac{(m+1)}{\sqrt{d}}\Big]^2\\
    \end{split}
\end{equation}
Now, observe that since $h$ is clearly increasing, one has
\begin{equation}
\begin{split}
& \E\Big[h(\|Z_1\|)\Big|\|Z_1\|_2\ge 1 + \frac{(m+1)}{\sqrt{d}}\Big] \ge 
\E\Big[h(\|Z_1\|)\Big|\|Z_1\|_2\ge 1 + \frac{m}{\sqrt{d}}\Big],\\
& \E\Big[h(\|Z_1\|)\Big|\|Z_1\|_2\ge 1 + \frac{(m+1)}{\sqrt{d}}\Big] \ge 
\E\Big[h(\|Z_1\| + \frac{1}{\sqrt{d}})\Big|\|Z_1\|_2< 1 + \frac{m}{\sqrt{d}}\Big].
\end{split}
\end{equation}

Using the concentration of $\|Z_2\|$ and $\|Z_1\|,$ we can show that 
\begin{equation}
\label{eq:comparisontolowerhalf}
\E\Big[h(\|Z_1\| + \frac{1}{\sqrt{d}})\Big|\|Z_1\|_2< 1 + \frac{m}{\sqrt{d}}\Big]\ge 
\E\Big[h(\|Z_1\|)\Big|\|Z_1\|_2< 1 + \frac{m}{\sqrt{d}}\Big] + \frac{p \times (\log d)^{-2})}{\sqrt{d}}.
\end{equation}

Indeed, this is true for the following reason. With probability $1 - \exp(- (\log d)^2),$ both of $\|Z\|_1$ and $\|Z_2\|_2$ are of order $1 + O(\log d/\sqrt{d}),$ even conditioned on $\|Z\|_1<1 + m/\sqrt{d}$ since this latter event has probability $1/2.$
However, whenever this occurs, one has 
\begin{align*}
& \frac{\rho^p_d}{\|Z_1\|\times \|Z_2\|} - 
\frac{\rho^p_d}{(\|Z_1\| + \frac{1}{\sqrt{d}})\times \|Z_2\|}
 =\frac{\rho^p_d}{\|Z_1\|\times \|Z_2\|\times ((\|Z_1\| + \frac{1}{\sqrt{d}}))} \frac{1}{\sqrt{d}} \ge 
\frac{K\sqrt{\log d}}{{d}}
\end{align*}
for some absolute constant $K.$ We used $p\le 1/2-\varepsilon$ so that $\rho^p_d\gtrsim 1/\sqrt{d}.$
This is enough to show that  
$$
\Pr\Big[\langle V_1,V_2\rangle \ge  \frac{\rho^p_d}{(\|Z_1\| + \frac{1}{\sqrt{d}})\times \|Z_2\|}\Big] - 
\Pr\Big[\langle V_1,V_2\rangle \ge  \frac{\rho^p_d}{\|Z_1\|\times \|Z_2\|}\Big]\ge 
\frac{p(\log d)^{-2}}{\sqrt{d}}
$$
by \cref{cor:innerproductsinintervals}. Finally, \cref{eq:comparisontolowerhalf} implies the following bound.
\begin{equation}
\begin{split}
& \E\Big[h(\|Z_1\|)\Big|\|Z_1\|_2\ge 1 + \frac{(m+1)}{\sqrt{d}}\Big] \\
& \ge \frac{1}{2}\Big(
\E\Big[h(\|Z_1\|)\Big|\|Z_1\|_2\ge 1 + \frac{m}{\sqrt{d}}\Big] + 
\E\Big[h(\|Z_1\|)\Big|\|Z_1\|_2< 1 + \frac{m}{\sqrt{d}}\Big] + \frac{p \times (\log d)^{-2}}{\sqrt{d}}
\Big)\\
& = 
\E\Big[h(\|Z_1\|)\Big] + \frac{p \times (\log d)^{-2}}{\sqrt{d}}= 
\frac{p \times (\log d)^{-2}}{\sqrt{d}}.\\
\end{split}
\end{equation}
Going back to \cref{eq:isolatingonlyhigherhalf}, this gives the desired result.
\end{proof}
\subsection{Lower Bound on Expected Signed Triangles in Gaussian RGG}
\label{sec:lowertriangles}
\begin{proof}[Proof of \cref{thm:signedtrianglegauss}] First, note that 
\begin{equation*}
\begin{split}
    &\expect_{\bfG\sim\RGGgauss}\Big[(G_{21}-p)(G_{32}-p)(G_{31}-p)\Big]\\
    &\expect[G_{21}G_{32}G_{31}]-p^3\\
    &\quad\quad\quad-
    \Big(
    p\expect[(G_{21}- p)(G_{31} - p)] + 
    p\expect[(G_{21}- p)(G_{32} - p)] + 
    p\expect[(G_{31}- p)(G_{31} - p)]
    \Big)\\
    & \quad\quad\quad
    -\Big(
        p^2\expect[(G_{21}- p)]+
        p^2\expect[(G_{31}- p)]+
        p^2\expect[(G_{32}- p)]
    \Big).
\end{split}
\end{equation*}     
By \cref{lemma:treecounts}, the expression in the first brackets is on the order of $\tilde{O}(p^3/d).$ the expression in the second brackets it is trivially zero. Thus, we simply need to show that
$$
\expect[G_{21}G_{32}G_{31}]-p^3 =\Omega(p^3(\log d)^{-1}/\sqrt{d}).
$$
Now, observe that
\begin{align*}
    &\expect[G_{21}G_{32}G_{31}]\\
    &= 
    \prob[\langle Z_{3},Z_2\rangle \ge \rho^p_d,\langle Z_1,Z_2\rangle\ge \rho^p_d, \langle Z_1, Z_3\rangle\ge \rho^p_d]\\
    & \ge 
    \prob[\langle Z_{3},Z_2\rangle \ge \rho^p_d,\langle Z_1,Z_2\rangle\ge \rho^{p/2}_d, \langle Z_1, Z_3\rangle\ge \rho^{p/2}_d]\\
    & = 
    \prob[\langle Z_{3},Z_2\rangle \ge \rho^p_d|\langle Z_1,Z_2\rangle\ge \rho^{p/2}_d, \langle Z_1, Z_3\rangle\ge \rho^{p/2}_d]\times \prob[\langle Z_1,Z_2\rangle\ge \rho^{p/2}_d, \langle Z_1, Z_3\rangle\ge \rho^{p/2}_d].
\end{align*}
\cref{lem:wedgelowerbound} shows that $G_{21},G_{31}$ are positively correlated, hence
$\prob[\langle Z_1,Z_2\rangle\ge \rho^{p/2}_d, \langle Z_1, Z_3\rangle\ge \rho^{p/2}_d]\ge p^2/4.$ Using the Bartlett decomposition, we are left to bound
$$
\prob[Z_{22}Z_{32}+Z_{21}Z_{31}\ge \rho^{p}_{d}|Z_{11}Z_{21}\ge \rho^{p/2}_d, Z_{11}Z_{31}\ge \rho^{p/2}_d].
$$
Now, we use that with probability at least $1 - \exp(-(\log d)^2),$ one has $Z_{11}\le 1+O(\frac{\log d}{d})\le 2.$ Thus, we have that
\begin{align*}
    &\prob[Z_{22}Z_{32}+Z_{21}Z_{31}\ge \rho^{p}_{d}|Z_{11}Z_{21}\ge \rho^{p/2}_d, Z_{11}Z_{31}\ge \rho^{p/2}_d]\\
    &\ge
    \prob[Z_{22}Z_{32}+Z_{21}Z_{31}\ge \rho^{p}_{d}|Z_{11}Z_{21}\ge \rho^{p/2}_d, Z_{11}Z_{31}\ge \rho^{p/2}_d,Z_{11}\le 2]- \exp(-(\log d)^2)\\
    & \ge 
    \prob[Z_{22}Z_{32}\ge\rho^{p}_d - \frac{(\rho^{p/2}_d)^2}{4}]- \exp(-(\log d)^2)\\
    & =
    \prob\Big[Z_{22}Z_{32}\ge\rho^{p}_d\Big]+ 
    \prob\Big[Z_{22}Z_{32}\in\Big[\rho^{p}_d - \frac{(\rho^{p/2}_d)^2}{4},\rho^{p}_d\Big]\Big]- \exp(-(\log d)^2).
\end{align*}
Finally, note that $\prob\Big[Z_{22}Z_{32}\ge\rho^{p}_d\Big]= 
\prob\Big[Z_{11}Z_{31}\ge\rho^{p}_d\Big] = p.
$ However, $\frac{(\rho^{p/2}_d)^2}{4}\ge \frac{(\log d)^{-1}}{d}$ (this is the reason we used $\rho^{p/2}_d.$ When $p= 1/2,$ one has $\rho^{1/2}_d=0$). The result follows from \cref{cor:innerproductsinintervals}.\end{proof}

\subsection{Upper Bound on Expected Signed Graphs with Leaves in Gaussian RGG}
\label{sec:fourierwithleaveslowerGaussian}
We begin by analyzing the simplest case of graphs with leaves - that of trees. 
\begin{lemma}[Signed Expectations of Trees]
\label{lemma:treecounts}
Suppose that $T$ is a tree with $t$ vertices $[t]$ such that $v = |V(T)|\le d(\log d)^{-4}.$ Let $\mathcal{E}$ be the event that $\Big|\|Z_i\|_2-1\Big|= \psi_i,$ where each $\psi_i$ is some constant satisfying $|\psi_i|\le \frac{A\log d }{\sqrt{d}}$ holds for all $i\in [t]$ for some $A$ such that $\sqrt{d}(\log d)^{-2}\ge A\ge C\sqrt{|V(T)|}$ for some absolute constant $C.$ Then,
\begin{align}
    & \Big|\E_{\bfG\sim\RGGgauss}\Big[\prod_{(ji)\in E(T)}(G_{ji}-p)\Big|\mathcal{E}\Big]\Big|\le \Big(\frac{p A(\log d)^2}{\sqrt{d}}\Big)^{|E(T)|},\\ 
    & \Big|\E_{\bfG\sim\RGGgauss}\Big[\prod_{(ji)\in E(T)}(G_{ji}-p)\Big]\Big|\le \Big(\frac{Cp\sqrt{|V(T)|} (\log d)^2}{\sqrt{d}}\Big)^{|E(T)|}.
\end{align}
\end{lemma}
\begin{proof}
Suppose that $V(T) = \{1,2,\ldots, t\}.$
Let $Z_1, Z_2, \ldots, Z_{t}$ be the latent Gaussian vectors. Let $V_i = Z_i/\|Z_i\|_2.$ Note that $V_i$ and $\|Z_i\|_2$  are independent as in \cref{lem:wedgelowerbound}. Again, we write,

\begin{equation}
    \begin{split}
    \label{eq:splittingnormanddirection}
        & \E\Big[\prod_{(ji)\in E(T)}(G_{ji}-p)\Big|{\mathcal{E}}\Big]\\
        & = \E\Big[\prod_{(ji)\in E(T)}\Big(\ind\Big[ \langle V_i, V_j\rangle \ge \frac{\rho^p_d}{\|Z_i\|_2\times \|Z_j\|_2}\Big]-p\Big)\Big|{\|Z_i\|_2=\psi_i\; \forall i}\Big]\\
        & =
        \E\Big[\E\Big[\prod_{(ji)\in E(T)}\Big(\ind\Big[ \langle V_i, V_j\rangle \ge \frac{\rho^p_d}{\|Z_i\|_2\times \|Z_j\|_2}\Big]-p\Big)\Big|{\|Z_i\|_2=\psi_i\; \forall i}\Big]\Big].\\
    \end{split}
\end{equation}
Conditioned on $\{\|Z_\ell\|_2\}_{\ell = 1}^t,$ the values of $\ind\Big[ \langle V_i, V_j\rangle \ge \frac{\rho^p_d}{\|Z_i\|_2\times \|Z_j\|_2}\Big]$ are independent since $T$ is a tree by \cref{prop:leavesinspherical}. Thus, we can rewrite \cref{eq:splittingnormanddirection} as
\begin{equation}
    \begin{split}
        \E\Big[\prod_{(ji)\in E(T)}\E\Big[\Big(\ind\Big[ \langle V_i, V_j\rangle \ge \frac{\rho^p_d}{\|Z_i\|_2\times \|Z_j\|_2}\Big]-p\Big)\Big|\{\|Z_\ell\|_2\}_{\ell = 1}^t,\mathcal{E}\Big]\Big]
    \end{split}
\end{equation}
Note that under $\mathcal{E},$ we have that for each $(ji)\in E(T),$ one has
\begin{align*}
& \frac{\rho^p_d}{\|Z_i\|_2\times \|Z_j\|_2} = 
\rho^p_d + O\Big(\frac{(\log d)^{3/2}A}{d}\Big).
\end{align*}
It immediately follows from \cref{cor:innerproductsinintervals} that for each $(ji)$ one has
$$
\Big|
\E\Big[\Big(\ind\Big[ \langle V_i, V_j\rangle \ge \frac{\rho^p_d}{\|Z_i\|_2\times \|Z_j\|_2}\Big]-p\Big)\; \Big|\; \mathcal{E}\Big]\Big| = O\Big(\frac{p\times (\log d)^{3/2}A}{\sqrt{d}}\Big).
$$
The first inequality follows.

Now, by \cref{prop:chisquaredconcnetrations} and union bound, with probability $1 - |V(T)|\times \exp\Big(-(\log d)^{2}A^2\Big) \ge 1-\exp(-(\log d)^2|V(T)|/2)=1 - o\Big(\Big(p/\sqrt{d}\Big)^{|V(T)|}\Big)$ each $\|Z_i\|_2$ satisfies $\Big|\|Z_i\|_2 -1\Big| \le \frac{C\log d A^2}{\sqrt{d}}$ for some absolute constant $C.$ Let $\mathcal{F}$ be the event that each $\Big|\|Z_i\|_2 -1\Big| \le \frac{C\log d A^2}{\sqrt{d}}$ holds for all $i.$
Since $$\Big|\prod_{(ji)\in E(T)}\E\Big[\Big(\ind\Big[ \langle V_i, V_j\rangle \ge \frac{\rho^p_d}{\|Z_i\|_2\times \|Z_j\|_2}\Big]-p\Big)\Big|\{\|Z_\ell\|_2\}_{\ell = 1}^t\Big]\Big|\le 1$$ almost surely, we conclude that 
\begin{equation}
    \begin{split}
        & \Big|\E\Big[\prod_{(ji)\in E(T)}(G_{ji}-p)\Big]\Big|\\
        & 
        \le \Big|\E\Big[\prod_{(ji)\in E(T)}\E\Big[\Big(\ind\Big[ \langle V_i, V_j\rangle \ge \frac{\rho^p_d}{\|Z_i\|_2\times \|Z_j\|_2}\Big]-p\Big)\Big|\{\|Z_\ell\|_2\}_{\ell = 1}^t\Big]\; \Big|\; \mathcal{F}\Big]\Big| + 
        (1 - \Pr[\mathcal{F}])\\
        & 
        \le \Big|\E\Big[\prod_{(ji)\in E(T)}\E\Big[\Big(\ind\Big[ \langle V_i, V_j\rangle \ge \frac{\rho^p_d}{\|Z_i\|_2\times \|Z_j\|_2}\Big]-p\Big)\Big|\{\|Z_\ell\|_2\}_{\ell = 1}^t\Big]\; \Big|\; \mathcal{F}\Big]\Big|+o\Big(\Big(\frac{p}{\sqrt{d}}\Big)^{|V(T)|}\Big).
    \end{split}
\end{equation}
The second inequality follows.
\end{proof}

Now, we can easily combine the statement for trees \cref{lemma:treecounts} and \cref{thm:oeiboundonfourier} as follows. 

\begin{proof}[Proof of \cref{prop:graphswithleaves}] Let $A  = |V(H)|\log d.$ Let $\mathcal{F}$ be the event that each $\Big|\|Z_i\|_2- 1\Big|\le \frac{|V(H)|\log d}{\sqrt{d}}.$ Clearly, 
$\prob[\mathcal{F}]\ge 1 - \exp(-4|V(H)|^2\log d)= 1 - o((p/\sqrt{d})^{|E(H)|}).$ Let also $\TP(H)$ be composed of $k$ trees $T_1,T_2,\ldots, T_H.$ Then,
     \begin{align*}
         &\Big|
            \expect\Big[\prod_{(ji)\in E(H)}
            (G_{ji}-p)
            \Big]
         \Big|\\
         &= \Big|
            \expect\Big[\prod_{i=1}^k \SW_{T_i}(\bfG)\times\SW_{\NLP(H)}(\bfG)
            \Big]
         \Big|\\
         & \le
         \Big|
         \Big[
            \prod_{i=1}^k \SW_{T_i}(\bfG)\times\SW_{\NLP(H)}(\bfG)
            \Big|\mathcal{F}\Big]
         \Big|+ (1- \prob[\mathcal{F}])\\
         & \le
         \Big|\expect\Big[\expect\Big[\
            \prod_{i=1}^k \SW_{T_i}(\bfG)\times\SW_{\NLP(H)}(\bfG)
            \Big|\{\|Z_i\|_2\}_{i= 1}^{|V(H)|}\Big]\Big|\mathcal{F}\Big]
         \Big| + o((p/\sqrt{d})^{|E(H)|}).
     \end{align*}
     
Exactly as in the proof of \cref{lemma:treecounts}, conditioned on the latent vectors, the expectation of each tree part is independent of the rest conditioned on the norms. Thus,  
\begin{align*}
    & \Big|\expect\Big[\expect\Big[\
            \prod_{i=1}^k \SW_{T_i}(\bfG)\times\SW_{\NLP(H)}(\bfG)
            \Big|\{\|Z_i\|_2\}_{i= 1}^{|V(H)|}\Big]\Big|\mathcal{F}\Big]
         \Big|\\
    & = \Big|\expect\Big[\prod_{i=1}^k\expect\Big[\
             \SW_{T_i}(\bfG)\Big|\{\|Z_i\|_2\}_{i= 1}^{|V(H)|}\Big]\times\expect\Big[\SW_{\NLP(H)}(\bfG)
            \Big|\{\|Z_i\|_2\}_{i= 1}^{|V(H)|}\Big]\Big|\mathcal{F}\Big]
         \Big|\\
    & \le 
    \expect\Big[\prod_{i=1}^k\Big|\expect\Big[\
             \SW_{T_i}(\bfG)\Big|\{\|Z_i\|_2\}_{i= 1}^{|V(H)|}\Big]\Big|\times\Big|\expect\Big[\SW_{\NLP(H)}(\bfG)
            \Big|\{\|Z_i\|_2\}_{i= 1}^{|V(H)|}\Big|\Big]\Big|\mathcal{F}\Big]
         \Big|.
\end{align*}
Now, we are essentially done. We use \cref{lemma:treecounts} on each tree to get the
$\le (8p)^{|E(\TP(H))|}\times 
\Big(\frac{C(km)^3\log d}{\sqrt{d}}\Big)^{|E(\TP(H))| }$ factor. On the other hand, 
$\SW_{\NLP(H)}(\bfG)
\Big|_{\{\|Z_i\|_2\}_{i= 1}^{|V(H)|}}$is a spherical random geometric graph with inhomogeneous connections 
$\indicator\Big[\langle V_i,V_j\rangle \ge \frac{\rho^p_d}{\|Z_i\|_2\times \|Z_j\|_2}\Big]$ at each edge $(ji).$ However, under $\mathcal{F},$ one has 
$\frac{\rho^p_d}{\|Z_i\|_2\times \|Z_j\|_2} = \tau^p_d + O(\frac{k(\log d)^2}{d}).$ As 
$\frac{k(\log d)^2}{d}\le \frac{C{km(\log d)^2}}{d}$
, one can argue the same way as in \cref{sec:mainclaim} (by potentially enlarging the fragile interval by a benign $(\log d)$ factor on each side) that the signed expectation of this inhomogeneous spherical random geometric graphs is bounded by $(8p)^{|E(\NLP(H))|}\Big(\frac{C(km)(\log d)^2}{\sqrt{d}}\Big)^{\OEI(\NLP(H))}.$ The conclusion follows.
\end{proof}

\section{Omitted Details from \texorpdfstring{\cref{section:RGGtoPC}}{PCol Theorem}}
\label{appendix:rggtopccalc}
\subsection{Proof of \texorpdfstring{\cref{prop:pccolors}}{Chi squared for PCol}}
\label{appendix:signedweightsinpcol}
\begin{proof} Let $V(H)=[v]$ and let the random labels of vertices are $x_1,x_2, \ldots,x_v.$

\paragraph{1)} Suppose that, WLOG, $1$ is a leaf and is adjacent to 2. Then,
\begin{align*}
& \expect[\SW^{1/2}_H(\bfG)] = 
\expect\Big[
\expect\Big[
\prod_{(ji)\in E(H)}(G_{ji} - 1/2)
\Big|
x_2,x_3,\ldots,x_v
\Big]
\Big]\\
& = 
\expect\Big[
\prod_{(ji)\in E(H)\backslash\{(21)\}}(G_{ji} - 1/2)
\expect\Big[
(G_{21} - 1/2)
\Big|
x_2,x_3,\ldots,x_v
\Big]
\Big].
\end{align*}
The last expectation equals zero. Indeed, $$\expect\Big[
(G_{21} - 1/2)
\Big|
x_2,x_3,\ldots,x_v
\Big] = \frac{1}{2}\prob[x_1= x_2] - \frac{1}{2(q-1)}\prob[x_1\neq x_2] = \frac{1}{2q}- \frac{1}{2(q-1)}\times\frac{q-1}{q}=0.$$

\paragraph{2)} We compute $\expect[(G_{21}-1/2)(G_{31} - 1/2)(G_{32}-1/2)]$ based on the number of equalities among labels.
\begin{enumerate}
    \item All three labels are equal with probability $\frac{1}{q^2}.$ In that case, $$\expect[(G_{21}-1/2)(G_{31} - 1/2)(G_{32}-1/2)|x_1= x_2 = x_3]= \frac{1}{8}.$$
    \item Two of the labels are equal and the third is different with probability $\frac{3(q-1)}{q^2}.$
    In that case, $$\expect[(G_{21}-1/2)(G_{31} - 1/2)(G_{32}-1/2)|x_1= x_2 \neq x_3]= \frac{1}{8(q-1)^2}.$$
    \item All three labels are distinct with probability 
    $\frac{(q-1)(q-2)}{q^2}.$ In that case, $$\expect[(G_{21}-1/2)(G_{31} - 1/2)(G_{32}-1/2)|x_1\neq x_2 \neq x_3]= -\frac{1}{8(q-1)^3}.$$
\end{enumerate}
\paragraph{3)} Clearly, it is enough to prove the claim for $H$ connected. Suppose that the labels partition $V(H)$ into sets $C_1,C_2,\ldots,C_m,$ determined by label equalities, i.e. $a,b$ are in the same $C_i$ if and only if $x_a=x_b.$ Clearly, there are at most $|V(H)|^{|V(H)|-1}$ such partitions.

Furthermore, each partition $C_1,C_2, \ldots,C_m$ occurs with probability at most $\prod_{i= 1}^m(1/q)^{|C_i|- 1} = (1/q)^{|V(H)|-m}$ as in each $C_i,$ the probability that all vertices get the same label as the lexicographically first one is $(1/q)^{|C_i|- 1}.$ Finally, suppose that there are $k$ edges with two endpoints in different sets $C_i,C_j.$ Then,
$$
\expect\Big[\prod_{(ji)\in E(H)}(G_{ji}- 1/2)|C_1,C_2,\ldots,C_m\Big] = 
\Big(\frac{1}{2}\Big)^{|E(H)| - k}\times 
\Big(-\frac{1}{2(q-1)}\Big)^{k}.
$$
However, as $H$ is connected, $k \ge m-1.$ Hence, the above expression is at most $(1/q)^{m-1}$ in absolute value. Summing over the $|V(H)|^{|V(H)|-1}$ partitions gives the result.
\end{proof}

\subsection{The Low-Degree \texorpdfstring{$\chi^2$}{chi squared} calculation}
\label{appendix:chisqcalcpcol}
Finally, note that we defined $\widehat{w}_H\coloneqq  M_{H,H}w_H.$ However, $M_{H,H}\ge 1 - |V(H)|^2/q\ge 1/2$ whenever $q = \tilde{\Theta}(d^{-1/4}), |V(H)|\le (\log n)^{1.1} = \polylog(d)$ under \eqref{eq:assumption}. Thus, $|w_H|\le 2|\widehat{w}_H|.$ Therefore, going back to
\cref{prop:chisquaredtocoloring}, we can bound

\begin{equation*}
    \begin{split}
& \sum_{H\;:\;  1\le|E(H)|\le D}w_H^2\\
& \le 
4\sum_{H\;:\;  1\le|E(H)|\le D}\widehat{w}_H^2\\
& \le \sum_{H\;:\;  1\le|E(H)|\le D} (12|E(H)|)^{2|E(H)|}
    \Big(\frac{(\log n)^{9}}{q}\Big)^{3|V(H)|/2}\\
& \le \sum_{v = 1}^{2D}\sum_{H\;:\;  1\le|E(H)|\le D, |V(H)| = v}
(12|E(H)|)^{2|E(H)|}
    \Big(\frac{(\log n)^{9}}{q}\Big)^{3v/2}\\
& \le \sum_{v = 1}^{2D}
\sum_{e = 1}^{\max(D, \binom{v}{2})}
\binom{n}{v}\binom{\binom{v}{2}}{e}
(12 e)^{2e}
\Big(\frac{(\log n)^{9}}{q}\Big)^{3v/2}\\
& \le \sum_{v = 1}^{2D}
\sum_{e = 1}^{\max(D, \binom{v}{2})}
v^{2e}
\Big(\frac{n(\log n)^{20}}{q^{3/2}}\Big)^{v}.
    \end{split}
\end{equation*}
Observe that whenever $q\ge n^{2/3+c},$ equivalently $d\ge n^{8/3 + c}$ by \cref{eq:choiceofq}, the last expression is of the form 
$$
\sum_{v = 1}^{2D}
\sum_{e = 1}^{\max(D, \binom{v}{2})}
v^{2e}n^{-\theta v}
$$
for some positive constant $v.$ However, $e\log v\le v\times (e/v\log v)\le v\sqrt{e}\log v\le v \sqrt{D}\log v\le v(\log n)^{.6} = o(v\log n).$ Thus, the sum is bounded by  
$$
\sum_{v = 1}^{2D}
\sum_{e = 1}^{\max(D, \binom{v}{2})}
n^{-\theta v/2} = o_n(1),
$$
using \cref{prop:whenchisquaredissmall}.
\subsection{Proof of \texorpdfstring{\cref{prop:chisquaredtocoloring}}{chi squared for planted coloring}}
\label{sec:proofofldptoPCol}
\begin{proof}
Consider any polynomial $f$ of degree at most $D.$ By 
\cref{eq:pbasedFourierexpansion}, we can write
\begin{equation}
    \label{eq:fdefinition}
f(G)= \sum_{H\subseteq E(K_n)\;:\; 0\le |E(H)|\le D}
\widehat{f}(H)
\sqrt{\psi_q(1-\psi_q)}^{-|E(H)|}\SW^{\psi_q}_H(G).
\end{equation}
In particular, this means that 
\begin{equation}
    \label{eq:fexpectationRGG}
    \expect_{\bfG\sim\RGGspherehalf}[
    f(\bfG)
    ]= 
    \sum_{H\subseteq E(K_n)\; 1\le |E(H)|\le D}
\widehat{f}(H)\Phi^{\psi_q}_\RGGspherehalf(H) = 
\langle \widehat{f}_{\le D}, (\Phi^{\psi_q}_\RGGspherehalf)(\le D)\rangle,
\end{equation}
where $\widehat{f}_{\le D}, (\Phi^{\psi_q}_\RGGspherehalf)(\le D)$ denotes the vectors of the corresponding variables indexed by subgraphs of $K_n$ on at most $D$ edges in lexicographically increasing order. On the other hand, using convexity of $x\longrightarrow x^2,$
\begin{equation*}
    \begin{split}
        &\expect_{\bfG\sim \PCol(n,q)}[f(\bfG)^2]= 
        \expect_{\bfH\sim\mathsf{G}(n,\psi_q), \bfX\sim\Label(n,q)}
        [f(\bfH\vee\bfX)^2]\\
        & \ge 
        \expect_{\bfH\sim\mathsf{G}(n,\psi_q)}\Big[
        \expect_{\bfX\sim\Label(n,q)}[f(\bfH\vee\bfX)]^2
        \Big]\\
        & = 
        \expect_{\bfH\sim\mathsf{G}(n,\psi_q)}\Big[
        g(\bfH)^2
        \Big],
    \end{split}
\end{equation*}
where $g(\bfH)\coloneqq \expect_{\bfX\sim\Label(n,q)}[f(\bfH\vee\bfX)].$ We now analyze the function $g$ by performing a tower property on $\bfX$ restricted to edges in $H.$

\begin{equation}
    \begin{split}
    &g(\bfH) = 
    \expect_{\bfX\sim\Label(n,q)}[f(\bfH\vee\bfX)]\\
    & = 
    \sum_{H\subseteq E(K_n)\; :\;0\le |E(H)|\le D}
    \widehat{f}(H)
    \sqrt{\psi_q(1-\psi_q)}^{-|E(H)|}
    \expect_{\bfX}\Big[
    \prod_{(ji)\in H}
    ((\bfX\vee\bfH)_{ji}- \psi_q)
    \Big]\\
    & = 
    \sum_{H\subseteq E(K_n)\; :\; 0\le |E(H)|\le D}
    \widehat{f}(H)\sqrt{\psi_q(1-\psi_q)}^{-|E(H)|}\times\\
    & \quad\quad\quad\quad\times
    \Big(\sum_{K \subseteq H}
    \prob[\bfX|_{E(H)} = E(H)\backslash E(K)]
    (1 - \psi_q)^{E(H)\backslash E(K)}\prod_{(ji)\in  E(K)}(\bfH_{ji}- \psi_q)\Big)\\
    & =
    \sum_{K\subseteq E(K_n)\; :\;0\le |E(K)|\le D}
    \Big(
    \sqrt{\psi_q(1-\psi_q)}^{-|E(K)|}
    \prod_{(ji)\in  E(K)}(\bfH_{ji}- \psi_q)\times\\
    & \quad\quad\quad\quad\quad\sum_{H\subseteq K_n\; :\; 0\le |E(H)|\le D, E(K)\subseteq E(H)}
    \widehat{f}(H)
    \prob[\bfX|_{E(H)} = E(H)\backslash E(K)]
    \sqrt{(1 - \psi_q)/\psi_q}^{E(H)\backslash E(K)}\Big)\\
    & = 
    \sum_{K\subseteq E(K_n)\; :\;0\le |E(K)|\le D}
    \sqrt{\psi_q(1-\psi_q)}^{-|E(K)|}
    \prod_{(ji)\in  E(K)}(\bfH_{ji}- \psi_q)\times
    \sum_{H\subseteq K_n\; :\; 0\le |E(H)|\le D, E(K)\subseteq E(H)} 
    M_{K,H}\widehat{f}(H),
    \end{split}
\end{equation}
where $M$ is the upper-triangular matrix indexed by subgraphs of cardinality at most $D$ and defined by the above equation. 

Now, since $\bfH\sim\mathsf{G}(n,\psi_q),$ the polynomials
$$
\Big\{\sqrt{\psi_q(1-\psi_q)}^{-|E(K)|}
    \prod_{(ji)\in  E(K)}(\bfH_{ji}- \psi_q)\Big\}_{K\subseteq E(K_n)\; :\;1\le |E(K)|\le D}
$$
are orthonormal with respect to the measure of $\bfH$ (see \cref{eq:orthonormalityofstandardbasis}). Hence, 
\begin{align*}
& \expect_{\bfH\sim\mathsf{G}(n,\psi_q)}\Big[
        g(\bfH)^2
        \Big]\\
& = 
\sum_{K\subseteq E(K_n)\; :\;0\le |E(K)|\le D}
    \Big(
    \sum_{H\subseteq K_n\; :\; 0\le |E(H)|\le D, E(K)\subseteq E(H)} 
    M_{K,H}\widehat{f}(H)\Big)^2\\
& = \|M\widehat{f}_{\le D}\|^2_2.
\end{align*}

Thus, 
\begin{align*}
    \begin{split}
    &\frac{\expect_{\bfG\sim\RGGspherehalf}[
    f(\bfG)
    ]}{\expect_{\bfG\sim \PCol(n,q)}[f(\bfG)^2]^{1/2}} = 
    \frac{\langle \widehat{f}_{\le D}, (\Phi^{\psi_q}_\RGGspherehalf)(\le D)\rangle}{\|M\widehat{f}_{\le D}\|_2}.
    \end{split}
\end{align*}
A simple Cauchy-Schwartz argument as in \cite{kothari2023planted} shows that the above expression is at most $\|w\|_2,$ where $w\coloneqq (M^{-1})^T(\Phi^{\psi_q}_\RGGspherehalf)(\le D).$ Using the upper-triangular structure of $M,$ one can find
$$
w_H\coloneqq 
\frac{1}{M_{H,H}}
\Big(
\Phi^{\psi_q}_\RGGspherehalf(H)-
\sum_{K\subsetneq H}w_K
M_{K,H}
\Big),
$$
as long as $M_{H,H}\ge 0 $ holds for all $H.$
\end{proof}

\newpage
\input

\appendix

\end{document}